\documentclass[12pt]{article}%

\RequirePackage[OT1]{fontenc}

\RequirePackage[colorlinks]{hyperref}

 \ifx\pdfoutput\undefined
              \usepackage[dvips]{graphicx}
      \else
               \usepackage[pdftex]{graphicx}
       \fi
     \usepackage{epstopdf}

\usepackage{amssymb}
\usepackage{amsfonts}
\usepackage{amsmath}
\usepackage[nohead]{geometry}
\usepackage[singlespacing]{setspace}
\usepackage[bottom]{footmisc}
\usepackage{indentfirst}
\usepackage{endnotes}
\usepackage{graphicx}%
\usepackage{rotating}
\usepackage{verbatim}
\usepackage{setspace}
\usepackage{multirow}
\usepackage{latexsym}

\usepackage{amsthm}
\usepackage{makeidx}
\usepackage{fancyhdr}
\usepackage{type1cm}

\newcommand{\bb}{\mathrm{\bf b}}
\newcommand{\bff}{\mathrm{\bf f}}
\newcommand{\bx}{\mathrm{\bf x}}
\newcommand{\by}{\mathrm{\bf y}}

\newcommand{\bA}{\mathrm{\bf A}}

\newcommand{\bu}{\mathrm{\bf u}}

\newcommand{\bB}{\mathrm{\bf B}}
\newcommand{\bZ}{\mathrm{\bf Z}}
\newcommand{\bC}{\mathrm{\bf C}}
\newcommand{\bD}{\mathrm{\bf D}}

\newcommand{\bF}{\mathrm{\bf F}}
\newcommand{\bG}{\mathrm{\bf G}}
\newcommand{\bH}{\mathrm{\bf H}}
\newcommand{\bI}{\mathrm{\bf I}}
\newcommand{\bV}{\mathrm{\bf V}}

\newcommand{\bR}{\mathrm{\bf R}}

\newcommand{\bY}{\mathrm{\bf Y}}
\newcommand{\balpha}{\mbox{\boldmath $\alpha$}}
\newcommand{\bbeta}{\mbox{\boldmath $\beta$}}
\newcommand{\bone}{\mathrm{\bf 1}}

\newcommand{\bxi}{\mbox{\boldmath $\xi$}}
\newcommand{\beps}{\mbox{\boldmath $\varepsilon$}}
\newcommand{\bmu}{\mbox{\boldmath $\mu$}}
\newcommand{\bLambda}{\mbox{\boldmath $\Lambda$}}

\newcommand{\bSigma}{\mbox{\boldmath $\Sigma$}}

\newcommand{\norm}{\bigg{\|}}
\newcommand{\bPhi}{\mbox{\boldmath $\Phi$}}

\newcommand{\ty}{\tilde \by}

\newcommand{\tB}{\tilde \bB}
\newcommand{\tb}{\widetilde \bb}
\newcommand{\tF}{\tilde \bF}

\newcommand{\hb}{\widehat \bb}

\newcommand{\hF}{\widehat \bF}
\newcommand{\hf}{\widehat \bff}
\newcommand{\hR}{\widehat \bR}
\newcommand{\hu}{\widehat \bu}

\newcommand{\hcov}{\widehat \cov}

\newcommand{\heps}{\widehat\beps}
\newcommand{\hSig}{\widehat\Sig}
\newcommand{\tSig}{\widetilde\Sig}
\newcommand{\hsig}{\widehat\sigma}
\newcommand{\hlam}{\widehat\lambda}
\newcommand{\hLam}{\widehat \bLambda}

\newcommand{\hxi}{\widehat\bxi}

\newcommand{\sam}{_{\text{sam}}}
\newcommand{\cov}{\mathrm{cov}}
\newcommand{\Sig}{\mathbf{\Sigma}}

\newcommand{\tr}{\mathrm{tr}}
\newcommand{\diag}{\mathrm{diag}}
\newcommand{\argmin}{\mathrm{argmin}}

\newcommand{\bw}{\mbox{\bf w}}
\newcommand{\hw}{\widehat \bw}
\newcommand{\var}{\mathrm{var}}
\newcommand{\beq}{\begin{eqnarray*}}
\newcommand{\eeq}{\end{eqnarray*}}

\numberwithin{equation}{section}
\theoremstyle{plain}
\newtheorem{thm}{Theorem}[section]

\newtheorem{lem}{Lemma}[section]
\newtheorem{cor}{Corollary}[section]
\newtheorem{prop}{Proposition}[section]
\newtheorem{assum}{Assumption}[section]
\theoremstyle{definition}
\newtheorem{exm}{Example}[section]
\newtheorem{remark}{Remark}[section]

\makeatletter
\def\@biblabel#1{\hspace*{-\labelsep}}
\makeatother
\begin{document}

 \title{Large Covariance Estimation by Thresholding Principal Orthogonal Complements  }
\author{Jianqing Fan
\thanks{Address: Department of ORFE, Sherrerd Hall, Princeton University, Princeton, NJ 08544, USA, e-mail: \textit{jqfan@princeton.edu},
\textit{yuanliao@umd.edu}, \textit{mincheva@princeton.edu}.  The research was partially supported by NIH R01GM100474-01,
NIH R01-GM072611, DMS-0704337,   Bendheim Center for Finance at Princeton University, and Department of Mathematics at University of Maryland.  The bulk of the research was carried out while Yuan Liao was a postdoctoral fellow at Princeton University.}$\; ^\dag$, Yuan Liao$^\ddag$ and Martina Mincheva$^*$
\medskip\\{\normalsize $^*$Department of Operations Research and Financial Engineering,  Princeton University}
\medskip\\{\normalsize $^\dag$ Bendheim Center for Finance, Princeton University}
\medskip\\{\normalsize $^\ddag$ Department of Mathematics, University of Maryland}}
 
\date{}

\maketitle

\sloppy

\onehalfspacing
 
\begin{abstract}
This paper deals with the estimation of a high-dimensional covariance with a conditional sparsity structure and fast-diverging eigenvalues.  By assuming sparse error covariance matrix in an approximate factor model,  we allow for the presence of some   cross-sectional correlation even after taking out common but unobservable factors.  We introduce the Principal Orthogonal complEment Thresholding (POET) method to explore such an approximate factor structure with sparsity.  The POET estimator includes the sample covariance matrix, the factor-based covariance matrix (Fan, Fan, and Lv, 2008), the thresholding estimator (Bickel and Levina, 2008) and the adaptive thresholding estimator (Cai and Liu, 2011) as specific examples.  We provide mathematical insights when the factor analysis is approximately the same as the principal component analysis for high-dimensional data.  The rates of convergence of the sparse residual covariance matrix and the conditional sparse covariance matrix are studied under various norms.  It is shown that the impact of estimating the unknown factors vanishes as the dimensionality increases. The uniform rates of convergence for the unobserved factors and their factor loadings are derived.  The asymptotic results are also verified by extensive simulation studies. Finally, a real data application on portfolio allocation is presented.
\end{abstract}


\pagebreak%
\doublespacing

\onehalfspacing

\section{Introduction}

Information and technology make large data sets widely available for scientific discovery.  Much statistical analysis of such high-dimensional data involves the estimation of a covariance matrix or its inverse (the precision matrix).  Examples include portfolio management and risk assessment (Fan, Fan and Lv, 2008), high-dimensional classification such as Fisher discriminant (Hastie, Tibshirani and Friedman, 2009), graphic models (Meinshausen and B\"uhlmann, 2006),  statistical inference such as controlling false discoveries in multiple testing (Leek and Storey, 2008; Efron, 2010), finding quantitative trait loci based on longitudinal data (Yap, Fan, and Wu, 2009; Xiong et al. 2011),  and testing the capital asset pricing model (Sentana, 2009), among others.  See Section 5 for some of those applications.  Yet, the dimensionality is often either comparable to the sample size or even larger.  In such cases, the sample covariance is known to have  poor performance  (Johnstone, 2001), and some regularization is needed.

Realizing the importance of estimating large covariance matrices and the challenges brought by the high dimensionality, in recent years researchers have proposed various regularization techniques to consistently estimate $\Sig$. One of the key assumptions is that the covariance matrix is sparse, namely,   many entries are zero or nearly so (Bickel and Levina, 2008, Rothman et al, 2009, Lam and Fan 2009, Cai and Zhou, 2010, Cai and Liu, 2011).  In many applications, however,  the sparsity assumption directly on $\Sig$ is not appropriate.  For example, financial returns depend on the equity market risks, housing prices depend on the economic health, gene expressions can be stimulated by cytokines, among others.
Due to the presence of common factors, it is unrealistic to assume that many outcomes are uncorrelated.  An alternative method is to assume a factor model structure, as in Fan, Fan and Lv (2008). However, they restrict themselves to the strict factor models with known factors.

A natural extension is the conditional sparsity.  Given the common factors, the outcomes are weakly correlated. In order to do so, we consider an approximate factor model, which has been frequently used in economic and financial studies (Chamberlain and Rothschild, 1983; Fama and French 1993; Bai and Ng, 2002, etc):
\begin{equation} \label{eq1.1}
y_{it}=\bb_i'\bff_t+u_{it}.
\end{equation}
Here $y_{it}$ is the observed response for the $i$th ($i=1,...,p$) individual at time $t=1,...,T$; $\bb_i$ is a  vector of factor loadings; $\bff_t$  is a $K\times 1$ vector of common factors, and $u_{it}$  is the error term, usually called \textit{idiosyncratic  component}, uncorrelated with $\bff_t$.  Both $p$ and $T$ diverge to infinity, while $K$ is assumed fixed throughout the paper, and $p$ is possibly much larger than $T$.  

We emphasize that in model (\ref{eq1.1}), only $y_{it}$ is observable. It is intuitively clear that the unknown common factors can only be inferred reliably when there are sufficiently many cases, that is, $p \to \infty$. In a data-rich environment, $p$ can diverge at a rate faster than $T$. The factor model (\ref{eq1.1}) can be put in a matrix form as
\begin{equation}\label{eq1.2}
    \by_t=\bB \bff_t+{\bu_t}.
\end{equation}
where $\by_t=(y_{1t},...,y_{pt})'$, $\bB=({\bb}_1,...,{\bb}_p)'$ and $\bu_t=(u_{1t},...,u_{pt})'$.  We are interested in $\Sig$,  the $p\times p$ covariance matrix of $\by_t$, and its inverse, which are assumed to be time-invariant.   Under model (\ref{eq1.1}),  $\Sig$ is given by
\begin{equation}\label{eq1.3}
\Sig=\bB\cov(\bff_t)\bB'+\Sig_u,
\end{equation}
where $\Sig_u=(\sigma_{u,ij})_{p\times p}$ is the covariance matrix of $\bu_t$.  The literature on approximate factor models typically assumes that the first $K$ eigenvalues of $\bB\cov(\bff_t)\bB'$ diverge at rate $O(p)$, whereas all the eigenvalues of $\Sig_u$ are bounded as $p\rightarrow\infty$.
This assumption holds easily when the factors are pervasive in the sense that a non-negligible fraction of factor loadings should be non-vanishing.
The decomposition (\ref{eq1.3}) is then asymptotically identified as $p\rightarrow\infty$.  In addition to it, in this paper we  assume that $\Sig_u$ is \textit{approximately sparse} as in Bickel and Levina (2008) and Rothman et al. (2009): for some $q\in[0,1)$,
$$
m_p=\max_{i\leq p}\sum_{j\leq p}|\sigma_{u,ij}|^q
$$
does not grow too fast as $p\rightarrow\infty.$ In particular, this includes the exact sparsity assumption ($q=0$) under which $m_p=\max_{i\leq p}\sum_{j\leq p}I_{(\sigma_{u,ij}\neq0)}$, the maximum number of nonzero elements in each row.

The conditional sparsity structure  of  (\ref{eq1.2}) was explored by Fan, Liao and Mincheva (2011) in estimating the covariance matrix, when the factors $\{\bff_t\}$ are observable.  This allows them to use regression analysis to estimate $\{\bu_t\}_{t=1}^T$.  This paper deals with the situation in which the factors are unobservable and have to be inferred.  Our approach is simple, optimization-free and it uses the data only through the sample covariance matrix.  Run the singular value decomposition on the sample covariance matrix $\hSig_{\sam}$ of $\by_t$, keep the covariance matrix formed by the first $K$ principal components, and apply the thresholding procedure to the remaining covariance matrix.  This results in a Principal Orthogonal complEment Thresholding (POET) estimator. When the number of common factors $K$ is unknown, it can be estimated from the data.  See Section 2 for additional details.  We will investigate various properties of POET under the assumption that the data are serially dependent, which includes independent observations as a specific example. The rate of convergence under various norms for both estimated $\Sig$ and $\Sig_u$ and their precision (inverse) matrices  will be derived. We show that the effect of estimating the unknown factors on the rate of convergence vanishes when  $p\log p\gg T$, and in particular, the rate of convergence for $\Sig_u$ achieves the optimal rate in Cai and Zhou (2012).

This paper focuses on the high-dimensional \textit{static   factor model}  (\ref{eq1.2}), which is innately related to the principal component analysis (PCA), as clarified in Section 2.  This  feature makes it different from the classical factor model with fixed dimensionality (e.g., Lawley and Maxwell 1971).  In the last ten years, much theory on the estimation and  inference of the static factor model has been developed, for example,  Stock and Watson (1998, 2002),  Bai and Ng (2002),   Bai (2003), Doz, Giannone and Reichlin (2011),  
  among others.  Our contribution  is  on the  estimation of   covariance matrices and their inverse in large factor models.

The \textit{static} model considered in this paper is to be distinguished from the \textit{dynamic factor model} as in  Forni, Hallin, Lippi and Reichlin (2000); the latter   allows  $\by_t$ to also depend on $\bff_t$ with  lags in time.   Their approach  is based on the eigenvalues and principal components of spectral density matrices, and on the frequency domain analysis.  
Moreover, as shown in  Forni and Lippi (2001),   the dynamic factor model does not really impose a restriction on the data generating process, and the assumption of idiosyncrasy (in their terminology, a $p$-dimensional process is idiosyncratic if all the eigenvalues of its spectral density matrix remain bounded as $p\rightarrow\infty$) asymptotically identifies the decomposition of $y_{it}$ into the common component and idiosyncratic error.  The literature   includes, for example, Forni et al. (2000, 2004), Forni and Lippi (2001), Hallin and Li\v{s}ka (2007, 2011), and many other references therein.  Above all, both the static and  dynamic factor models are receiving increasing attention in  applications of many fields where information usually is scattered through a (very) large number   of interrelated time series.

There has been extensive literature in recent years that deals with sparse principal components, which has been widely used to enhance the convergence of the principal components in high-dimensional space.  d'Aspremont, Bach and El Ghaoui (2008),  Shen and Huang (2008), Witten, Tibshirani, and Hastie (2009) and Ma (2011) proposed and studied various algorithms for computations. More literature on sparse PCA is found in  Johnstone and Lu (2009),  Amini and Wainwright (2009), Zhang and El Ghaoui (2011),  Birnbaum et al. (2012), among others. In addition, there has also been a growing literature that theoretically studies the recovery from a low-rank plus sparse matrix estimation problem, see for example,    Wright et al. (2009), Lin et al. (2009), Cand\`{e}s et al. (2011), Luo (2011), Agarwal, Nagahban, Wainwright (2012), Pati et al. (2012).  It corresponds to the identifiability issue of our problem.

There is  a big difference between our model and those considered in the aforementioned literature.  In the current paper, the first $K$ eigenvalues  of $\Sig$ are spiked and grow at a rate $O(p)$, whereas the eigenvalues of the matrices studied in the existing literature on covariance estimation  are usually assumed to be either bounded or slowly growing.   Due to this distinctive feature, the common components and the idiosyncratic components can be identified, and in addition, PCA on the sample covariance matrix can consistently estimate the space spanned by the eigenvectors of $\Sig$.  The existing methods of either thresholding directly or solving a constrained optimization method can fail in the presence of very spiked principal eigenvalues.  However, there is a price to pay
here: as the first $K $ eigenvalues are ``too spiked",  one can hardly obtain a satisfactory  rate of convergence  for estimating $\Sig$ in absolute term, but it can be estimated accurately in relative term (see Section 3.3 for details).  In addition, $\Sig^{-1}$ can be estimated  accurately.

We would like to further note that the low-rank plus sparse representation of our model is on the population covariance matrix, whereas  Cand\`{e}s et al. (2011),  Wright et al. (2009), Lin et al. (2009)\footnote{We thank a referee for reminding us these related works.} considered such a representation on the data matrix.  As there is no $\Sig$ to  estimate, their goal  is limited to producing a low-rank plus sparse matrix decomposition of the data matrix, which  corresponds to the identifiability issue of our study, and does not involve estimation and inference. In contrast, our ultimate goal is to estimate the population covariance matrices as well as the precision matrices.  For this purpose, we require the idiosyncratic components and common factors to be uncorrelated and  the data generating process to be strictly stationary.  The covariances considered in this paper are constant over time, though slow-time-varying covariance matrices are applicable through localization in time (time-domain smoothing).  Our consistency result on $\Sig_u$ demonstrates that the decomposition (\ref{eq1.3}) is identifiable, and hence our results also shed the light of the ``surprising phenomenon" of Cand\`{e}s et al. (2011) that one can separate fully a sparse matrix from a low-rank matrix when only the sum of these two components is available.

The rest of the paper is organized as follows. Section 2 gives our estimation procedures and builds the relationship between the principal components analysis and the factor analysis in high-dimensional space.  Section 3 provides the asymptotic theory for various estimated quantities.  Section 4 illustrates how to choose  the thresholds using cross-validation and guarantees the positive definiteness in any finite sample. Specific applications of regularized covariance matrices are given in Section 5. Numerical results are reported in Section 6. Finally,  Section 7 presents a real data application on portfolio allocation. All proofs are given in the appendix. Throughout the paper, we use $\lambda_{\min}(\bA)$ and $\lambda_{\max}(\bA)$ to denote the minimum and maximum eigenvalues of a matrix $\bA$. We also denote by $\|\bA\|_F$,  $\|\bA\|$, $\|\bA\|_1$ and $\|\bA\|_{\max}$ the Frobenius norm,  spectral norm (also called operator norm), $L_1$-norm, and elementwise   norm of a matrix $\bA$, defined respectively by $\|\bA\|_F=\tr^{1/2}(\bA'\bA)$, $\|\bA\|=\lambda_{\max}^{1/2}(\bA'\bA)$, $\|\bA\|_1 =\max_{j} \sum_{i} |a_{ij}|$ and $\|\bA\|_{\max}=\max_{i,j}|a_{ij}|$. Note that  when $\bA$ is a vector, both $\|\bA\|_F$ and $\|\bA\|$ are equal to the Euclidean norm. Finally, for two sequences, we write $a_T\gg b_T$ if $b_T=o(a_T)$ and $a_T\asymp b_T$ if  $a_T=O(b_T)$ and $b_T=O(a_T).$

\section{Regularized Covariance Matrix via PCA}

There are three  main objectives of this paper: (i) understand the relationship between principal component analysis (PCA) and the high-dimensional factor analysis; (ii) estimate both covariance matrices $\Sig$ and the idiosyncratic $\Sig_u$ and their precision matrices in the presence of common factors, and (iii) investigate the impact of estimating the unknown factors on the covariance estimation. The propositions in Section \ref{s2.1} below show  that  the space spanned by the principal components in the population level $\Sig$ is close to the space spanned by the columns of the factor loading matrix $\bB$.

\subsection{High-dimensional PCA and factor model}\label{s2.1}
 Consider a factor model
$$
y_{it}=\bb_i'\bff_t+u_{it}, \hspace{1em}i\leq p, t\leq T,
$$
where the number of common factors, $K=\dim(\bff_t)$,  is small compared to $p$ and $T$, and thus is assumed to be fixed throughout the paper.  In the model, the only observable variable is the data $y_{it}$. One of the distinguished features of the factor model is that  the principal eigenvalues of $\Sig$ are no longer bounded, but growing fast with the dimensionality.
 
 We illustrate this in the  following example.
\begin{exm} Consider a single-factor model
$
y_{it}=b_if_t+u_{it} $ where $b_i\in \mathbb{R}. $ Suppose that the factor is pervasive in the sense that it has non-negligible impact on a non-vanishing proportion of outcomes.  	It is then reasonable to assume $\sum_{i=1}^pb_i^2>cp$ for some $c>0$. Therefore, assuming that $\lambda_{\max}(\Sig_u)=o(p)$, an application of (\ref{eq1.3}) yields,
$$
\lambda_{\max}(\Sig) \geq\var(f_t)\sum_{i=1}^pb_i^2 -\lambda_{\max}(\Sig_u)>\frac{c}{2}\var(f_t)p
$$
for all large $p$, assuming $\var(f_t)>0$.
\end{exm}

We now elucidate why PCA can be used for the factor analysis in the presence of spiked eigenvalues.
Write $\bB=(\bb_1,...,\bb_p)'$  as  the $p\times K$ loading matrix.    Note that the linear space spanned by the first $K$ principal components of $\bB\cov(\bff_t)\bB'$ is the same as that spanned by the columns of $\bB$ when $\cov(\bff_t)$ is non-degenerate.  Thus, we can assume without loss of generality that
the columns of $\bB$ are orthogonal and $\cov(\bff_t) = \bI_K$, the identity matrix.  This canonical form corresponds to the identifiability condition in decomposition (\ref{eq1.3}).  Let ${\tb}_1, \cdots,{\tb}_K$ be the columns of $\bB$, ordered such that $\{\|\tb_j\|\}_{j=1}^K$ is in a non-increasing order. Then, $\{{\tb}_j/\|{\tb}_j\|\}_{j=1}^K$ are eigenvectors of the matrix $\bB\bB'$ with eigenvalues $\{\|{\tb}_j\|^2\}_{j=1}^K$ and the rest zero. We will impose the pervasiveness  assumption that all eigenvalues of the $K\times K$ matrix $p^{-1}\bB'\bB$ are bounded away from zero, which holds if the factor loadings $\{\bb_i\}_{i=1}^p$ are independent realizations from a non-degenerate population.    Since the non-vanishing eigenvalues of the matrix $\bB \bB'$ are the same as those of $\bB'\bB$, from the pervasiveness assumption it follows that $\{\|\tb_j\|^2\}_{j=1}^K$ are all growing at rate $O(p)$.

Let $\{\lambda_j\}_{j=1}^p$ be the eigenvalues of $\bSigma$ in a descending order and $\{\bxi_j\}_{j=1}^p$ be their corresponding eigenvectors.  Then, an application of Weyl's eigenvalue theorem (see the appendix) yields that

\begin{prop}  \label{prop21} Assume that the eigenvalues of $p^{-1}\bB'\bB$ are bounded away from zero for all large $p$. For the factor model (\ref{eq1.3}) with the canonical condition
\begin{equation} \label{eq2.7}
\mbox{$\cov(\bff_t)=\bI_K$ and
$\bB' \bB$ is diagonal},
\end{equation}
we have
$$
  | \lambda_j - \|\tb_j\|^2 | \leq \|\Sig_u\|, \qquad \mbox{for $j \leq K$},
  \qquad | \lambda_j | \leq \|\Sig_u\|, \qquad \mbox{for $j > K$}.
$$
In addition, for $j\leq K$,  $\liminf_{p\rightarrow\infty}\|\tb_j\|^2/p>0$.
\end{prop}

Using Proposition~\ref{prop21} and the $\sin \theta$ theorem of Davis and Kahn (1970, see the appendix), we have the following:

\begin{prop}  \label{prop22}
Under the assumptions of Proposition \ref{prop21}, if $\{\|\tb_j\|\}_{j=1}^K$ are distinct , then
$$
     \| \bxi_j  - {\tb}_j/\|{\tb}_j\| \| = O (p^{-1} \|\Sig_u\|), \qquad \mbox{for $j \leq K$}.
$$
\end{prop}

Propositions~\ref{prop21} and \ref{prop22} state that PCA and factor analysis are approximately the same if $\|\bSigma_u\| = o(p)$. This is assured through a sparsity condition on $\Sig_u=(\sigma_{u,ij})_{p\times p}$, which is frequently measured through
\begin{equation}\label{eq2.5}
 m_p=\max_{i\leq p}\sum_{ j\leq p}|\sigma_{u,ij}|^q, \quad
 \mbox{for some $q\in[0,1]$.}
\end{equation}
The intuition is that, after taking out the common factors, many pairs of the cross-sectional units become weakly correlated.    This generalized notion of sparsity was used in Bickel and Levina (2008) and Cai and Liu (2011).  Under this generalized measure of sparsity,   we have
$$\|\Sig_u \| \leq \|\Sig_u\|_1 \leq \max_{i} \sum_{j=1}^p |\sigma_{u, ij}|^q (\sigma_{u, ii} \sigma_{u, jj})^{(1-q)/2} = O(m_p),
$$
if the noise variances $\{\sigma_{u, ii}^2\}$ are bounded.  Therefore, when $m_p=o(p)$, Proposition~\ref{prop21} implies that we have distinguished eigenvalues between the principal components $\{\lambda_j\}_{j=1}^K$ and the rest of the components $\{\lambda_j\}_{j=K+1}^p$
and Proposition~\ref{prop22} ensures that the first $K$ principal components are approximately the same as the columns of the factor loadings.

The aforementioned sparsity assumption appears reasonable in empirical applications. Boivin and Ng (2006) conducted an empirical study and showed that  imposing zero correlation between weakly correlated idiosyncratic components improves forecast\footnote{We thank a referee for   this interesting reference.}. More recently, Phan (2012) empirically  estimated the level of sparsity of the idiosyncratic   covariance   using the UK market data.

Recent developments  on random matrix theory, for example, Johnstone and Lu (2009) and Paul (2007), have shown that  when $p/T$ is not negligible, the eigenvalues and eigenvectors of $\Sig$ might not be consistently estimated from the sample covariance matrix. A distinguished feature of the covariance considered in this paper is that  there are some very spiked eigenvalues. By Propositions 2.1 and 2.2,  in the factor model, the pervasiveness condition
\begin{equation}\label{e2.4}
\lambda_{\min}(p^{-1}\bB'\bB)>c>0
\end{equation}
implies that the first $K$ eigenvalues are growing at a rate $p$. Moreover,
when $p$ is large, the principal components
$\{\bxi_j\}_{j=1}^K$ are close to the normalized vectors $\{{\tb}_j\}_{j=1}^K$ when $m_p = o(p)$.  This provides the mathematics for using the first $K$ principal components as a proxy of the space spanned by the columns of the factor loading matrix $\bB$. In addition, due to (\ref{e2.4}), the signals of the first $K$ eigenvalues are  stronger than those of the spiked covariance model considered by  Jung and Marron (2009) and Birnbaum et al. (2012). Therefore, our other conditions  for the consistency of principal components at the population level are  much weaker than those in the spiked covariance literature. On the other hand, this also shows that, under our setting the PCA is  a valid approximation to factor analysis only if $p\rightarrow\infty$.  The fact that the PCA on the sample covariance is inconsistent when $p$ is bounded  was also previously demonstrated in the literature (See e.g., Bai (2003)).

With assumption (\ref{e2.4}), the standard literature on approximate factor models has shown that the PCA on the sample covariance matrix $\hSig_{\sam}$ can consistently estimate the space spanned by the factor loadings (e.g., Stock and Watson (1998), Bai (2003)). Our contribution in Propositions 2.1 and 2.2 is that we connect the high-dimensional  factor model to the principal components, and obtain the  consistency of the spectrum in the population level $\Sig$ instead of the sample level $\hSig_{\sam}$. The spectral consistency also enhances  the results in Chamberlain and Rothschild (1983). This provides the rationale behind the consistency results in the factor model literature.

\subsection{POET}\label{s2.2}

Sparsity assumption directly on $\Sig$ is inappropriate in many applications due to  the presence of common factors. Instead, we propose a nonparametric estimator of $\Sig$ based on the principal component analysis.   Let $\hlam_1\geq\hlam_2\geq\cdots\geq\hlam_p$ be the ordered eigenvalues of the sample covariance matrix $\widehat{\bSigma}_{\sam}$ and $\{\hxi_i\}_{i=1}^p$ be their corresponding eigenvectors. Then the sample covariance has the following spectral decomposition:
\begin{eqnarray}\label{eq2.1}
\widehat{\bSigma}_{\sam}
&=&\sum_{i=1}^K\hlam_i\hxi_i \hxi_i'+\hR_K,
\end{eqnarray}
where $\hR_K=\sum_{i=K+1}^p\hlam_i\hxi_i \hxi_i'=(\hat{r}_{ij})_{p\times p}$ is the principal orthogonal complement, and $K$ is the number of diverging eigenvalues of $\Sig$. Let us first assume $K$ is known.

Now we apply thresholding on $\hR_K$.  Define
\begin{equation}\label{eq2.2}
\hR^{\mathcal{T}}_K=(\hat{r}_{ij}^{\mathcal{T}})_{p\times p},\hspace{1em}\hat{r}_{ij}^{\mathcal{T}}=\begin{cases}
\hat{r}_{ii}, & i=j;\\
s_{ij}(\hat{r}_{ij})I(|\hat{r}_{ij}| \geq \tau_{ij}), & i\neq j.
\end{cases}
\end{equation}
where $s_{ij}(\cdot)$ is a generalized shrinkage function  of   Antoniadis and Fan (2001), employed by Rothman et al. (2009) and Cai and Liu (2011), and $\tau_{ij}>0$ is an  entry-dependent threshold. In particular, the hard-thresholding rule $s_{ij}(x) =x I(|x| \geq \tau_{ij})$ (Bickel and Levina, 2008)  and the constant thresholding parameter $\tau_{ij} = \delta$ are allowed. In practice, it is more desirable to have $\tau_{ij}$ be entry-adaptive.  An example of the adaptive thresholding is
\begin{equation}\label{eq2.3}
   \tau_{ij} = \tau (\hat{r}_{ii} \hat{r}_{jj})^{1/2}, \quad \mbox{for a given $\tau > 0$}
\end{equation}
where $\hat{r}_{ii}$ is the $i^{th}$ diagonal element of $\hR_K$.  This corresponds to applying the thresholding with parameter $\tau$ to the correlation matrix of $\hR_K$.

The estimator of $\Sig$ is then defined as:
\begin{equation}\label{eq2.4}
\hSig_K=\sum_{i=1}^K\hlam_i\hxi_i\hxi_i'+\hR^{\mathcal{T}}_K.
\end{equation}
We will call this estimator the Principal Orthogonal complEment thresholding (POET) estimator.
It is obtained by thresholding the remaining components of the sample covariance matrix, after taking out the first $K$ principal components. One of the attractiveness of POET is that it is optimization-free, and hence is computationally appealing. \footnote{We have written an R package for POET, which outputs the estimated  $\Sig$, $\Sig_u$, $K$, the factors and loadings.}

With the choice of $\tau_{ij}$ in (\ref{eq2.3}) and the hard thresholding rule, our estimator encompasses many popular estimators as its specific cases.  When $\tau = 0$, the estimator is the sample covariance matrix and when $\tau = 1$, the estimator becomes that based on the strict factor model (Fan, Fan, and Lv , 2008).  When $K = 0$, our estimator is the same as the thresholding estimator of Bickel and Levina (2008) and (with a more general thresholding function) Rothman et al. (2009) or the adaptive thresholding estimator of Cai and Liu (2011) with a proper choice of $\tau_{ij}$.

In practice, the number of diverging eigenvalues (or common factors) can be estimated based on the sample covariance matrix. Determining $K$ in a data-driven way is an important topic, and is well understood in the literature.  We will describe the POET with a data-driven $K$ in Section \ref{s2.4}.

\subsection{Least squares point of view}

The POET (\ref{eq2.4}) has an equivalent representation using a constrained least squares method.  The least squares method seeks for $\hLam_K=(\hb_1^K,...,\hb_p^K)'$ and $\hF_K'=(\hf_1^K,...,\hf_T^K)$ such that
\begin{equation} \label{eq2.10}
(\hLam_K, \hF_K)=\arg\min_{\bb_i\in\mathbb{R}^K, \bff_t\in\mathbb{R}^K}\sum_{i=1}^p\sum_{t=1}^T(y_{it}-{\bb}_i'\bff_t)^2,
\end{equation}
subject to the normalization
\begin{equation}\label{eq2.11}
\frac{1}{T}\sum_{t=1}^T\bff_t\bff_t'=\bI_K, \text{ and } \frac{1}{p}\sum_{i=1}^p\bb_i\bb_i'\text{ is diagonal}.
\end{equation}
The constraints (\ref{eq2.11}) correspond to the normalization (\ref{eq2.7}).
Here we assume that the mean of each variable $\{y_{it}\}_{t=1}^T$ has been removed, that is, $Ey_{it}=Ef_{jt}=0$ for all $i\leq p, j\leq K$ and $t\leq T.$ Putting it in a matrix form, the optimization problem can be written as
\begin{eqnarray}   \label{eq2.12}
 &&  \arg\min_{\bB, \bF} \|\bY - \bB \bF' \|_F^2  \\
 && T^{-1}\bF'\bF=\bI_K, \hspace{1em}\bB'\bB \text{ is diagonal.} \nonumber
\end{eqnarray}
where $\bY=(\by_1,...,\by_T)$ and $\bF' = (\bff_1, \cdots, \bff_T)$. For each given $\bF$, the least-squares estimator of $\bB$ is $\bLambda=T^{-1}\bY\bF$, using the constraint (\ref{eq2.11}) on the factors. Substituting this into (\ref{eq2.12}), the objective function now becomes
$
   \|\bY - T^{-1} \bY \bF \bF' \|_F^2 = \mbox{tr} [(\bI_T - T^{-1} \bF \bF') \bY' \bY].
$
The minimizer is now clear:  the columns of $\hF_K/\sqrt{T}$ are the eigenvectors corresponding to the $K$ largest eigenvalues of the $T\times T$ matrix $ \bY'\bY$ and $\hLam_K=T^{-1}\bY\hF_K$ (see e.g., Stock and Watson (2002)).

We will show that under some mild regularity conditions, as $p$ and $T\rightarrow\infty$, $\hb_i^{K'}\hf_t^K$ consistently estimates the true $\bb_i'\bff_t$ uniformly over $i\leq p$ and $t\leq T$. Since $\Sig_u$ is assumed to be sparse,  we can construct an estimator of $\Sig_u$ using the adaptive thresholding method by Cai and Liu (2011) as follows.   Let
$
 \hat{u}_{it}=y_{it}-\hb_i^{K'}\hf_t^K,  \hsig_{ij}=\frac{1}{T}\sum_{t=1}^T\hat{u}_{it}\hat{u}_{jt},  $ and $ \hat{\theta}_{ij}=\frac{1}{T}\sum_{t=1}^T\left(\hat{u}_{it}\hat{u}_{jt}-\hsig_{ij}\right)^2.
$
For some pre-determined decreasing sequence $\omega_T>0$, and large enough $C>0$, define the adaptive threshold parameter as
$
\tau_{ij}=C\sqrt{\hat{\theta}_{ij}}\omega_{T}.
$
The estimated idiosyncratic covariance estimator is then  given by
\begin{equation} \label{eq2.13}
\hSig_{u, K}^{\mathcal{T}}=(\hsig_{ij}^{\mathcal{T}})_{p\times p}, \quad \hsig_{ij}^{\mathcal{T}}=\begin{cases}
\hsig_{ii}, & i=j\\
s_{ij}(\hsig_{ij}), & i\neq j,
\end{cases}
\end{equation}
where for all $z\in\mathbb{R}$ (see Antoniadis and Fan, 2001),
$$
s_{ij}(z)=0 \text{ when } |z|\leq \tau_{ij}, \quad |s_{ij}(z)-z|\leq \tau_{ij}.
$$
It is easy to verify that $s_{ij}(\cdot)$ includes many interesting thresholding functions such as the hard thresholding ($s_{ij}(z)= z I_{(|z|\geq \tau_{ij})}$), soft thresholding ($s_{ij}(z)=\text{sign}(z)(|z|-\tau_{ij})_{+}$),  SCAD, and adaptive lasso (See Rothman et al. (2009)).

Analogous to  the decomposition (\ref{eq1.3}), we  obtain the following substitution estimators
\begin{equation}\label{eq2.14}
\tSig_K=\hLam_K\hLam_K'+\hSig_{u, K}^{\mathcal{T}},
\end{equation}
and by the Sherman-Morrison-Woodbury formula, noting that $\frac{1}{T}\sum_{t=1}^T\hf_t^K\hf_t^{K'}=\bI_K,$
\begin{equation} \label{eq2.15}
(\tSig_K)^{-1}=(\hSig_{u,K}^{\mathcal{T}})^{-1}-(\hSig_{u,K}^{\mathcal{T}})^{-1}
\hLam_K[\bI_K+\hLam_K'(\hSig_{u,K}^{\mathcal{T}})^{-1}
\hLam_K]^{-1}\hLam_K'(\hSig_{u,K}^{\mathcal{T}})^{-1},
\end{equation}

In practice, the true number of factors $K$ might be unknown to us. However, for any determined $K_1\leq p$, we can always construct    either $(\hSig_{K_1}, \hR^{\mathcal{T}}_{K_1})$ as in (\ref{eq2.4}) or  $(\tSig_{K_1}, \hSig_{u,K_1}^{\mathcal{T}})$ as in (\ref{eq2.14}) to estimate $(\Sig,\Sig_u)$. The following theorem shows that for each given $K_1$,  the two estimators based on either regularized PCA or least squares substitution are equivalent.  Similar results were obtained by Bai (2003) when $K_1=K$ and no thresholding was imposed.

\begin{thm} \label{thm2.1} Suppose that  the entry-dependent threshold in (\ref{eq2.2}) is the same as the thresholding parameter used in  (\ref{eq2.13}). Then for any $K_1\leq p$,
the  estimator (\ref{eq2.4}) is equivalent to the substitution estimator (\ref{eq2.14}), that is,  
$$\hSig_{K_1}=\tSig_{K_1},\hspace{2em}
\mbox{and} \hspace{2em} \hSig_{u, K_1}^{\mathcal{T}}=\hR_{K_1}^{\mathcal{T}}.$$
\end{thm}

In this paper, we will use a data-driven $\widehat{K}$  to construct  the POET (see Section 2.4 below), which has two equivalent representations according to Theorem \ref{thm2.1}.

\subsection{POET with Unknown $K$}\label{s2.4}

 Determining the number of factors in a data-driven way has been an important research topic in the econometric literature.  Bai and Ng (2002) proposed a consistent estimator as both $p$ and $T$ diverge.  Other recent criteria are proposed by Kapetanios (2010), Onatski (2010),  Alessi et al. (2010), etc.

Our method also allows  a data-driven $\widehat{K}$ to estimate the covariance matrices. In principle, any procedure that gives a consistent estimate of $K$ can be adopted. In this paper we apply the well-known method in Bai and Ng (2002). It estimates $K$ by
\begin{equation}\label{eq2.16add}
\widehat{K}=\arg\min_{0\leq K_1\leq  M}\log\left\{\frac{1}{pT}\|\bY-T^{-1}\bY\hF_{K_1}\hF_{K_1}'\|_F^2\right\}+K_1g(T,p),
\end{equation}
where $M$ is a prescribed upper bound, $\hF_{K_1}$ is a $T\times K_1$ matrix whose columns are $\sqrt{T}$ times the eigenvectors corresponding to the $K_1$ largest eigenvalues of the $T\times T$ matrix $\bY'\bY$; $g(T,p)$ is a penalty  function of $(p, T)$ such that $g(T,p)=o(1)$ and $\min\{p, T\}g(T,p)\rightarrow\infty.$ Two examples suggested by Bai and Ng  (2002) are
$$
\mbox{IC1}: g(T,p)=\frac{p+T}{pT}\log\left(\frac{pT}{p+T}\right),
$$
$$
\mbox{IC2}: g(T,p)=\frac{p+T}{pT}\log\min\{p, T\}.
$$

Throughout the paper, we let $\widehat{K}$ be the solution to (\ref{eq2.16add}) using either IC1 or IC2. The asymptotic results are not affected regardless of the specific choice of $g(T,p)$. We define the POET estimator with unknown $K$ as
\begin{equation}\label{eq2.16}
\hSig_{\widehat{K}}=\sum_{i=1}^{\widehat{K}}\hlam_i\hxi_i\hxi_i'+\hR^{\mathcal{T}}_{\widehat{K}}.
\end{equation}
 The procedure is as stated in Section \ref{s2.2} except that $\widehat K$ is now data-driven.

\section{Asymptotic Properties}

\subsection{Assumptions}

This section presents the assumptions  on the model (\ref{eq1.2}), in which only $\{\by_t\}_{t=1}^T$ are observable.  Recall  the  identifiability condition (\ref{eq2.7}).

The first assumption     has been one of the most  essential ones in the literature of approximate factor models. Under this assumption and other regularity conditions, the number of factors,  loadings and common factors can be consistently estimated (e.g., Stock and Watson (1998, 2002), Bai and Ng (2002), Bai (2003), etc.).

\begin{assum}\label{a35}  All the eigenvalues of the  $K\times K$  matrix $p^{-1}{\bB}'{\bB}$ are  bounded away from both zero and infinity as $p\rightarrow\infty$.
\end{assum}

\begin{remark}
\begin{enumerate}
  \item
It implies from Proposition 2.1 in Section 2  that the first $K$ eigenvalues of $\Sig$ grow at rate $O(p)$. This unique feature distinguishes our work from most of other low-rank plus sparse covariances considered in the literature, e.g.,   Luo (2011), Pati et al. (2012), Agarwal et al. (2012), Birnbaum et al. (2012). \footnote{To our best knowledge, the only other papers that estimate large covariances with   diverging eigenvalues (growing at the rate of dimensionality $O(p)$) are Fan et al. (2008, 2011) and Bai and  Shi (2011). While Fan et al. (2008, 2011) assumed the factors are observable, Bai and Shi (2011) considered the strict factor model in which $\Sig_u$ is diagonal.}

\item Assumption 3.1 requires the factors to be pervasive, that is, to impact a non-vanishing proportion of individual time series. See Example 2.1 for its meaning.    \footnote{It is important to distinguish the model we consider in this paper from the  ``sparse factor model" in the literature, e.g., Carvalho et al. (2009), Pati et al. (2012), which  assumes that  the loading matrix $\bB$ is sparse The intuition of a sparse loading matrix is that each factor is related to only a relatively small number of stocks, assets, genes, etc.  With $\bB$ being sparse,   all the eigenvalues of $\bB'\bB$ and hence those of $\Sig$ are bounded.}

\item  As to be illustrated in Section 3.3 below, due to the fast diverging eigenvalues,  one can hardly achieve a good rate of convergence for estimating $\Sig$ under either the spectral norm or Frobenius norm when $p>T$. This phenomenon arises naturally from the characteristics of the high-dimensional  factor model, which is another distinguished feature compared to those convergence results in the existing literature.

\end{enumerate}

\end{remark}

\begin{assum}\label{a21} (i)
 $\{\bu_t, \bff_t\}_{t\geq1}$ is strictly stationary. In addition, $Eu_{it}=Eu_{it}f_{jt}=0$ for all $i\leq p, j\leq K$ and $t\leq T.$
\\
(ii) There exist  constants $c_1, c_2>0$   such that $\lambda_{\min}(\Sig_u)>c_1$, $\|\Sig_u\|_1<c_2,$
 and $\min_{i\leq p, j\leq p}\var(u_{it}u_{jt})>c_1.$\\
(iii)   There exist $r_1, r_2>0$ and $b_1, b_2>0$, such that for any $s>0$, $i\leq p$ and $j\leq K$,
\begin{equation*}
P(|u_{it}|>s)\leq\exp(-(s/b_1)^{r_1}), \quad P(|f_{jt}|>s)\leq \exp(-(s/b_2)^{r_2}).
\end{equation*}
\end{assum}

Condition (i)  requires strict stationarity as well as the non-correlation between $\{\bu_{t}\}$ and $\{\bff_t\}$. These conditions are slightly stronger than those in the literature, e.g., Bai (2003), but are still standard and simplify our  technicalities.  Condition (ii) requires that $\Sig_u$ be well-conditioned.  The condition $\|\Sig_u\|_1\leq c_2$ instead of a weaker condition $\lambda_{\max}(\Sig_u)\leq c_2$ is imposed here in order to consistently estimate $K$. But it  is still standard in the approximate factor model literature as in  Bai and Ng (2002), Bai (2003), etc.    When $K$ is known, such a condition can be removed. Our working  paper\footnote{See Fan, Liao and Mincheva (2011), working paper, arxiv.org/pdf/1201.0175.pdf} shows that the results continue to hold for a growing (known) $K$ under the weaker condition $\lambda_{\max}(\Sig_u)\leq c_2$. Condition (iii) requires     exponential-type tails, which allows us to apply the large deviation theory to $\frac{1}{T}\sum_{t=1}^Tu_{it}u_{jt}-\sigma_{u,ij}$ and $\frac{1}{T}\sum_{t=1}^Tf_{jt}u_{it}$.

We impose the strong mixing condition.  Let $\mathcal{F}_{-\infty}^0$ and $\mathcal{F}_{T}^{\infty}$ denote the $\sigma$-algebras generated by $\{(\bff_t,\bu_t): t\leq 0\}$ and  $\{(\bff_t,\bu_t):  t\geq T\}$ respectively. In addition, define the mixing coefficient
\begin{equation} \label{eq3.1}
\alpha(T)=\sup_{A\in\mathcal{F}_{-\infty}^0, B\in\mathcal{F}_{T}^{\infty}}|P(A)P(B)-P(AB)|.
\end{equation}
\begin{assum} \label{a32}
Strong mixing: There exists   $r_3>0$ such that $3r_1^{-1}+1.5r_2^{-1}+r_3^{-1}>1$, and $C>0$ satisfying:  for all $T\in\mathbb{Z}^+$,
$$\alpha(T)\leq \exp(-CT^{r_3}).$$
\end{assum}

In addition, we impose the following regularity conditions.
\begin{assum}\label{a33}   There exists $M>0$  such that  for all $i\leq p$,  $t\leq T$ and $s\leq T$,\\
(i) $\|\bb_{i}\|_{\max}<M$,\\
(ii)  $E[p^{-1/2}({\bu}_s'{\bu}_t-E{\bu}_s'{\bu}_t)]^4<M$,\\
(iii) $E\|p^{-1/2}\sum_{i=1}^p\bb_iu_{it}\|^4<M$.
\end{assum}

These conditions  are needed to consistently estimate the transformed common factors  as well as the factor loadings. Similar conditions were also assumed in Bai (2003), and Bai and Ng (2006). The number of factors is assumed to be fixed. Our conditions in Assumption \ref{a33} are weaker than those in Bai (2003) as we focus on different aspects of the study.

\subsection{Convergence of the idiosyncratic covariance}
Estimating  the covariance matrix $\Sig_u$ of the idiosyncratic components $\{\bu_t\}$ is important for many statistical inferences. For example, it is needed for large sample inference of the unknown factors and their loadings, for testing the capital asset pricing model (Sentana, 2009), and large-scale hypothesis testing (Fan, Han and Gu, 2012).  See Section 5.

We estimate $\Sig_u$ by thresholding the principal orthogonal complements after   the first $\widehat{K}$ principal components of the sample covariance  are taken out:
$
\hSig_{u,\widehat{K}}^{\mathcal{T}}=\hR_{\widehat{K}}^{\mathcal{T}}.
$
By Theorem \ref{thm2.1}, it also has an equivalent expression  given by (\ref{eq2.13}), with $\hat{u}_{it}=y_{it}-(\hb_i^{\widehat{K}})'\hf_t^{\widehat{K}}$.
 
Throughout the paper, we apply the  adaptive threshold
\begin{equation}\label{eq3.3}
   \tau_{ij}= C\sqrt{\hat{\theta}_{ij}} \omega_T, \quad \omega_T=\frac{1}{\sqrt{p}}+\sqrt{\frac{\log p}{T}}
\end{equation}
where $C>0$ is a sufficiently large constant, though the results hold for other types of thresholding.   As in Bickel and Levina (2008) and Cai and Liu (2011), the  threshold chosen in the current paper  is in fact obtained from the optimal uniform rate of convergence of  $\max_{i\leq p, j\leq p}|\hsig_{ij}-\sigma_{u,ij}|.$  When direct observation of $u_{it}$ is not available,  the effect of estimating the unknown factors also contributes to this uniform estimation error, which is why $p^{-1/2}$ appears in the threshold.

The following theorem gives the rate of convergence of the estimated idiosyncratic covariance.  Let $\gamma^{-1}=3r_1^{-1}+1.5r_2^{-1}+r_3^{-1}+1$.  In the convergence rate   below, recall that $m_p$ and $q$ are defined in the measure of sparsity (\ref{eq2.5}).

\begin{thm} \label{thm31}
Suppose $\log p=o(T^{\gamma/6})$, $ T=o(p^2)$, and Assumptions \ref{a35}-\ref{a33} hold. Then for a sufficiently large constant $C>0$ in the threshold (\ref{eq3.3}),  the POET estimator $\hSig_{u,\widehat{K}}^{\mathcal{T}}$  satisfies
$$
    \|\hSig_{u,\widehat{K}}^{\mathcal{T}}-\Sig_u\|=O_p\left( \omega_T^{1-q}m_p \right).
$$
If further $\omega_T^{1-q}m_p =o(1)$, then the eigenvalues of $\hSig_{u,\widehat{K}}^{\mathcal{T}}$ are all bounded away from zero with probability approaching  one, and
$$\|(\hSig_{u,\widehat{K}}^{\mathcal{T}})^{-1}-\Sig_u^{-1}\|=O_p\left(\omega_T^{1-q}m_p\right).$$
\end{thm}

When estimating $\Sig_u$, $p$ is allowed to grow exponentially fast in $T$, and $\hSig_{u,\widehat{K}}^{\mathcal{T}}$ can be made consistent under the spectral norm. In addition, $\hSig_{u,\widehat{K}}^{\mathcal{T}}$ is asymptotically invertible while the classical sample covariance matrix based on the residuals is not when $p>T.$

\begin{remark} \label{rem3.3}
\begin{enumerate}
\item Consistent estimation of $\Sig_u$ indicates that $\Sig_u$ is identifiable in (\ref{eq1.3}), namely, the sparse $\Sig_u$ can be separated perfectly from the low-rank matrix there.  The result here gives another proof (when assuming $\omega_T^{1-q}m_p=o(1)$)  of the ``surprising phenomenon" in Cand\`{e}s et al (2011) under different technical conditions.

\item
Fan, Liao and Mincheva (2011) recently showed that when $\{\bff_t\}_{t=1}^T$ are observable and $q=0$, the rate of convergence of the adaptive thresholding estimator is given by
$\|\hSig_u^{\mathcal{T}}-\Sig_u\| =O_p\left(m_p\sqrt{\frac{\log p}{T}}\right)=\|(\hSig_u^{\mathcal{T}})^{-1}-\Sig_u^{-1}\|.  $
Hence when the common factors are unobservable, the rate of convergence has an additional  term $m_p/\sqrt{p}$, coming from the impact of estimating the unknown factors. This impact vanishes when $p\log p \gg T$, in which case  the minimax rate as in Cai and Zhou (2010) is achieved. As $p$ increases, more information about the common factors is collected, which results in more accurate estimation of the common factors $\{\bff_t\}_{t=1}^T$.   

\item When $K$ is known and grows with $p$ and  $T$, with slightly weaker assumptions,  our working paper  (Fan et al. 2011) shows that under the exactly sparse case (that is, $q=0$), the result continues to hold with convergence rate $m_p(K^2 \sqrt{\frac{\log p}{T}} + \frac{K^3}{\sqrt{p}})$.

\end{enumerate}
\end{remark}

\subsection{Convergence of the POET estimator}

Since the first $K$ eigenvalues of $\Sig$ grow with $p$,  one can hardly estimate $\Sig$ with satisfactory accuracy  in the absolute term. This problem arises not from the limitation of any estimation method, but is due to the nature of the high-dimensional  factor model. We illustrate this using a simple example.

\begin{exm}   \label{exam31}
Consider an ideal case  where we know the spectrum except for the first eigenvector  of $\Sig$.
Let $\{\lambda_j, \bxi_j\}_{j=1}^p$ be the eigenvalues and vectors, and assume that  the largest eigenvalue  $\lambda_1\geq cp$ for some $c>0$.  Let $\hxi_1$ be the estimated first eigenvector and define the covariance estimator
$
\hSig=\lambda_1\hxi_1\hxi_1'+\sum_{j=2}^{p}\lambda_j\bxi_j\bxi_j'.
$
Assume that $\hxi_1$ is a good estimator in the sense that $\|\hxi_1-\bxi_1\|^2=O_p(T^{-1})$. However,
$$
\|\hSig-\Sig\|=\|\lambda_1(\hxi_1\hxi_1'-\bxi_1\bxi_1')\|=\lambda_1O_p(\|\hxi-\bxi\|)=O_p(\lambda_1T^{-1/2}),
$$
which can diverge when $T=O(p^2)$. $\square$
\end{exm}

In the presence of very spiked eigenvalues, while  the covariance $\Sig$ cannot be  consistently estimated in absolute term,  it can be   well estimated in terms of the \textit{relative error} matrix
$$
\Sig^{-1/2}\hSig \Sig^{-1/2} - \bI_p
$$
which is more relevant for many applications (see Example 5.2).  The relative error matrix can be measured by either its  spectral norm or the normalized Frobenius norm defined by
\begin{equation} \label{eq.fan}
 p^{-1/2}\|\Sig^{-1/2}\hSig \Sig^{-1/2} - \bI_p\|_F = \left(p^{-1}\tr[(\Sig^{-1/2}\hSig \Sig^{-1/2}-\bI_p)^2]\right)^{1/2}.
\end{equation}
In the last equality, there are $p$ terms being added in the trace operation and the factor $p^{-1}$ plays the role of normalization.  The loss (\ref{eq.fan}) is closely related to the entropy loss, introduced by James and Stein (1961). Also note that
$$
 p^{-1/2}\|\Sig^{-1/2}\hSig \Sig^{-1/2} - \bI_p\|_F=\|\hSig-\Sig\|_{\Sigma}
$$
where $
\|\bA\|_{\Sigma}=p^{-1/2}\|\Sig^{-1/2}\bA\Sig^{-1/2}\|_F $ is the weighted quadratic norm in Fan et al (2008).

Fan et al. (2008) showed that in a large factor model, the sample covariance is such that
$
\|\hSig_{\sam}-\Sig\|_{\Sigma}=O_p(\sqrt{p/T}),
$
which does not converge if $p>T$.  On the other hand, Theorem \ref{thm32} below shows that $\|\hSig_{\widehat{K}}-\Sig\|_{\Sigma}$ can still be convergent as long as $p=o(T^2)$.   Technically, the impact of high-dimensionality on the convergence rate of $\hSig_{\widehat{K}}-\Sig$ is via the number of rows in $\bB$. We show in the appendix that  $\bB$ appears in $\|\hSig_{\widehat{K}}-\Sig\|_{\Sigma}$ through $\bB'\Sig^{-1}\bB$ whose eigenvalues are bounded. Therefore it successfully cancels out the curse of high-dimensionality introduced by $\bB$.

Compared to estimating $\Sig$,  in a large approximate factor model,  we can estimate the precision matrix with a satisfactory rate under the spectral norm.    The intuition follows from the fact that $\Sig^{-1}$ has bounded eigenvalues.

The following theorem summarizes the rate of convergence under various norms.

\begin{thm}\label{thm32}  Under the assumptions of Theorem \ref{thm31}, the POET estimator defined in (\ref{eq2.16}) satisfies
  $$\|\hSig_{\widehat{K}}-\Sig\|_{\Sigma}=O_p\left(\frac{\sqrt{p}\log p}{T}+m_p\omega_T^{1-q} \right),\quad
\|\hSig_{\widehat{K}}-\Sig\|_{\max}=O_p(\omega_T).$$
In addition, if $m_p\omega_T^{1-q} =o(1)$, then
 $\hSig_{\widehat{K}}  $ is nonsingular with probability approaching one, with
  $$
\|\hSig_{\widehat{K}}^{-1}-\Sig^{-1}\|=O_p\left(m_p\omega_T^{1-q}  \right).
$$
\end{thm}

\begin{remark}

\begin{enumerate}
\item
When estimating $\Sig^{-1}$, $p$ is allowed to grow exponentially fast in $T$, and the estimator has the same rate of convergence as that of the estimator $\hSig_{u,\widehat{K}}^{\mathcal{T}}$ in Theorem~\ref{thm31}.  When $p$ becomes much larger than $T$, the precision matrix can be estimated  at the same rate as if the factors were observable.  
\item
As in Remark~\ref{rem3.3}, when $K>0$ is known and grows with $p$ and   $T$,  the working paper Fan et al. (2011) proves the following   results (when $q=0$) \footnote{The assumptions in the working paper Fan et al. (2011) are slightly weak than those presented here, in that it required $\lambda_{\max}(\Sig_u)$ instead of $\|\Sig_u\|_1$ be bounded.}:
\begin{eqnarray*}
&& \|\hSig^{\mathcal{T}}-\Sig\|_{\Sigma}=O_p\left(\frac{K\sqrt{p}\log p}{T}+K^2m_p\sqrt{\frac{\log p}{T}}+\frac{m_pK^3}{\sqrt{p}} \right) ,\\
&&\|\hSig^{\mathcal{T}}-\Sig\|_{\max}=O_p\left(K^3\sqrt{\frac{\log p}{T}}+ \frac{K^3}{\sqrt{p}}\right),\\
&&\|(\hSig^{\mathcal{T}})^{-1}-\Sig^{-1}\|=O_p\left( K^2 m_p\sqrt{\frac{\log p}{T}} + \frac{K^3m_p}{\sqrt{p}}\right),
\end{eqnarray*}
The results state explicitly the dependence of the rate of convergence on the number of factors.

\item The relative error $\|\Sig^{-1/2}\hSig_{\hat{K}} \Sig^{-1/2} - \bI_p\|$ in operator norm can be shown to have the same order as the maximum relative error of estimated eigenvalues.  It does not converge to zero nor diverge. It is much smaller than $\|\hSig_{\hat{K}} - \Sig\|$, which is of order $p/\sqrt{T}$ (see Example~\ref{exam31}).
\end{enumerate}
\end{remark}
\subsection{Convergence of unknown factors and factor loadings}

Many applications of the factor model require estimating  the unknown factors. In general,
factor loadings in ${\bB}$ and the common factors ${\bff_t}$ are not separably identifiable, as for any  matrix $\bH$ such that $\bH'\bH=\bI_K$, $\bB\bff_t={\bB}\bH'\bH{\bff_t}$. Hence
$({\bB}, {\bff_t})$ cannot be identified from  $({\bB}\bH', \bH{\bff_t})$.  
Note that the linear space spanned by the rows of $\bB$ is the same as that by those of ${\bB}\bH'$.  In practice, it often does not matter which one is used.

Let $\bV$ denote the $\widehat{K}\times \widehat{K}$ diagonal matrix of the first $\widehat{K}$ largest eigenvalues of the sample covariance matrix in decreasing order. Recall that $\bF'=(\bff_1,...,\bff_T)$ and define a  $\widehat{K}\times \widehat{K}$ matrix
$
\bH=\frac{1}{T}\bV^{-1}\hF'\bF {\bB}'{\bB}.
$
Then for  $t\leq T$,
$
\bH\bff_t=T^{-1}\bV^{-1}\hF'(\bB\bff_1,...,\bB\bff_T)'\bB\bff_t.
$
Note that $\bH\bff_t$ depends only on the data $\bV^{-1}\hF'$ and an identifiable part of parameters $\{\bB\bff_t\}_{t=1}^T$.
 Therefore, there is no identifiability issue in $\bH \bff_t$ regardless of the imposed  identifiability condition.

Bai (2003) obtained the rate of convergence for both $\hb_i$ and $\hf_t$ for any fixed $(i,t)$. However,   the uniform rate of convergence is more relevant for many applications (see Example 5.1).   The following theorem extends those results in Bai (2003) in a uniformity sense. In particular, with a more refined technique, we have improved the uniform convergence rate for $\hf_t$.

\begin{thm}\label{thm33}   Under the assumptions of Theorem \ref{thm31},
  $$\max_{i\leq p}\|\hb_i-\bH\bb_i\|=O_p\left(\omega_T\right), \quad\max_{t\leq T}\|\hf_t-\bH\bff_t\|=O_p\left(\frac{1}{T^{1/2}}+\frac{T^{1/4}}{\sqrt{p}}\right).$$
\end{thm}

As a consequence of Theorem~\ref{thm33}, we obtain the following: (recall that the constant $r_2$ is defined in Assumption \ref{a21}.)
\begin{cor} \label{c31}
 Under the assumptions of Theorem \ref{thm31},
$$
\max_{i\leq p, t\leq T }\|\hb_i' \hf_t -\bb_i' \bff_t \|=O_p\left((\log T)^{1/r_2}\sqrt{\frac{\log p}{T}}+\frac{T^{1/4}}{\sqrt{p}}\right).
$$
\end{cor}

The rates of convergence obtained above also explain the condition $T=o(p^2)$ in Theorems \ref{thm31} and \ref{thm32}.
It is needed in order to estimate the common factors $\{\bff_t\}_{t=1}^T$ uniformly in $t\leq T$. When we  do not observe $\{\bff_t\}_{t=1}^T$,  in addition to the factor loadings, there are $KT$ factors to estimate. Intuitively, the condition $T=o(p^2)$ requires the number of parameters introduced by the unknown factors  be ``not too many", so that we can consistently estimate them uniformly. Technically, as demonstrated by Bickel and Levina (2008), Cai and Liu (2011) and many other authors, achieving   uniform accuracy is essential for large covariance estimations.

\section{Choice of Threshold}
\subsection{Finite-sample positive definiteness}
 Recall that the threshold value $\tau_{ij}=C\sqrt{\hat{\theta}_{ij}}\omega_T$, where $C$ is determined by the users. To make POET operational in practice, one has to choose $C$ to maintain the positive definiteness of the estimated covariances  for any given finite sample. We   write $\hSig_{u,\widehat{K}}^{\mathcal{T}}(C)=\hSig_{u,\widehat{K}}^{\mathcal{T}}$, where the covariance estimator depends on $C$ via the threshold.   We choose $C$ in the range where $\lambda_{\min}(\hSig_{u,\widehat{K}}^{\mathcal{T}}) > 0$.
Define
\begin{equation}
C_{\min}=\inf\{C>0: \lambda_{\min}(\hSig_{u,\widehat{K}}^{\mathcal{T}}(M))>0,\quad \forall M>C\}
\end{equation}
When $C$ is sufficiently large, the estimator becomes diagonal, while  its minimum eigenvalue must retain strictly positive. Thus, $C_{\min}$ is well defined and for all $C>C_{\min}$, $\hSig_{u,\widehat{K}}^{\mathcal{T}}(C)$ is positive definite under finite sample. We can obtain  $C_{\min}$  by solving $
\lambda_{\min}(\hSig_{u,\widehat{K}}^{\mathcal{T}}(C))=0, C\neq0.
$
We can also approximate $C_{\min}$ by plotting $\lambda_{\min}(\hSig_{u,\widehat{K}}^{\mathcal{T}}(C))$ as a function of $C$, as illustrated in Figure \ref{mineig}.  In practice, we can choose $C$ in the range $(C_{\min}+\epsilon, M)$ for a small $\epsilon$ and large enough $M.$ Choosing the threshold in a range to guarantee the finite-sample positive definiteness has also been previously suggested by Fryzlewicz (2012).

\begin{figure}[htbp]
\begin{center}
\caption{Minimum eigenvalue of $\hSig_{u,\widehat{K}}^{\mathcal{T}}(C)$ as  a function of $C$ for three choices of thresholding rules.  The plot is based on the simulated data set in Section 6.2.}
\includegraphics[width=8cm]{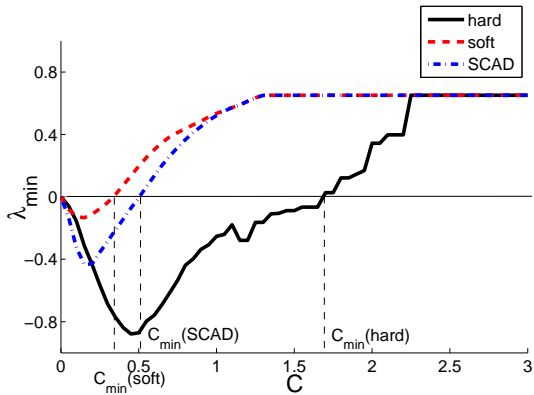}
\label{mineig}
\end{center}

\end{figure}

\subsection{Multifold Cross-Validation}
 In practice, $C$ can be data-driven, and chosen through multifold cross-validation. After obtaining the estimated residuals $\{\hu_t\}_{t\leq T}$  by the PCA, we   divide them randomly into two subsets, which are, for simplicity, denoted by $\{\hu_t\}_{t\in J_1}$ and $\{\hu_t\}_{t\in J_2}$. The sizes of $J_1$ and $J_2$, denoted by $T(J_1)$ and $T(J_2)$,  are $T(J_1)\asymp T$ and $T(J_2)+T(J_1)=T.$ For example, in sparse matrix estimation, Bickel and Levina (2008) suggested to choose $T(J_1)=T(1-(\log T)^{-1})$.

We repeat this procedure $H$ times.  At the $j$th split,  we denote by $\hSig_u^{\mathcal{T}, j}(C)$ the POET estimator with the threshold $C\sqrt{\theta_{ij}}\omega_T$ on the training data set $\{\hu_t\}_{t\in J_1}.$ We also denote by $\hSig_u^j$ the sample covariance based on the validation set, defined by
$
 \hSig_u^j=T(J_2)^{-1}\sum_{t\in J_2}\hu_t\hu_t'.$    Then we choose the constant $C^*$ by minimizing a cross-validation objective function over a compact interval
  \begin{equation} \label{eq4.2}
C^*=\arg\min_{C_{\min}+\epsilon\leq C\leq M} \frac{1}{H}\sum_{j=1}^H\|\hSig_u^{\mathcal{T}, j}(C)-\hSig_u^j\|_F^2.
\end{equation}
Here $C_{\min}$ is the minimum constant that guarantees the positive definiteness of $\hSig_{u,\widehat{K}}^{\mathcal{T}}(C)$ for $C>C_{\min}$  as described in the previous subsection, and $M$ is a large constant such that $\hSig_{u,\widehat{K}}^{\mathcal{T}}(M)$ is diagonal. The resulting $C^*$ is data-driven, so  depends  on $\bY$ as well as $p$ and $T$ via the data. On the other hand, for each given $N\times T$ data matrix $\bY$, $C^*$ is a universal constant in the threshold $\tau_{ij}=C^*\sqrt{\hat{\theta}_{ij}}\omega_T$ in the sense that it does not change with respect to the position $(i,j)$. We also note that the cross-validation is  based on the estimate of $\Sig_u$ rather than $\Sig$ because POET thresholds the error covariance matrix. Thus cross-validation improves the performance of thresholding. 

It is possible to derive the rate of convergence for $\hSig_{u,\widehat{K}}^{\mathcal{T}}(C^*)$ under the current model setting, but it ought to be  much more technically involved than the regular sparse matrix estimation considered by Bickel and Levina (2008) and Cai and Liu (2011). To keep our presentation simple we do not pursue it in the current paper.

\section{Applications of POET}

We give four examples to which  the results in Theorems \ref{thm31}--\ref{thm33} can be applied.  Detailed pursuits of these are beyond the scope of the paper.

\begin{exm}
[Large-scale hypothesis testing]
Controlling the false discovery rate in large-scale hypothesis testing based on correlated test statistics is an important and challenging problem in statistics
(Leek and Storey, 2008; Efron, 2010; Fan, et al., 2012).  Suppose that the test statistic for each of the hypothesis
$$
    H_{i0}: \mu_i = 0 \quad \mbox{vs.} \quad H_{i1}: \mu_i \not = 0
$$
is $Z_i \sim N(\mu_i, 1)$ and these test statistics $\bZ$ are jointly normal  $N(\bmu, \bSigma)$ where $\bSigma$ is unknown.  For a given critical value $x$, the false discovery proportion is then defined as  $\mbox{FDP}(x) = V(x) / R(x)$ where
$
    V(x) = p^{-1} \sum_{\mu_i = 0} I(|Z_i| > x)  $ and $R(x) = p^{-1} \sum_{i=1}^p I(|Z_i| > x)
$
are the total number of false discoveries and the total number of discoveries, respectively.  Our interest is to estimate $\mbox{FDP}(x)$ for each given $x$.  Note that $R(x)$ is an observable quantity.  Only $V(x)$ needs to be estimated.

If the covariance $\bSigma$ admits the approximate factor structure (\ref{eq1.3}),
then the test statistics can be stochastically decomposed as
\begin{equation} \label{eq4.1}
    \bZ = \bmu + \bB \bff+ \bu, \qquad \mbox{ where }\bSigma_u \mbox{ is  sparse}.
\end{equation}
By the principal factor approximation (Theorem 1, Fan, Han, Gu, 2012)
\begin{equation} \label{eq4.2}
   V(x) = \sum_{i = 1}^p \{\Phi(a_i (z_{x/2} + \eta_i)) + \Phi(a_i (z_{x/2} - \eta_i))\} + o_P(p),
\end{equation}
when $m_p = o(p)$ and the number of true significant hypothesis $\{i:  \mu_i \not = 0 \}$ is $o(p)$,  where $z_x$ is the upper $x$-quantile of the standard normal distribution, $\eta_i = (\bB \bff)_i$ and $a_i = \var(u_i)^{-1}$.

Now suppose that we have $n$ repeated measurements from the model (\ref{eq4.1}).  Then, by Corollary~\ref{c31}, $\{\eta_i\}$ can be uniformly consistently estimated, and hence $p^{-1} V(x)$ and $\mbox{FDP}(x)$ can be consistently estimated.  Efron (2010) obtained these repeated test statistics based on the bootstrap sample from the original raw data.  Our theory (Theorem~\ref{thm33}) gives a formal justification to the framework of Efron (2007, 2010).
 
\end{exm}

\begin{exm}[Risk management] The maximum elementwise estimation error $\|\hSig_{\widehat{K}}-\Sig\|_{\max}$ appears in risk assessment as in Fan, Zhang and Yu (2012). For a fixed portfolio  allocation vector $\bw$, the true portfolio variance and  the estimated one are given by $\bw'\Sig\bw$  and $\bw'\hSig_{\widehat{K}}  \bw$ respectively. The estimation error is bounded by
$$
|\bw'\hSig_{\widehat{K}}  \bw-\bw'\Sig\bw|\leq \|\hSig_{\widehat{K}}  -\Sig\|_{\max}\|\bw\|_1^2,
$$
where $\|\bw\|_1$, the $L_1$-norm of $\bw$, is the gross exposure of the portfolio. Usually a constraint is placed on the total percentage of the short positions, in which case we have a restriction $\|\bw\|_1\leq c$ for some $c>0.$  In particular, $c = 1$ corresponds to a portfolio with no-short positions (all weights are nonnegative). Theorem~\ref{thm32} quantifies the maximum approximation error.

The above compares the absolute error of perceived risk and true risk.  The relative error is bounded by
$$
|\bw'\hSig_{\widehat{K}}  \bw/\bw'\Sig\bw - 1| \leq \|\Sig^{-1/2}\hSig_{\widehat{K}} \Sig^{-1/2} - \bI_p\|
$$
for any allocation vector $\bw$.
Theorem~\ref{thm32} quantifies this relative error.
\end{exm}

\begin{exm}[Panel regression with a factor structure in the errors] Consider the following panel regression model
\begin{eqnarray*}
Y_{it}&=&\bx_{it}'\bbeta+\varepsilon_{it}, \qquad \varepsilon_{it}=\bb_i'\bff_t+u_{it}, \qquad  i\leq p, t\leq T,
\end{eqnarray*}
where $\bx_{it}$ is a vector of observable regressors with fixed dimension. The regression error $\varepsilon_{it}$ has a factor structure and is assumed to be independent of $\bx_{it}$, but $\bb_i$, $\bff_t$ and $u_{it}$ are all unobservable. We are interested in the common regression coefficients $\bbeta$. The above panel regression model has been considered by many researchers, such as Ahn, Lee and Schmidt (2001), Pesaran (2006),  and  has broad applications in social sciences.

Although OLS (ordinary least squares) produces a consistent estimator of $\bbeta$, a more efficient estimation can be obtained by GLS (generalized least squares). The GLS method depends, however, on an estimator of $\Sig_{\epsilon}^{-1}$, the inverse of the covariance matrix of $\beps_t=(\varepsilon_{1t},..., \varepsilon_{pt})'$.
By assuming the covariance matrix of $(u_{1t},...,u_{pt})$ to be sparse, we can successfully solve this problem by applying Theorem \ref{thm32}. Although $\varepsilon_{it}$ is unobservable, it can be replaced by the regression residuals $\hat{\varepsilon}_{it}$, obtained via first regressing $Y_{it}$ on $\bx_{it}$. We then apply the POET estimator to $T^{-1}\sum_{t=1}^T\heps_t\heps_t'$. By Theorem \ref{thm32}, the inverse of the resulting estimator is a consistent estimator of $\Sig_{\epsilon}^{-1}$ under the spectral norm. A slight difference  lies in the fact that when we apply POET, $T^{-1}\sum_{t=1}^T\beps_t\beps_t'$ is replaced with $T^{-1}\sum_{t=1}^T\heps_t\heps_t'$, which introduces an additional term $O_p(\sqrt{\frac{\log p}{T}})$ in the estimation error.
\end{exm}

\begin{exm}[Validating an asset pricing theory]
A celebrated financial economic theory is the capital asset pricing model (CAPM, Sharpe 1964) that makes William Sharpe  win the Nobel prize in Economics in 1990, whose extension is the multi-factor model (Ross, 1976, Chamberlain and Rothschild, 1983).  It states that in a frictionless market, the excessive return of any financial asset equals  the excessive returns of the risk factors times its factor loadings plus noises.  In the multi-period model, the excess return $y_{it}$ of firm $i$ at time $t$ follows model (\ref{eq1.1}), in which $\bff_t$ is the excess returns of the risk factors at time $t$.  To test the null hypothesis (\ref{eq1.2}), one embeds the model into the multivariate linear model
\begin{equation} \label{eq4.3}
   \by_t = \balpha + \bB \bff_t  + \bu_t, \qquad t = 1, \cdots, T
\end{equation}
and wishes to test $H_0: \balpha = 0$.  The F-test statistic involves the estimation of the covariance matrix $\Sig_u$, whose estimates are degenerate without regularization
when $p \geq T$.  Therefore, in the literature (Sentana, 2009, and references therein), one focuses on the case $p$ is relatively small.  The typical choices of parameters are $T = 60$ monthly data and the number of assets $p = 5$, 10 or 25.  However, the CAPM should hold for all tradeable assets, not just a small fraction of assets. With our regularization technique, non-degenerate estimate $\hSig_{u,\widehat{K}}^{\mathcal{T}}$ can be obtained and the F-test or likelihood-ratio test statistics can be employed even when $p \gg T$.

To provide some insights, let $\hat{\balpha}$ be the least-squares estimator of (\ref{eq4.3}).  Then, when $\bu_t \sim N(0, \Sig_u)$, $\hat{\balpha} \sim N(\balpha, \Sig_u/c_T)$ for a constant $c_T$ which depends on the observed factors.  When $\Sig_u$ is known, the Wald test statistic is $W = c_T \hat{\balpha}' \Sig_u^{-1} \hat{\balpha}$.  When it is unknown and $p$ is large, it is natural to use the F-type of test statistic $\hat W = c_T \hat{\balpha}' (\hSig_{u,\widehat{K}}^{\mathcal{T}})^{-1} \hat{\balpha}$.  The difference between these two statistics is bounded by
$$
| \hat{W} - W | \leq  c_T \| (\hSig_{u,\widehat{K}}^{\mathcal{T}})^{-1} -
\Sig_u^{-1} \|  \|\hat{\balpha}\|^2.
$$
Since under the null hypothesis $\hat{\balpha} \sim N(0, \Sig_u/c_T)$, we have $c_T \| \Sig_u^{-1/2} \hat{\balpha}\|^2 = O(p)$.  Thus, it follows from boundness of $\|\Sig_u\|$ that
$
   | \hat{W} - W | = O(p) \| (\hSig_{u,\widehat{K}}^{\mathcal{T}})^{-1} -
\Sig_u^{-1} \|.
$
Theorem 3.1 provides the rate of convergence for the above difference.
Detailed development is out of the scope of the current paper, and we will leave it as a separate research project.
\end{exm}

\section{Monte Carlo Experiments}
In this section, we will examine the performance of the POET method in a finite sample. We will also demonstrate the effect of this estimator on the asset allocation and risk assessment. Similarly to Fan, et al. (2008, 2011), we simulated from a standard Fama-French three-factor model, assuming a sparse error covariance matrix and three  factors. Throughout this section, the time span is fixed at $T=300$, and the dimensionality $p$ increases from $1$ to $600$. We assume that the excess returns of each of $p$ stocks over the risk-free interest rate follow the following model:
$$
y _{it}=b_{i1}f_{1t}+b_{i2}f_{2t}+b_{i3}f_{3t}+u_{it}.
$$
The factor loadings are drawn from a trivariate normal distribution $\bb\sim N_3(\bmu_B,\Sig_B)$,
the idiosyncratic errors from $\bu_t\sim N_p({\bf 0},\Sig_u)$, and the factor returns $\bff_t$ follow a VAR(1) model.  To make the simulation more realistic, model parameters are calibrated from the financial returns, as detailed in the following section.    

\subsection{Calibration}
To calibrate the model, we use the data on annualized returns of 100 industrial portfolios from the website of Kenneth French, and the data on 3-month Treasury bill rates from the CRSP database. These industrial portfolios are formed as the intersection of $10$ portfolios based on size (market equity) and $10$ portfolios based on  book equity to market equity ratio. Their excess returns $(\ty_t)$ are computed for the period from January $1^{st}$, 2009 to December $31^{st}$, 2010. Here, we present a short outline of the calibration procedure.
\begin{enumerate}
\item Given $\{\ty_t\}_{t=1}^{500}$ as the input data, we fit a Fama-French-three-factor model and calculate a $100\times 3$ matrix $\tB$, and $500 \times 3$ matrix $\tF$, using the principal components method described in Section 3.1.
\item
We summarize 100 factor loadings (the rows of $\tB$) by their sample mean vector $\bmu_B$ and sample covariance matrix $\Sig_B$, which are reported in Table 1. The factor loadings $\bb_i=(b_{i1}, b_{i2},b_{i3})^T$   for $i=1,...,p$ are drawn from $N_3(\bmu_B, \Sig_B)$.

\begin{table}[htdp]
\begin{center}
\caption{Mean and covariance matrix used to generate $\bb$}
\begin{tabular}{c|ccc}
\hline
 $\bmu_B$        &  & $\Sig_B$ &   \\
\hline
0.0047 &0.0767 & -0.00004& 0.0087  \\
0.0007 &-0.00004 & 0.0841&0.0013   \\
-1.8078&0.0087 & 0.0013 & 0.1649  \\
\hline
\end{tabular}
\end{center}
\end{table}

\item  We run the stationary vector autoregressive model $\bff_t=\bmu+\bPhi{\bff_{t-1}}+\beps_t$, a VAR(1) model, to the data
    $\tF$ to obtain the multivariate least squares estimator for $\bmu$ and $\bPhi$, and estimate $\Sig_{\epsilon}$. Note that all eigenvalues of $\bPhi$ in Table 2 fall within the unit circle, so our model is  stationary. The covariance matrix $\cov(\bff_t)$ can be obtained by solving the linear equation
$\cov(\bff_t)=\bPhi\cov(\bff_t)\bPhi'+\Sig_{\epsilon}.$
The estimated parameters are depicted in Table 2 and are used to generate $\bff_t$.
   \begin{table}[htdp]
   \label{tadd2}
   \begin{center}
       \caption{Parameters of $\bff_t$ generating process}
\begin{tabular}{c|ccc|ccc}
\hline
$\bmu$   & &$\cov(\bff_t)$   &  & &$\bPhi$& \\
\hline
-0.0050  &1.0037 &0.0011&-0.0009&-0.0712 &0.0468& 0.1413\\
0.0335 &0.0011  & 0.9999 &0.0042 &-0.0764  & -0.0008 &0.0646 \\
-0.0756 &-0.0009 &0.0042 &0.9973 &0.0195 &-0.0071 & -0.0544 \\
\hline
\end{tabular}
\end{center}

    \end{table}

\item
For each value of $p$, we generate a sparse covariance matrix  $\Sig_u$ of the form:
$$\Sig_u={\bD}\Sig_0{\bD}.$$
Here, $\Sig_0$ is the error correlation matrix, and $\bD$ is the diagonal matrix of the standard deviations of the errors. We set $\bD=\diag(\sigma_1,...,\sigma_p)$, where each $\sigma_i$ is generated independently from a Gamma distribution $G(\alpha, \beta)$, and $\alpha$ and $\beta$ are chosen to match the sample mean and sample standard deviation of the standard deviations of the errors. A similar approach to Fan et al. (2011) has been used in this calibration step. The off-diagonal entries of $\Sig_0$ are generated independently from a normal distribution, with mean and standard deviation equal to the sample mean and sample standard deviation of the sample correlations among the estimated residuals, conditional on their absolute values being no larger than $0.95$.  We then employ hard thresholding to make $\Sig_0$ sparse, where the threshold is found as the smallest constant that provides the positive definiteness of $\Sig_0$.  More precisely, start with threshold value 1, which gives $\Sig_0 = \bI_p$ and then decrease the threshold values in a grid until positive definiteness is violated.
\end{enumerate}

\subsection{Simulation}
For the simulation, we fix $T=300$, and let $p$ increase from $1$ to $600$. For each fixed $p$, we repeat the following steps $N=200$ times, and record the means and the standard deviations of each respective norm.  \begin{enumerate}
\item
Generate independently $\{\bb_i\}_{i=1}^{p}\sim N_3(\bmu_B,\Sig_B)$, and set $\bB=(\bb_1,...,\bb_p)'.$
\item
Generate independently $\{\bu_t\}_{t=1}^{T}\sim N_p(\bf0, \Sig_u)$.
\item
Generate $\{\bff_t\}_{t=1}^{T}$ as a vector autoregressive sequence of the form $\bff_t=\bmu+\Phi\bff_{t-1}+\beps_t$.
\item
Calculate $\{\by_t\}_{t=1}^T$ from $\by_t=\bB\bff_t+\bu_t$.
\item
 Set hard-thresholding with threshold $ 0.5\sqrt{\hat{\theta}_{ij}}(\sqrt{\frac{\log p}{T}}+\frac{1}{\sqrt{p}})$. Estimate $K$ using Bai and Ng (2002)'s IC1. Calculate covariance estimators using the POET method. Calculate the  sample covariance matrix  $\hSig_{sam}$.

\end{enumerate}

In the graphs below, we plot the averages and standard deviations of the distance from $\hSig_{\widehat{K}}  $ and $\hSig_{\sam}$ to the true covariance matrix $\Sig$, under norms $\|.\|_{\Sigma}$, $\|.\|$ and $\|.\|_{\max}$. We also plot the means and standard deviations of the distances from $(\hSig_{\widehat{K}}  )^{-1}$ and $\hSig_{\sam}^{-1}$ to $\Sig^{-1}$ under the spectral norm.  The dimensionality $p$ ranges from $20$ to $600$ in increments of $20$. Due to invertibility, the spectral norm for $\hSig_{\sam}^{-1}$ is plotted only up to $p=280$. Also, we zoom into these graphs by plotting the values of $p$ from $1$ to $100$, this time in increments of $1$. Notice that we also plot the distance from $\hSig_{obs}$ to $\Sig$ for comparison, where $\hSig_{obs}$ is the estimated covariance matrix proposed by Fan et al. (2011), assuming the  factors are observable.

\subsection{Results}

In a factor model, we expect POET to perform as well as $\hSig_{obs}$ when $p$ is relatively large, since the effect of estimating the unknown factors should vanish as $p$ increases. This  is illustrated in the plots below.

\begin{figure}[htbp]
\begin{center}
\caption{Averages (left panel) and standard deviations (right panel) of the relative error $p^{-1/2}\|\Sig^{-1/2}\hSig\Sig^{-1/2}-\bI_p\|_F$ with  known factors ($\hSig=\hSig_{obs}$ solid red curve),  POET ($\hSig=\hSig_{\widehat{K}}$ solid blue curve),  and sample covariance ($\hSig=\hSig_{\sam}$ dashed curve) over 200 simulations, as a function of the dimensionality $p$.
Top panel:  $p$ ranges in 20 to 600 with increment 20; bottom panel: $p$ ranges in 1 to 100 with increment 1.}
\includegraphics[width=6cm]{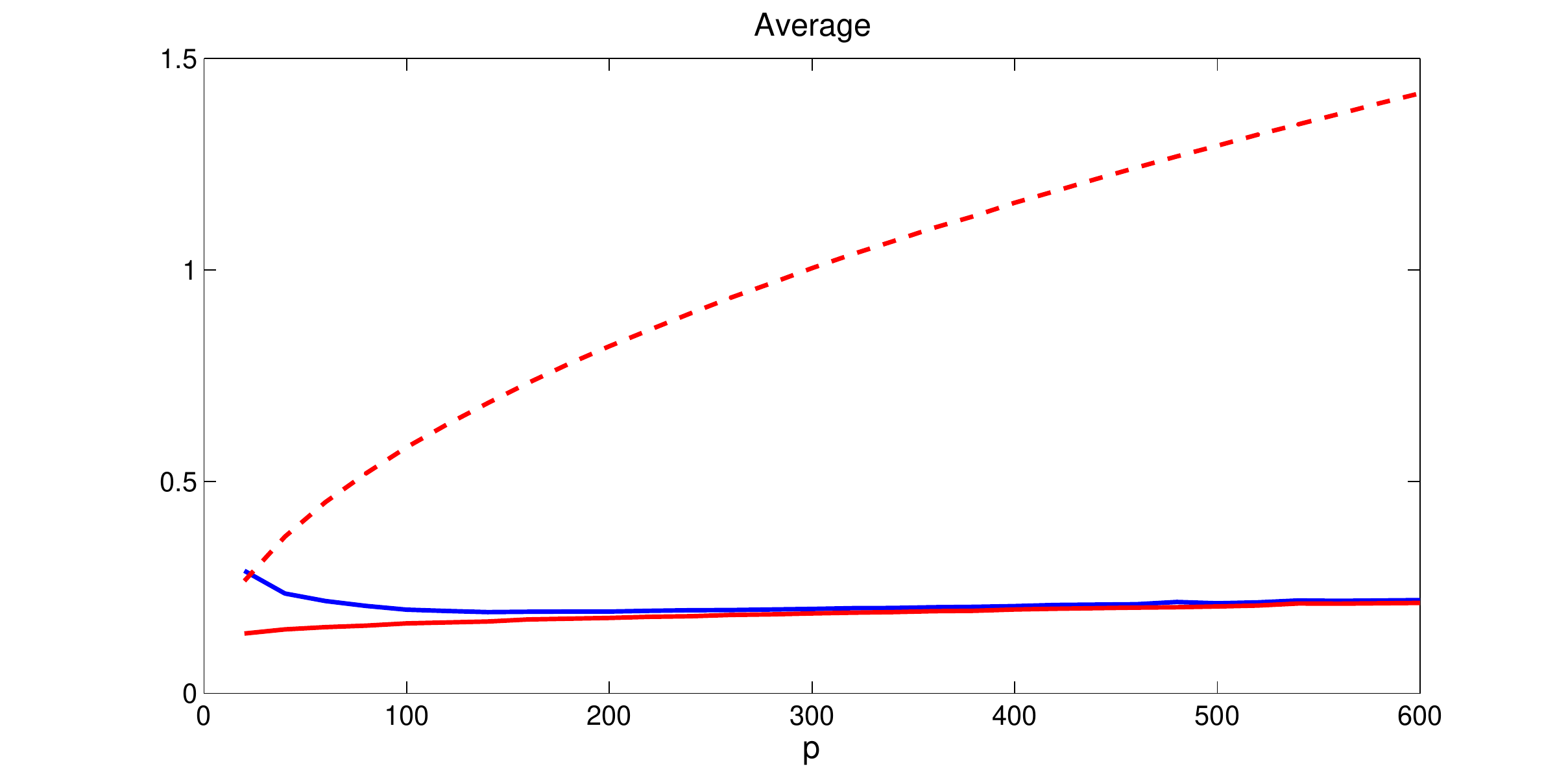}
\includegraphics[width=6cm]{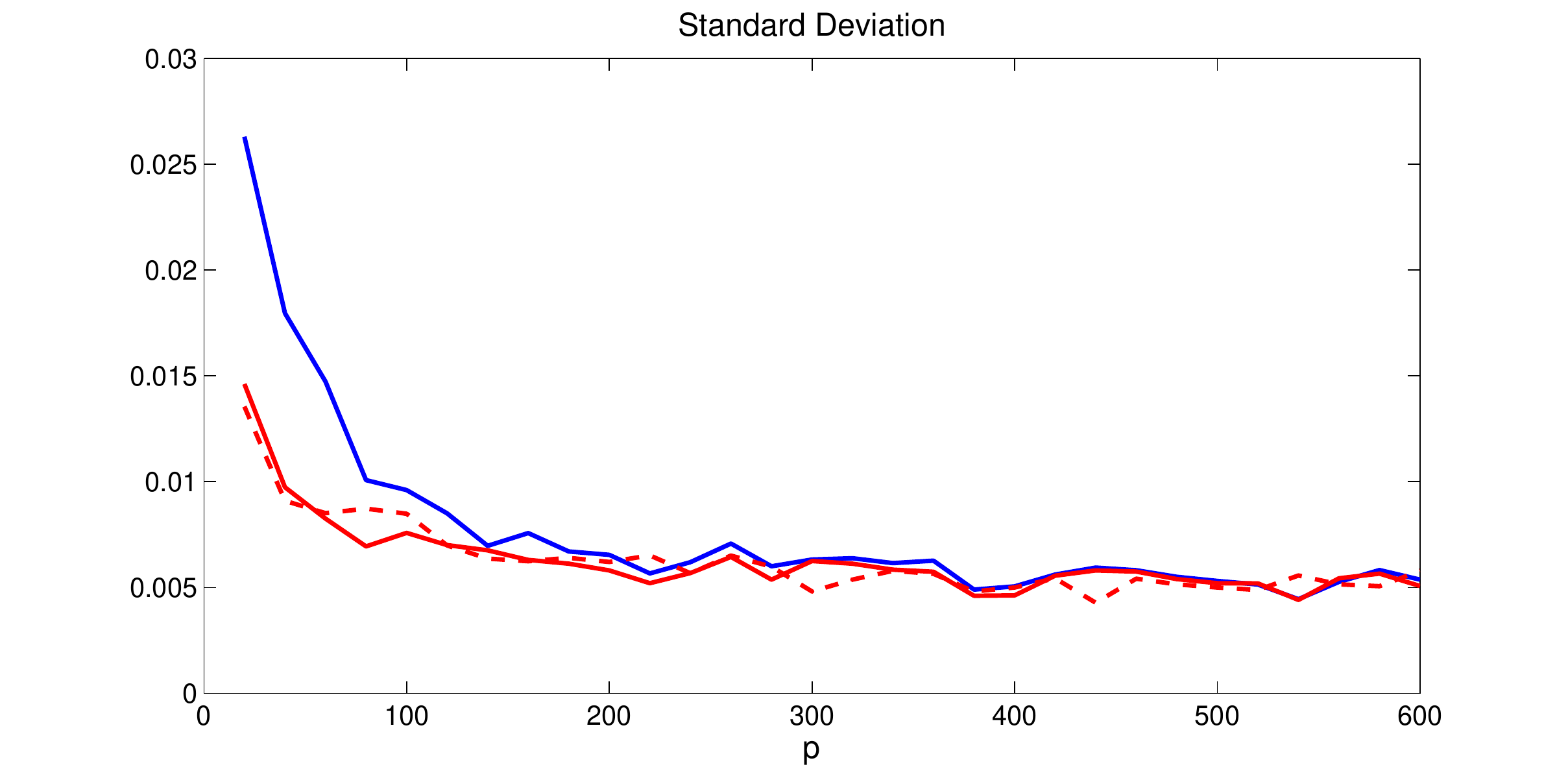}
\includegraphics[width=6cm]{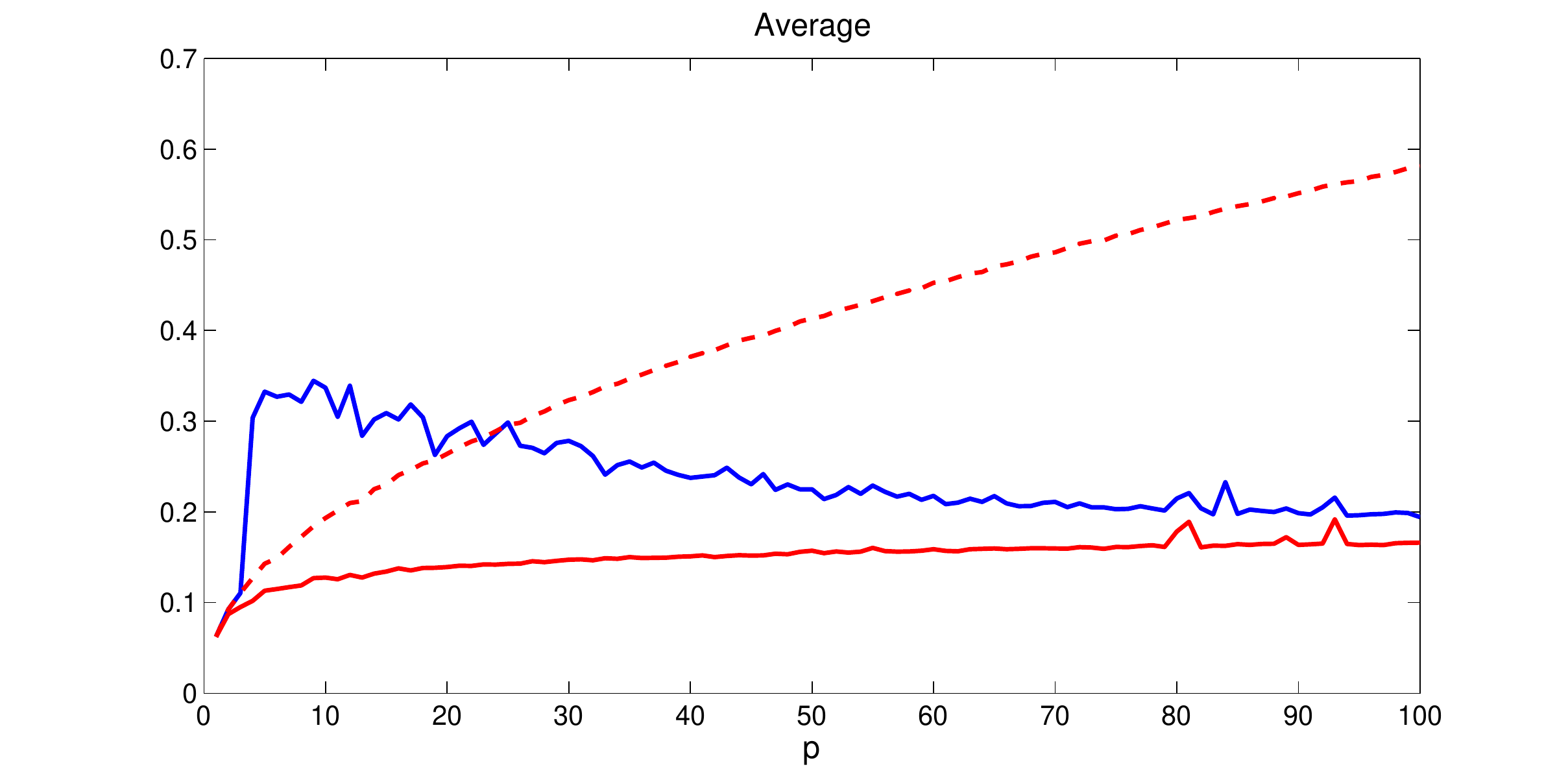}
\includegraphics[width=6cm]{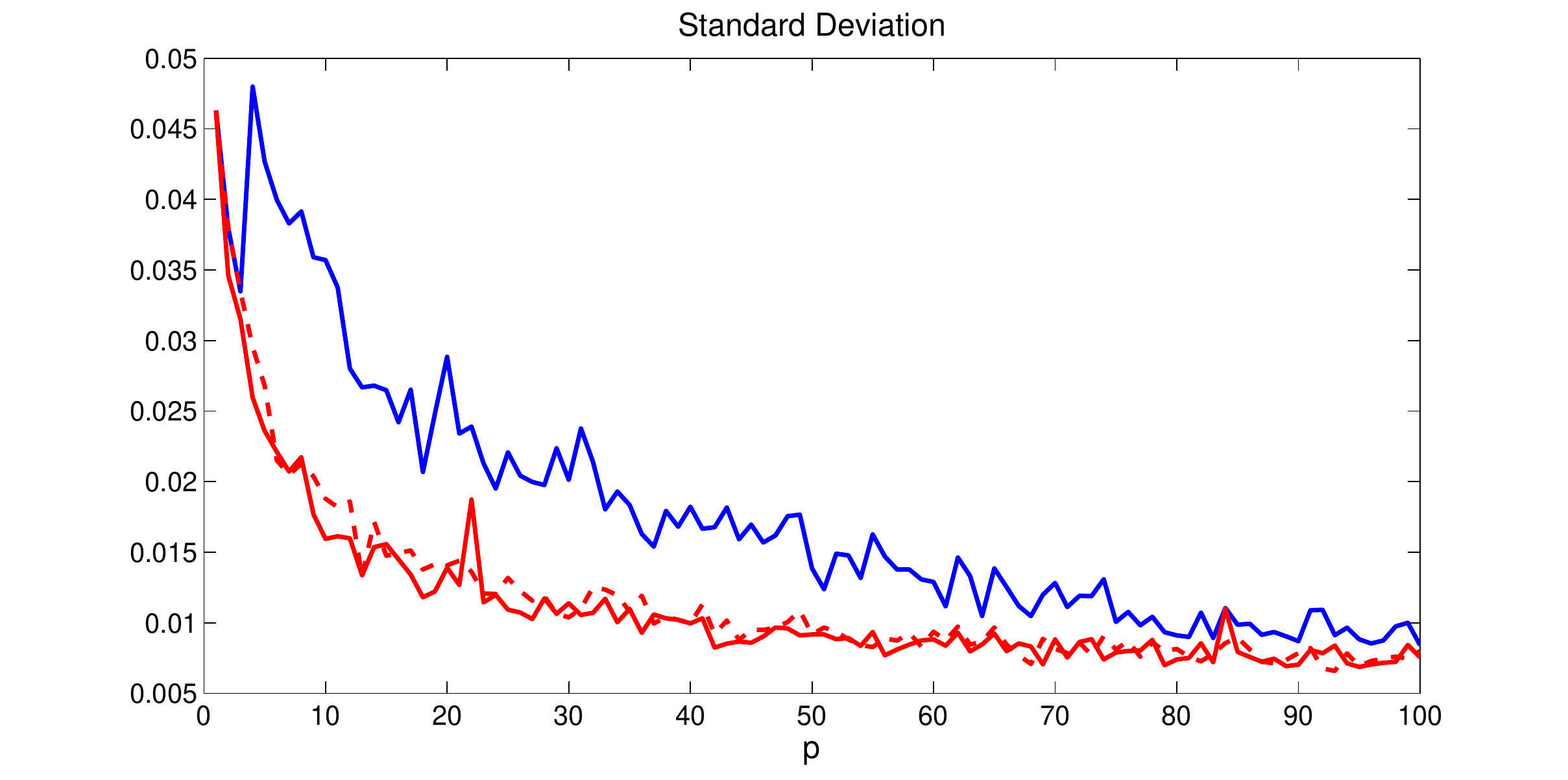}
\label{fig1}
\end{center}
\end{figure}

\begin{figure}[htbp]
\begin{center}
\caption{Averages (left panel) and standard deviations (right panel) of $\|\hSig^{-1}-\Sig^{-1}\|$ with known factors ($\hSig=\hSig_{obs}$ solid red curve),  POET ($\hSig=\hSig_{\widehat{K}}$ solid blue curve),  and sample covariance ($\hSig=\hSig_{\sam}$ dashed curve)  over 200 simulations, as a function of the dimensionality $p$.  Top panel:  $p$ ranges in 20 to 600 with increment 20; middle panel: $p$ ranges in 1 to 100 with increment 1; Bottom panel: the same as the top panel with dashed curve excluded.}
\includegraphics[width=6cm]{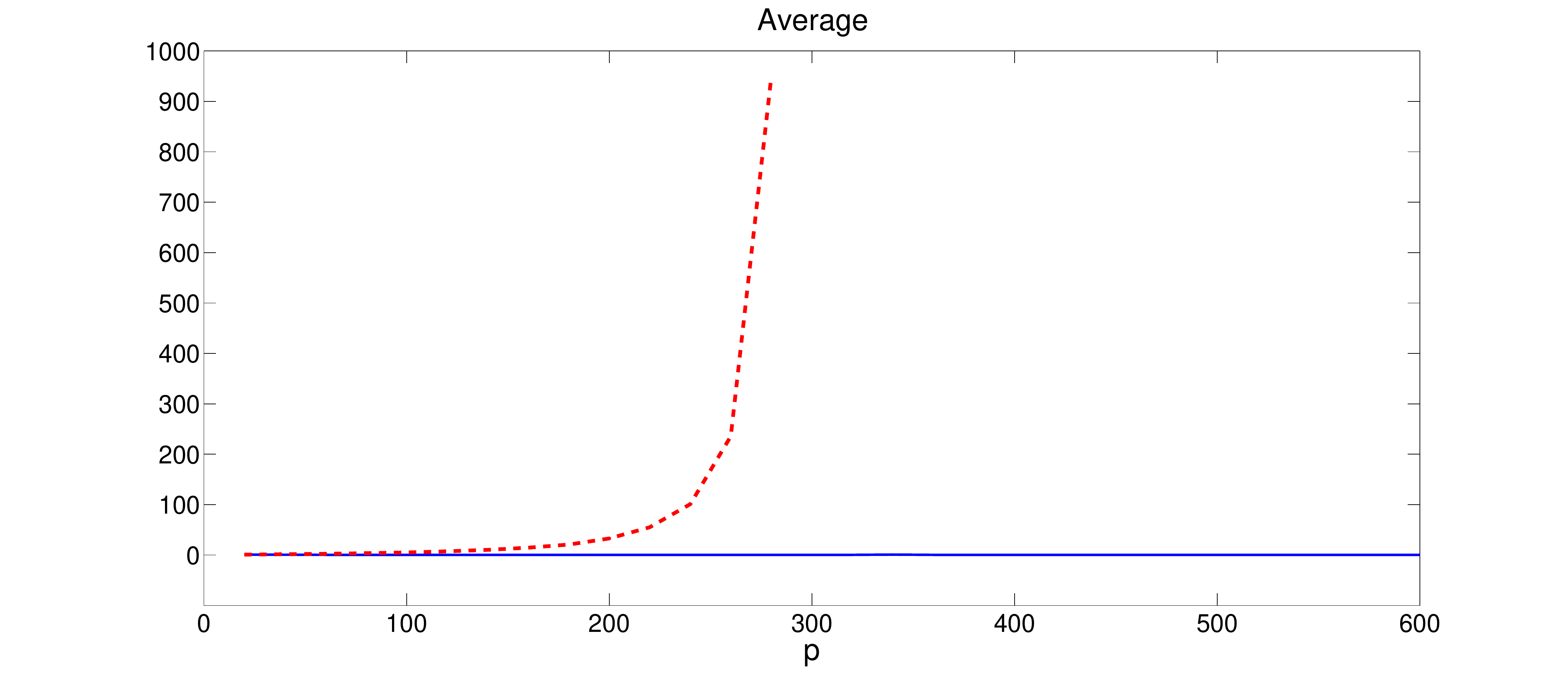}
\includegraphics[width=6cm]{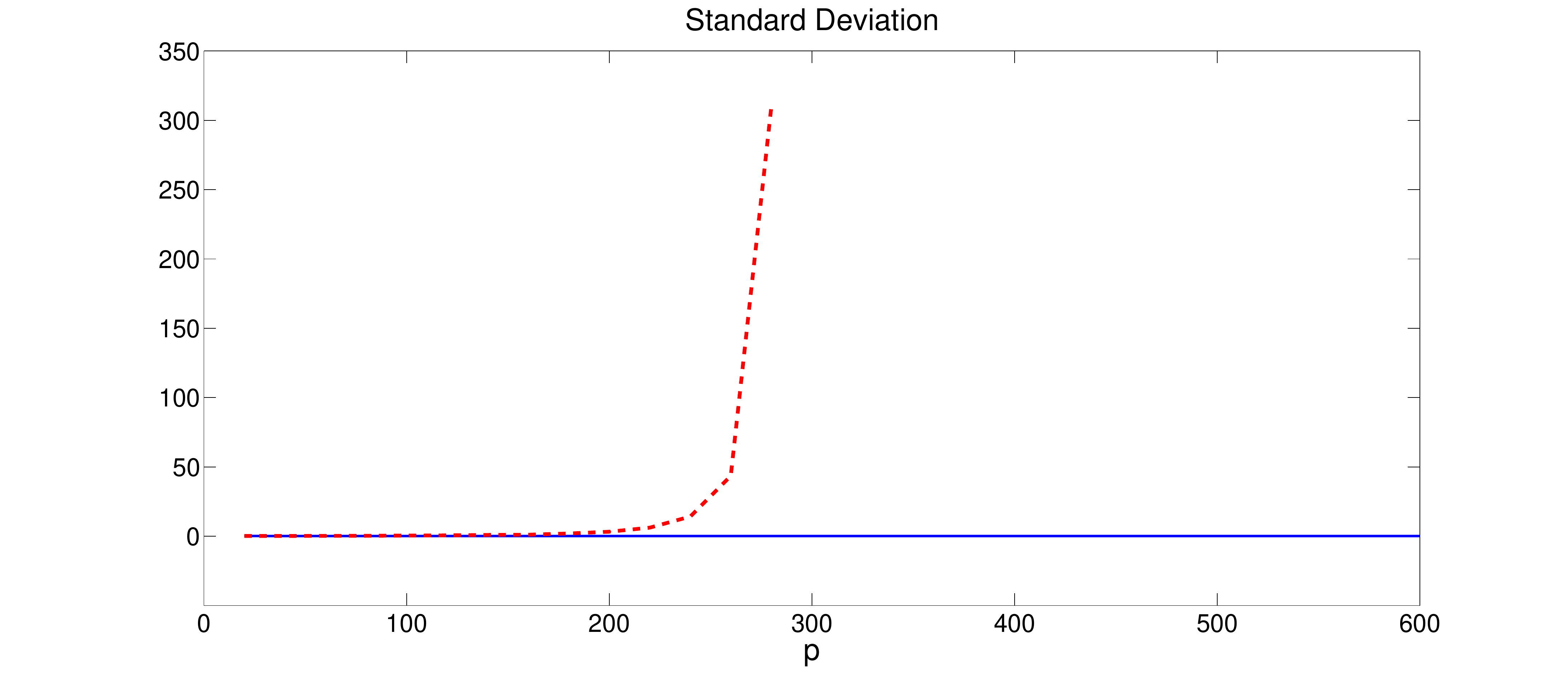}
\includegraphics[width=6cm]{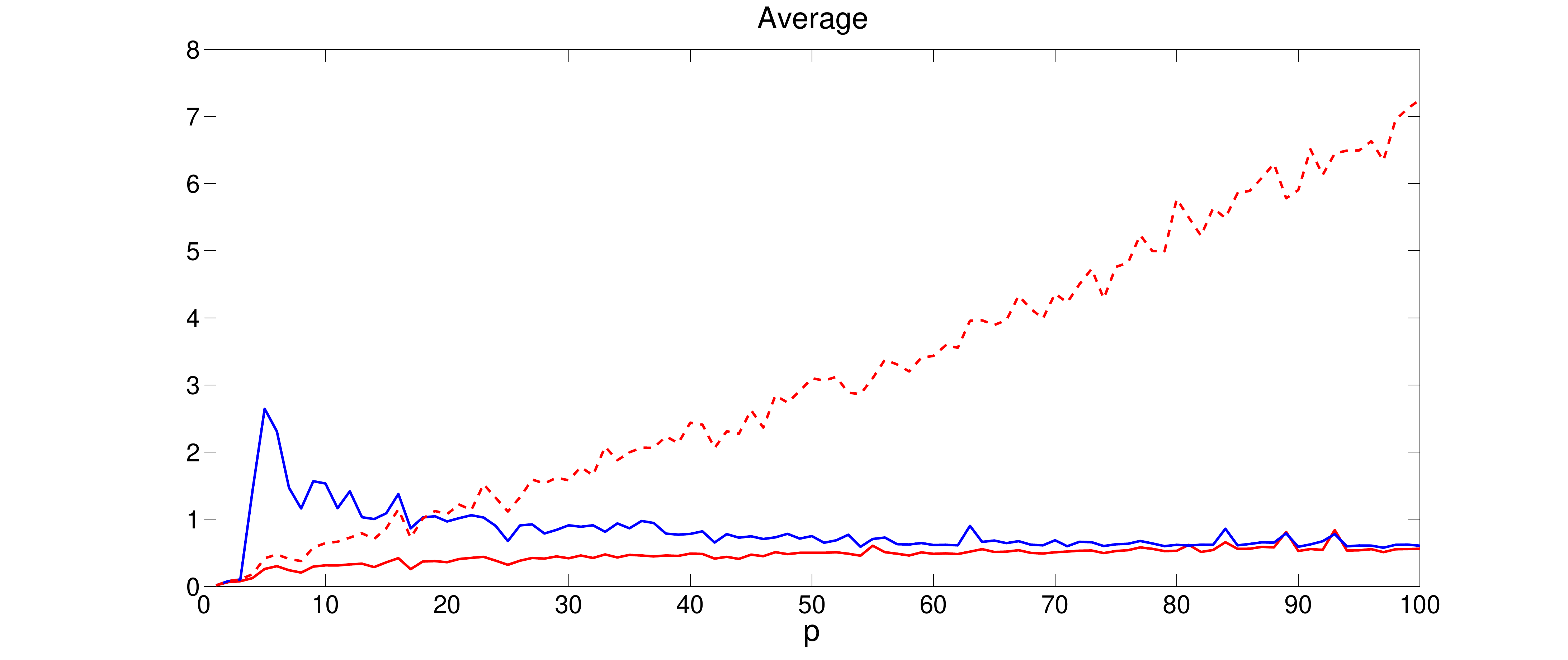}
\includegraphics[width=6cm]{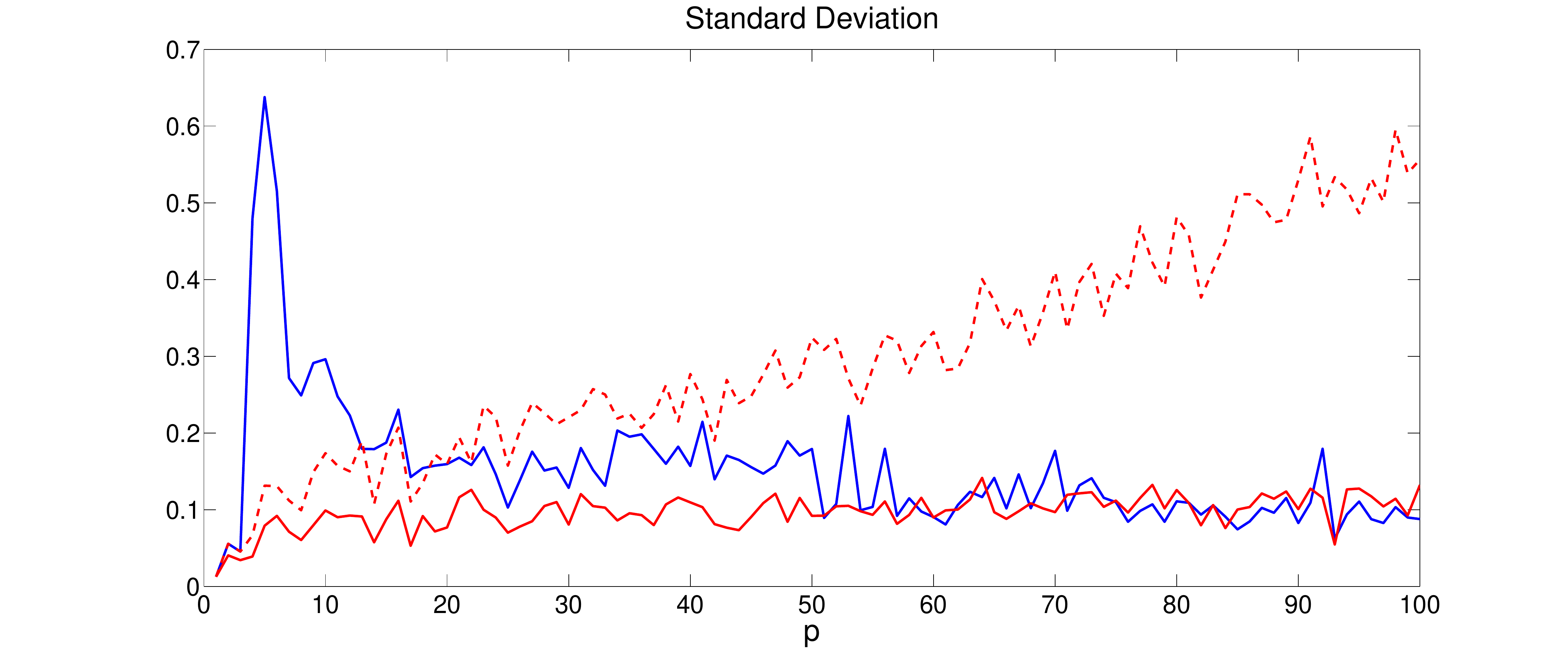}
\includegraphics[width=6cm]{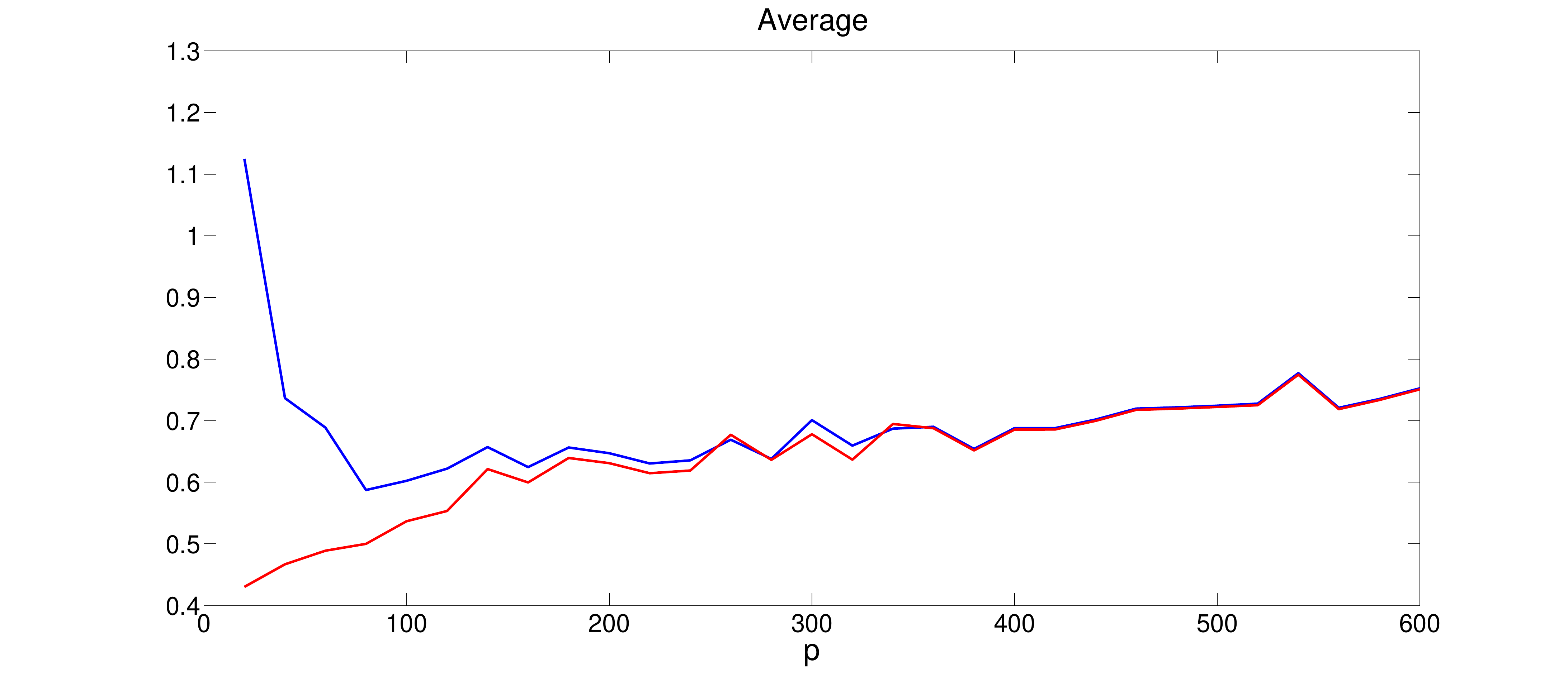}
\includegraphics[width=6cm]{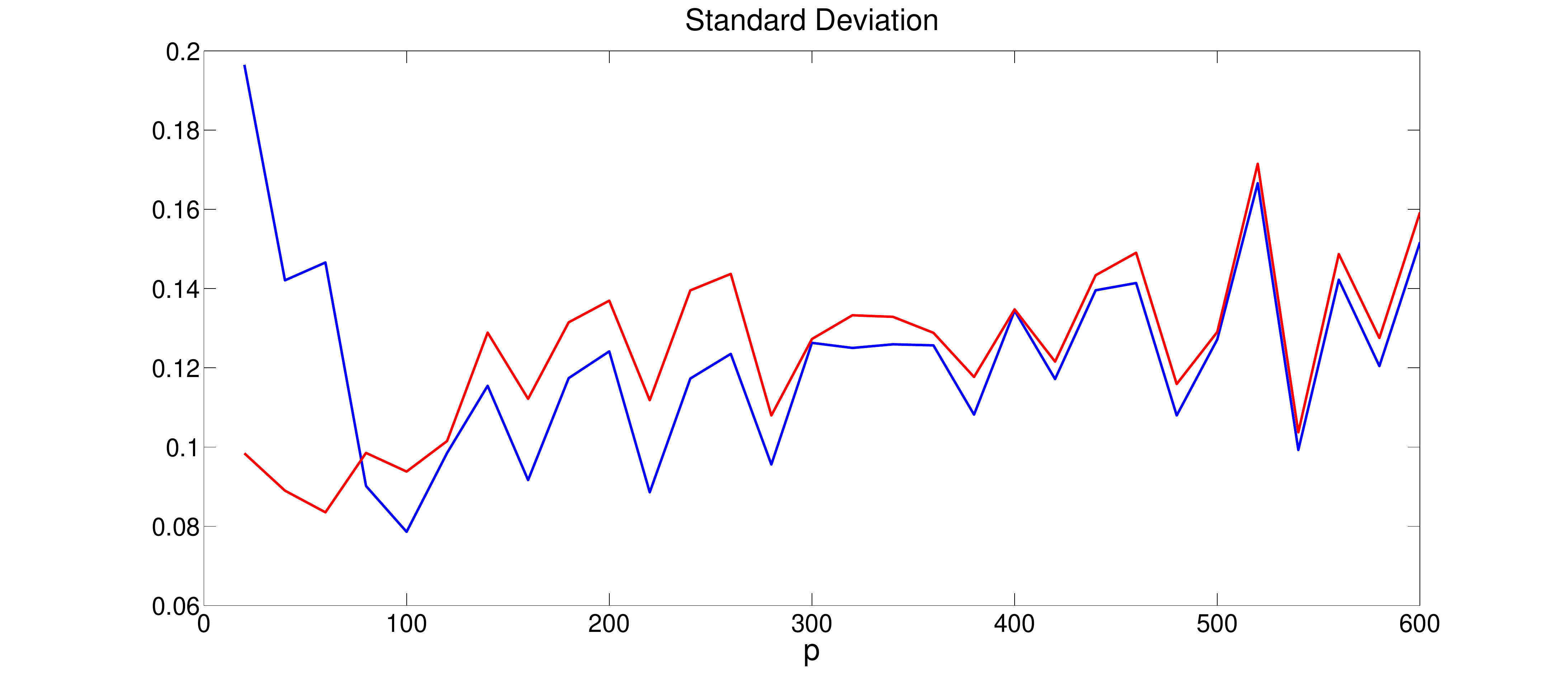}
\label{fig2}
\end{center}
\end{figure}

\begin{figure}[htbp]
\begin{center}
\caption{Averages (left panel) and standard deviations (right panel) of $\|\hSig-\Sig\|_{\max}$ with known factors ($\hSig=\hSig_{obs}$ solid red curve),  POET ($\hSig=\hSig_{\widehat{K}}$ solid blue curve),  and sample covariance ($\hSig=\hSig_{\sam}$ dashed curve)  over 200 simulations, as a function of the dimensionality $p$.   They are nearly     indifferentiable. }
\includegraphics[width=6cm]{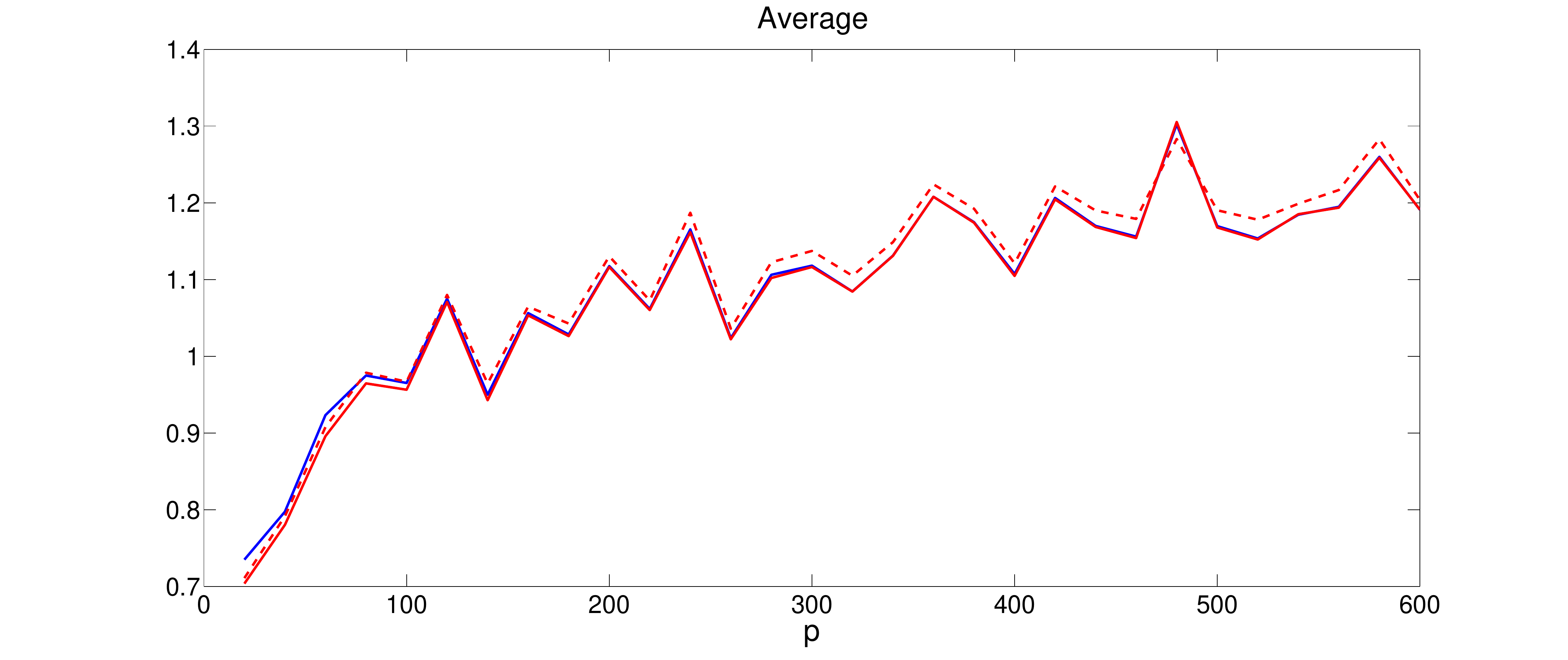}
\includegraphics[width=6cm]{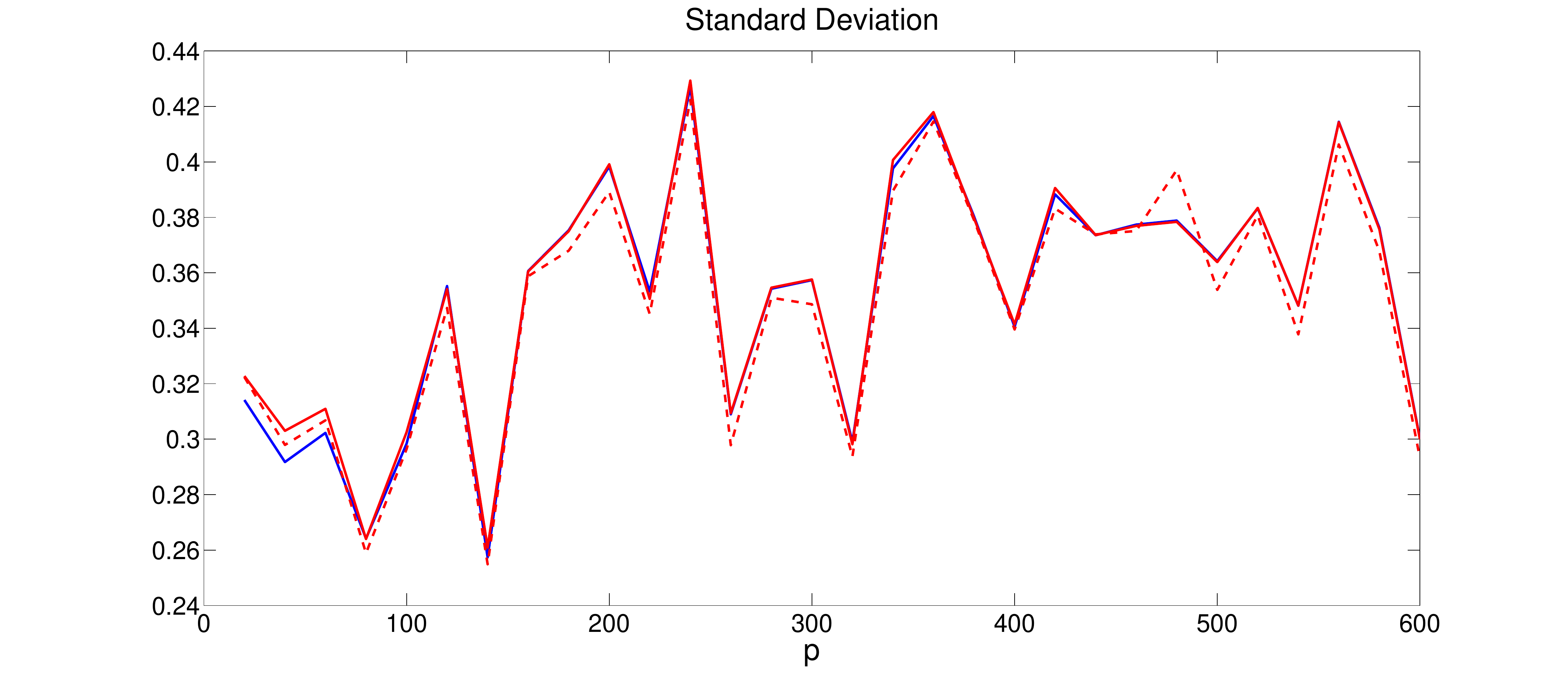}
\label{fig3}
\end{center}
\end{figure}

From the simulation results, reported in Figures~\ref{fig1}-\ref{fig4}, we observe that POET under the unobservable factor model  performs just as well as the estimator in Fan et al. (2011)   if the factors are known, when $p$ is large enough. The cost of not knowing the factors is approximately of  order $O_p(1/\sqrt p )$. It can be seen in Figures \ref{fig1} and \ref{fig2} that this cost vanishes for $p \ge 200$. To give a better insight of the impact of estimating the unknown factors for small $p$, a separate set of simulations is conducted for $p\leq 100$. As we can see from Figures \ref{fig1} (bottom panel) and \ref{fig2} (middle and bottom panels), the  impact decreases quickly. In addition, when estimating $\Sig^{-1}$,  it is hard to distinguish the estimators with known and unknown factors, whose performances are quite stable compared to the sample covariance matrix.
Also, the maximum absolute elementwise error (Figure~\ref{fig3}) of our estimator performs very similarly to that of the sample covariance matrix, which coincides with our asymptotic result.  Figure~\ref{fig4} shows that the performances of the three methods   are indistinguishable in the spectral norm, as expected.

\begin{figure}[htbp]
\begin{center}
\caption{Averages of $\|\hSig-\Sig\|$ (left panel) and $ \|\Sig^{-1/2}\hSig \Sig^{-1/2} - \bI_p\| $ with known factors ($\hSig=\hSig_{obs}$ solid red curve),  POET ($\hSig=\hSig_{\widehat{K}}$ solid blue curve),  and sample covariance ($\hSig=\hSig_{\sam}$ dashed curve)  over 200 simulations, as a function of the dimensionality $p$.  The three curves  are hardly distinguishable on the left panel.}
\includegraphics[width=7cm]{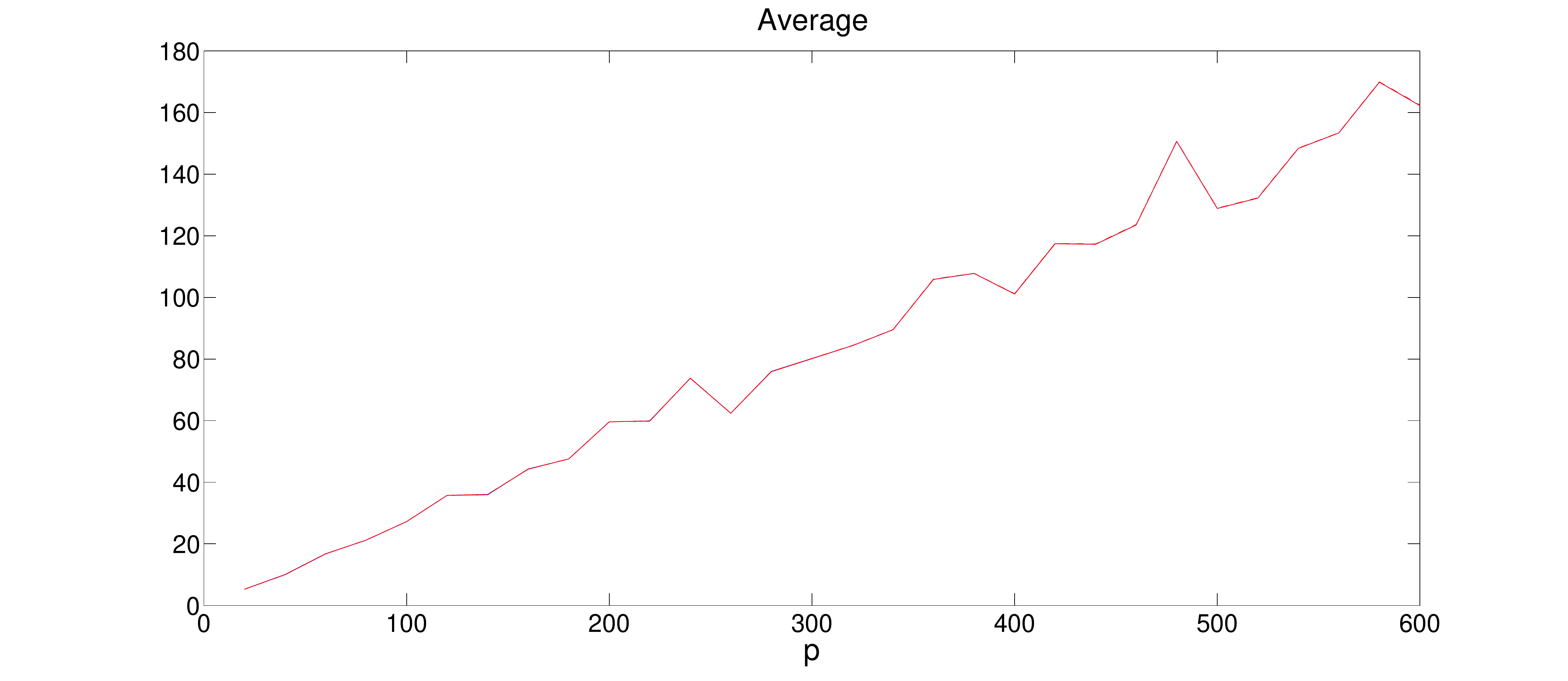}
\includegraphics[width=6.3cm]{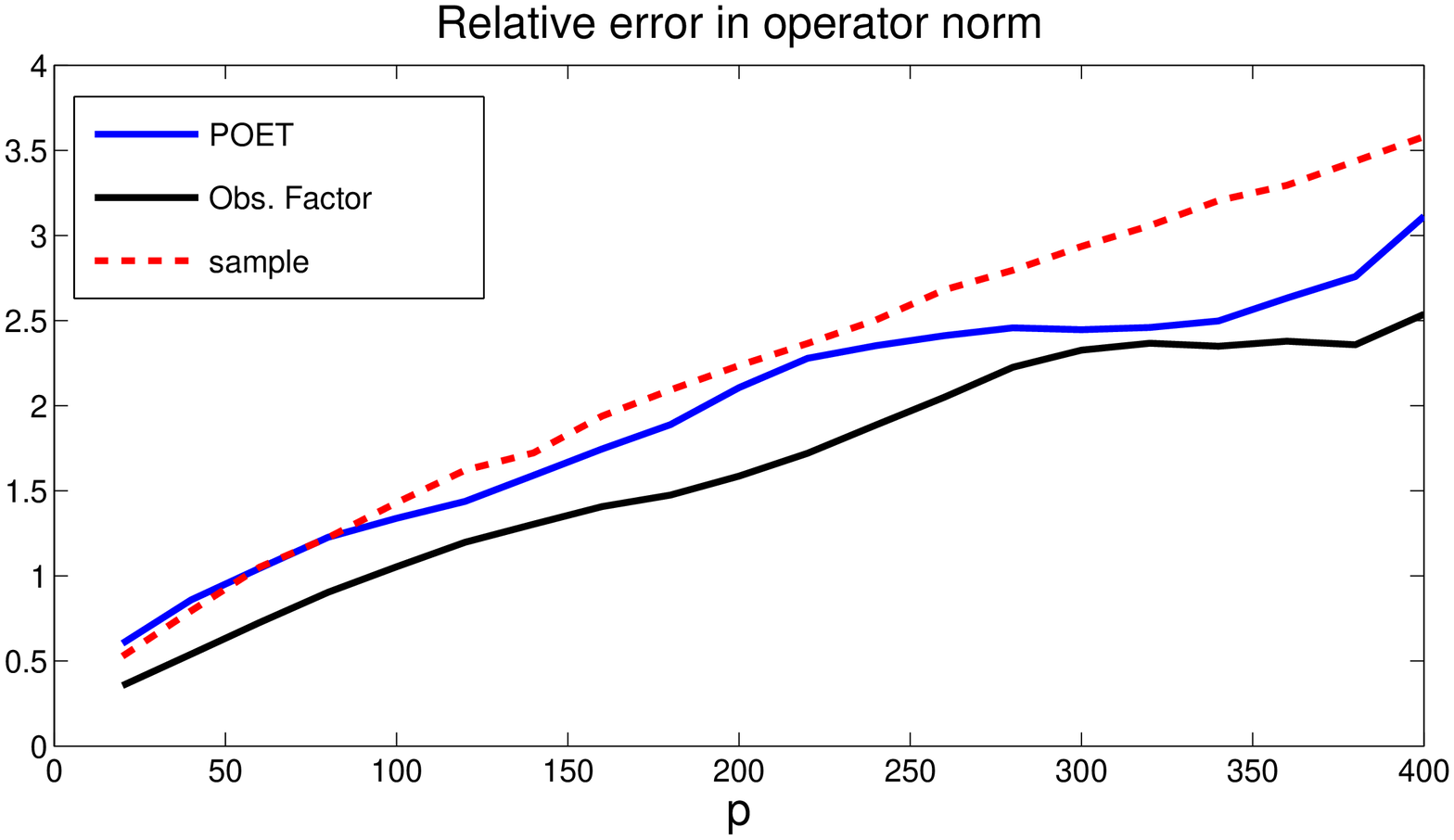}
\label{fig4}
\end{center}

\end{figure}

\subsection{Robustness to the estimation of $K$}
The POET estimator depends on the estimated number of factors. Our theory uses a consistent esimator
 $\widehat{K}$.  To assess the robustness of our procedure  to $\widehat{K}$ in finite sample, we calculate  $\hSig_{u,K}^{\mathcal{T}}$   for $K=1, 2,...,10$. Again, the threshold is fixed to be $ 0.5\sqrt{\hat{\theta}_{ij}}(\sqrt{\frac{\log p}{T}}+\frac{1}{\sqrt{p}})$.

\subsubsection{Design 1}

 The simulation setup is the same as before where the true $K_0=3$. We calculate  $\|\hSig_{u,K}^{\mathcal{T}}-\Sig_u\|$ ,  $\|(\hSig_{u,K}^{\mathcal{T}})^{-1}-\Sig_u^{-1}\|$,  $\|\hSig^{-1}_K-\Sig^{-1}\|$ and $\|\hSig_K-\Sig\|_{\Sigma}$ for $K=1,2,...,10$.
 Figure \ref{robK1} plots these norms as $p$ increases but with a fixed $T=300$. The results demonstrate a trend that is quite robust when $K\geq 3$; especially, the estimation accuracy of the spectral norms for large $p$ are close to each other. When $K=1$ or 2, the estimators perform badly due to modeling bias. Therefore, POET is robust to over-estimated $K$, but not to under-estimation.

\begin{figure}[htbp]
\begin{center}
\caption{Robustness of $K$ as $p$ increases for various choices of $K$ (Design 1, $T=300$).  Top left: $\|\hSig_{u,K}^{\mathcal{T}}-\Sig_u\|$; top right: $\|(\hSig_{u,K}^{\mathcal{T}})^{-1}-\Sig_u^{-1}\|$; bottom left: $\|\hSig_K-\Sig\|_{\Sigma}$; bottom right: $\|\hSig_K^{-1}-\Sig^{-1}\|$.}
\includegraphics[width=8cm]{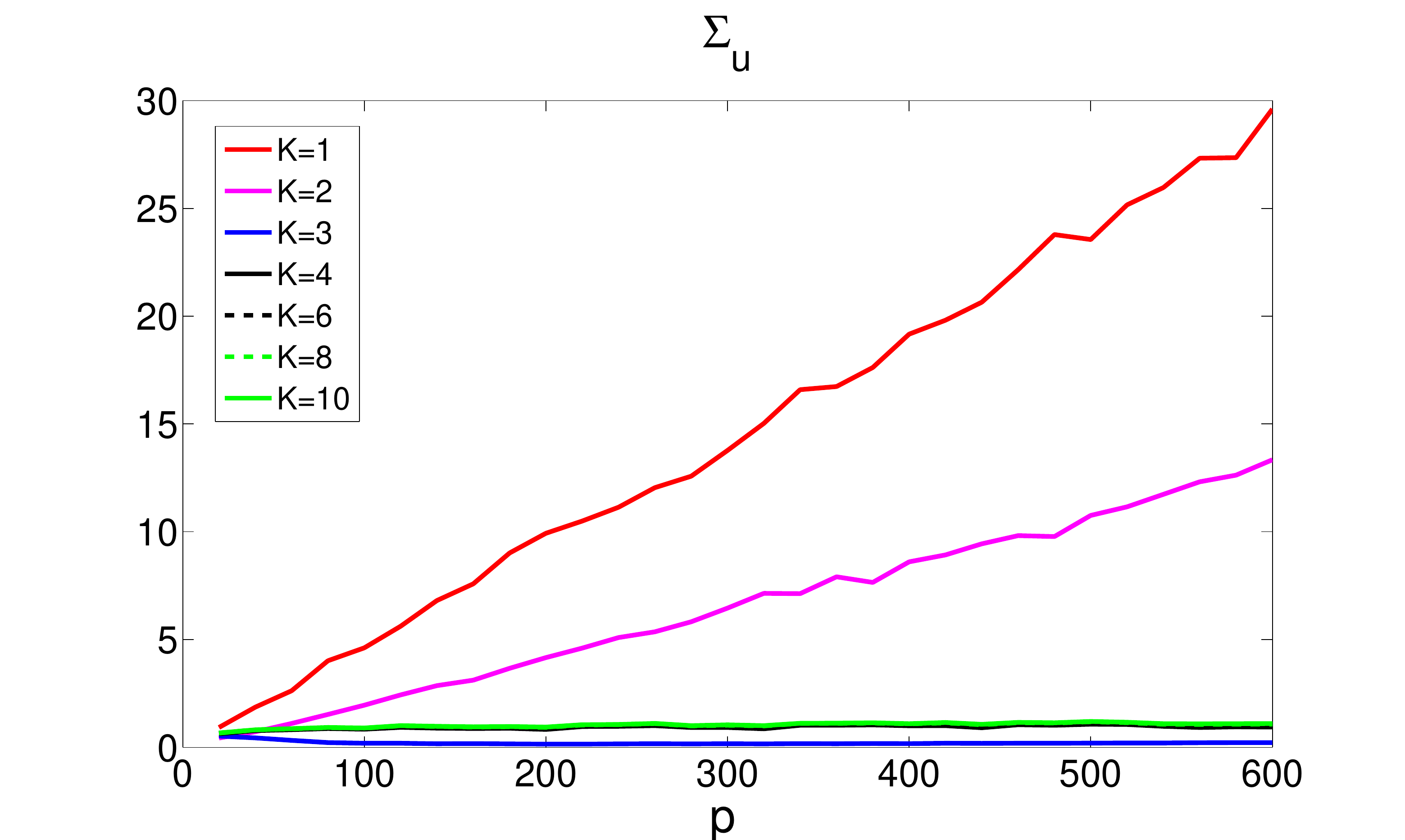}
\includegraphics[width=8cm]{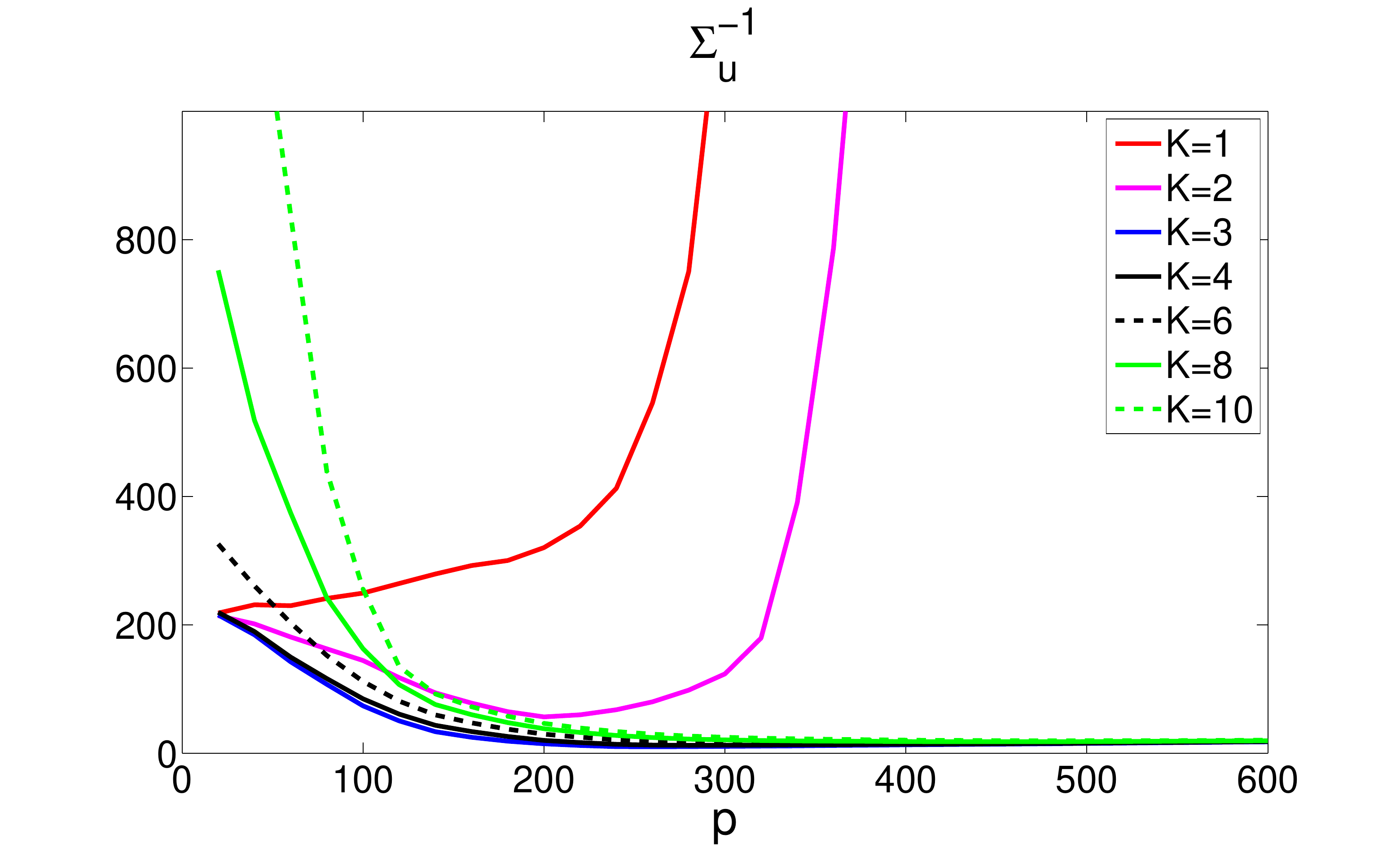}
\includegraphics[width=8cm]{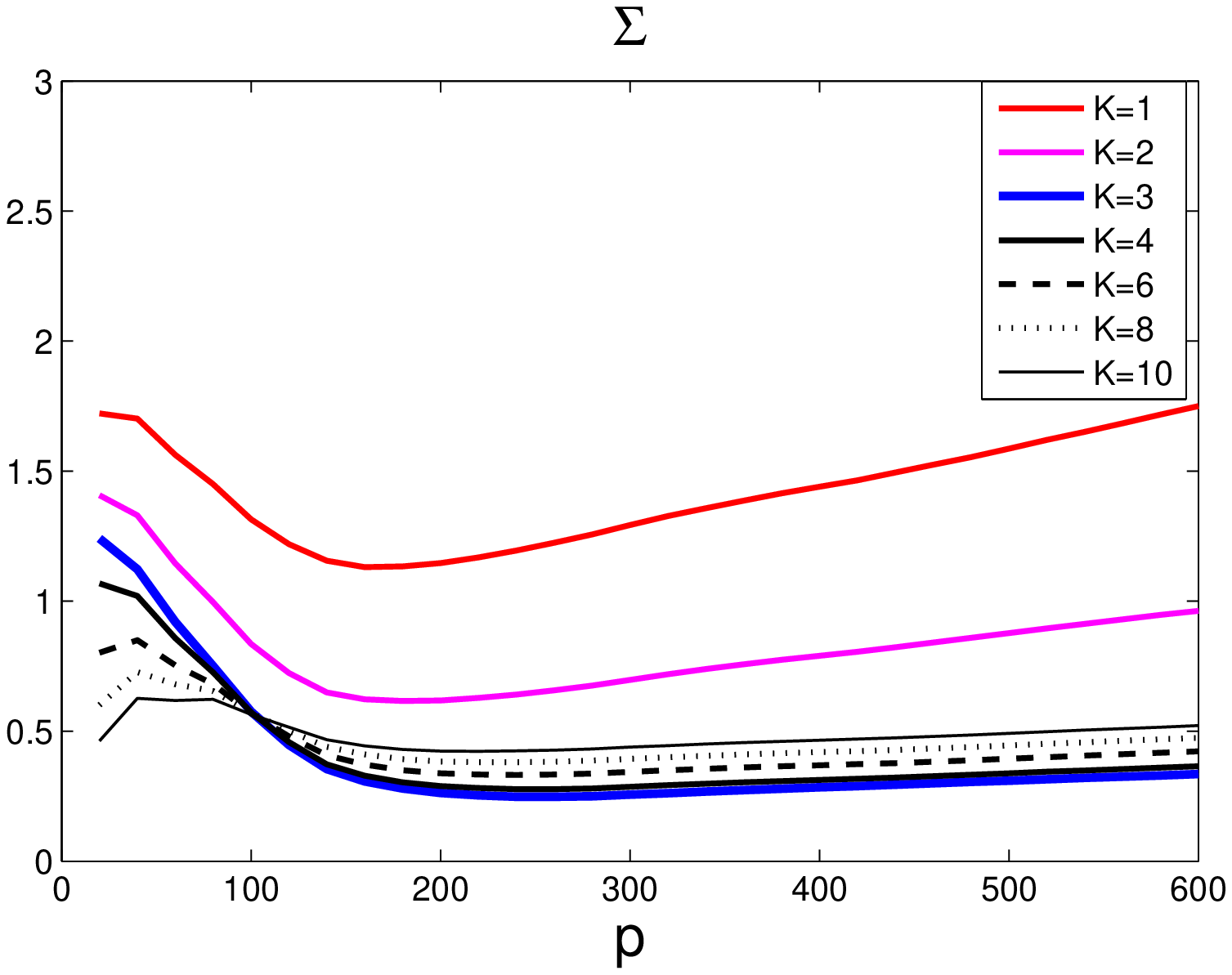}
\includegraphics[width=8cm]{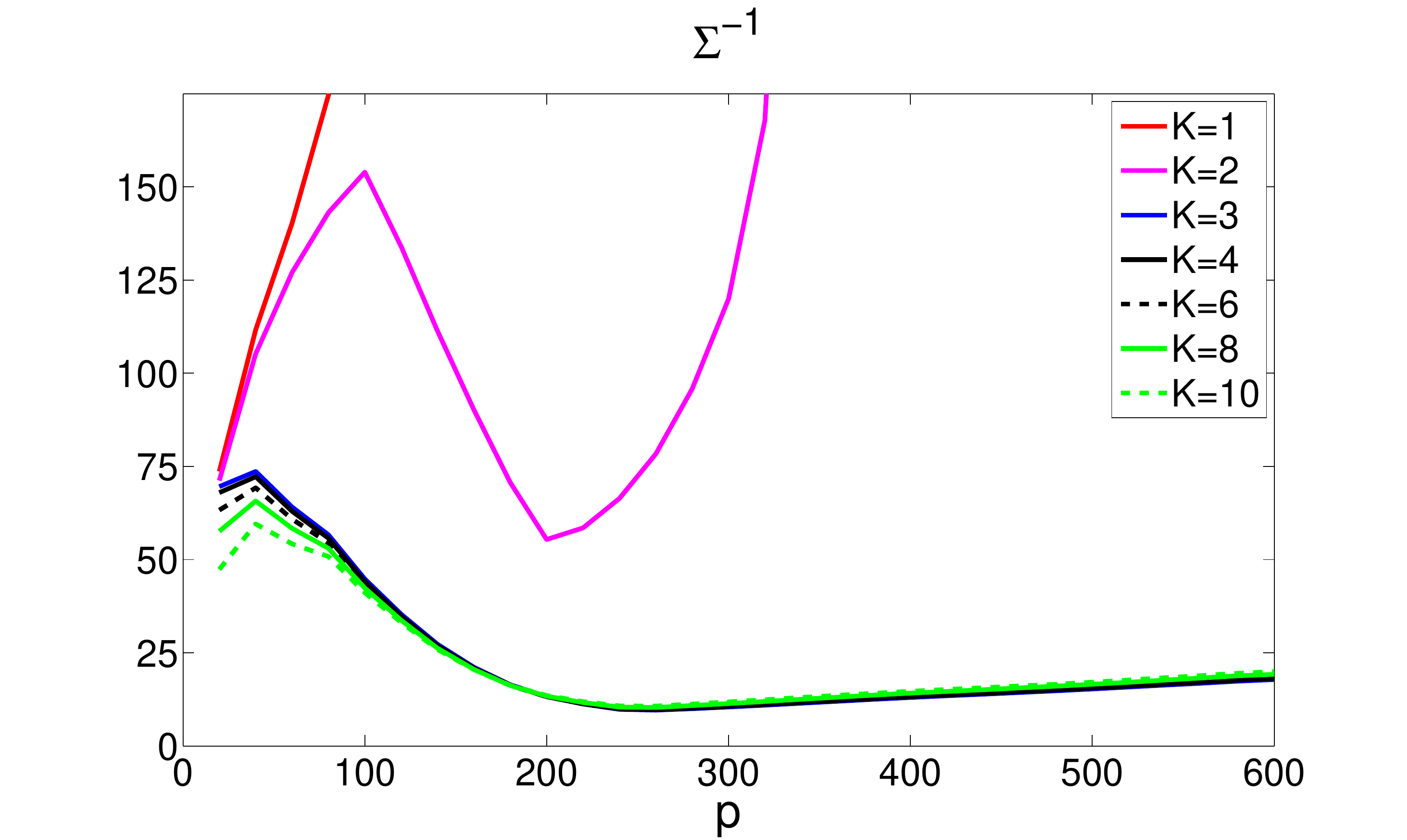}
\label{robK1}
\end{center}
\end{figure}

\subsubsection{Design 2}

We also simulated from a new data generating process for the robustness assessment. Consider a banded idiosyncratic matrix
$$
\sigma_{u,ij}=\begin{cases}
0.5^{|i-j|},&  |i-j|\leq 9\\
0,& |i-j|>9\end{cases},
\quad (\bu_1,...,\bu_T)\sim^{i.i.d.} N_p(0,\Sig_u).
$$
We still consider a $K_0=3$ factor model, where the factors are independently simulated as
$$f_{it}\sim N(0,1),\quad b_{ji}\sim N(0,1), \quad i\leq 3, j\leq p, t\leq T,$$
Table \ref{tableRob} summarizes the average estimation error of covariance matrices across $K$ in the spectral norm.  Each simulation is replicated 50 times and $T=200$.

 \begin{table}[htdp]

\begin{center}
\caption{Robustness of $K$. Design 2, estimation errors in spectral norm}

\label{tableRob}

\begin{tabular}{cc|ccccccccc}

\hline
 &  &  &  &  &   $K$ & &    & \\
  &  & 1 & 2 & 3 & 4 & 5 & 6 & 8 \\
  \hline
   &  &  &  &  &  &  &    & \\
$p=100$ & $\hSig_{u,K}^{\mathcal{T}}$ & 10.70 & 5.23 & 1.63 & 1.80 & 1.91 & 2.04 & 2.22 \\
 &  $(\hSig_{u,K}^{\mathcal{T}})^{-1}$ & 2.71 & 2.51 & 1.51 & 1.50 & 1.44 & 1.84 & 2.82 \\
 & $\hSig_{K}^{-1}$ & 2.69 & 2.48 & 1.47 & 1.49 & 1.41 & 1.56 & 2.35 \\
 &  $\hSig_{K}$ & 94.66 & 91.36 & 29.41 & 31.45 & 30.91 & 33.59 & 33.48 \\
 & $\Sig^{-1/2}\hSig_K\Sig^{-1/2}$& 17.37 & 10.04 & 2.05 & 2.83 & 2.94 & 2.95 & 2.93 \\
 \hline
 &  &  &  &  &  &  &  &  & \\
$p=200$ & $\hSig_{u,K}^{\mathcal{T}}$ & 11.34 & 11.45 & 1.64 & 1.71 & 1.79 & 1.87 & 2.01 \\
 &  $(\hSig_{u,K}^{\mathcal{T}})^{-1}$ & 2.69 & 3.91 & 1.57 & 1.56 & 1.81 & 2.26 & 3.42 \\
 & $\hSig_{K}^{-1}$ & 2.67 & 3.72 & 1.57 & 1.55 & 1.70 & 2.13 & 3.19 \\
 &  $\hSig_{K}$ & 200.82 & 195.64 & 57.44 & 63.09 & 64.53 & 60.24 & 56.20 \\
 & $\Sig^{-1/2}\hSig_K\Sig^{-1/2}$ & 20.86 & 14.22 & 3.29 & 4.52 & 4.72 & 4.69 & 4.76 \\
 \hline
 &  &  &  &  &  &  &  &   \\
$p=300$ & $\hSig_{u,K}^{\mathcal{T}}$ & 12.74 & 15.20 & 1.66 & 1.71 & 1.78 & 1.84 & 1.95 \\
 &  $(\hSig_{u,K}^{\mathcal{T}})^{-1}$ & 7.58 & 7.80 & 1.74 & 2.18 & 2.58 & 3.54 & 5.45 \\
 & $\hSig_{K}^{-1}$ & 7.59 & 7.49 & 1.70 & 2.13 & 2.49 & 3.37 & 5.13 \\
 &  $\hSig_{K}$ & 302.16 & 274.12 & 87.92 & 92.47 & 91.90 & 83.21 & 92.50 \\
 &$\Sig^{-1/2}\hSig_K\Sig^{-1/2}$ & 23.43 & 16.89 & 4.38 & 6.04 & 6.16 & 6.14 & 6.20 \\
\hline

\end{tabular}
\end{center}
\end{table}

Table \ref{tableRob} illustrates some interesting patterns. First of all, the best estimation accuracy is achieved when $K=K_0$. Second,   the estimation is robust for $K\geq K_0$. As $K$ increases from $K_0$, the estimation error becomes larger, but is increasing slowly in general, which  indicates the robustness when a slightly larger $K$ has been used.  Third, when the number of factors is under-estimated, corresponding to $K=1, 2$, all the estimators perform badly, which demonstrates the danger of missing any common factors. Therefore,   over-estimating the number of factors, while still maintaining a satisfactory estimation accuracy of the covariance matrices, is much better than under-estimating. The resulting bias caused by under-estimation is more severe than the additional variance introduced by over-estimation.  Finally, estimating $\Sig$, the covariance of $\by_t$, does not achieve a good accuracy even when $K=K_0$ in the absolute term $\|\hSig-\Sig\|$, but the relative error $\|\Sig^{-1/2}\hSig_K\Sig^{-1/2} - \bI_p\|$ is much smaller. This is consistent with our discussions in Section 3.3.

\subsection{Comparisons with Other Methods}

\subsubsection{Comparison with related methods}
We compare POET with related methods that address low-rank plus sparse covariance estimation, specifically,  LOREC proposed by Luo (2012), the strict factor model (SFM) by Fan, Fan and Lv (2008),  the Dual Method (Dual) by Lin et al. (2009),  and finally, the singular value thresholding (SVT) by Cai, Cand\`{e}s and Shen (2008). In particular, SFM is a special case of POET which employs a large threshold that forces $\hSig_u$ to be diagonal even when the true $\Sig_u$ might not be. Note that Dual, SVT  and many others dealing with low-rank plus sparse, such as Cand\`{e}s et al. (2011) and Wright et al. (2009), assume a known $\Sig$ and focus on recovering the decomposition. Hence they do not estimate $\Sig$ or its inverse, but    decompose the sample covariance into two components. The resulting sparse component may not be positive definite, which can lead to large estimation errors for $\hSig_u^{-1}$ and $\hSig^{-1}$.

Data are generated from the same setup as Design 2 in  Section 6.4. Table \ref{comp} reports the averaged estimation error of the four  comparing methods, calculated based on 50 replications for each simulation. Dual  and SVT assume the data matrix has a low-rank plus sparse representation, which is not the case for the sample covariance matrix (though the population $\Sig$ has such a representation).  
The tuning parameters for POET, LOREC, Dual  and SVT are chosen to achieve the best performance for each method.
\footnote{We used the R package for LOREC developed by Luo (2012) and the Matlab codes for Dual and SVT provided on Yi Ma's website ``Low-rank matrix recovery and completion via convex optimization" at University of Illinois. The tuning parameters for each method have been chosen to minimize the sum of relative errors $\|\Sig^{-1/2}\hSig \Sig^{-1/2} - \bI_p\| + \|\Sig_u^{-1/2}\hSig_u\Sig_u^{-1/2} - \bI_p\|$. We have also written an R package for POET.}

\begin{table}[htdp]
\begin{center}
\caption{Method Comparison under spectral norm for $T = 100$. RelE represents the relative error $\|\Sig^{-1/2}\hSig\Sig^{-1/2}-\bI_p\|$}

\label{comp}
\begin{tabular}{cc|ccccc}

\hline
 &  & $\hSig_u$ & $\hSig_u^{-1}$ &RelE& $\hSig^{-1}$ & $\hSig$\\
 \hline
  &  &  &  &  &  & \\
$p=100$ & POET & 1.624 & 1.336 & 2.080 & 1.309 & 29.107\\
 & LOREC & 2.274 & 1.880 & 2.564 & 1.511 & 32.365\\
 & SFM & 2.084 & 2.039 & 2.707 & 2.022 & 34.949\\
 & Dual & 2.306 & 5.654 & 2.707 & 4.674 & 29.000 \\
 & SVT & 2.59& 13.64 & 2.806 & 103.1 &29.670 \\
 \hline
 &  &  &  &  &  & \\
$p=200$ & POET & 1.641 & 1.358 & 3.295 & 1.346 & 58.769\\
 & LOREC & 2.179 & 1.767 & 3.874 & 1.543 & 62.731\\
 & SFM & 2.098 & 2.071 & 3.758 & 2.065 & 60.905\\
 & Dual & 2.41 & 6.554 & 4.541 & 5.813 & 56.264\\
 & SVT & 2.930 & $362.5$ & 4.680 &47.21  &63.670 \\
 \hline
 &  &  &  &  &  & \\
$p=300$ & POET & 1.662 & 1.394 & 4.337 & 1.395 & 65.392\\
 & LOREC & 2.364 & 1.635 & 4.909 & 1.742 & 91.618\\
 & SFM & 2.091 & 2.064 & 4.874 & 2.061 & 88.852\\
 & Dual & 2.475 & 2.602 & 6.190 &   2.234& 74.059 \\
 & SVT & 2.681 & $>10^3$ & 6.247& $>10^3$ & 80.954 \\
 \hline

\end{tabular}
\end{center}
\end{table}

\subsubsection{Comparison with direct thresholding}

This section compares POET with   direct thresholding on the sample covariance matrix without taking out common factors (Rothman et al. 2009, Cai and Liu 2011. We denote this method by THR). We also run simulations to demonstrate the finite sample performance when $\Sig$ itself is sparse and has bounded eigenvalues, corresponding to the case $K=0$. Three models are considered and both POET and THR use the soft  thresholding.  We fix $T=200$. Reported results are the average of 100 replications.

\textbf{Model 1:  one-factor.}  The factors and loadings are independently generated from $N(0,1)$. The error covariance is the same banded matrix as Design 2 in Section 6.4.  Here $\Sig$ has one diverging eigenvalue.

\textbf{Model 2:  sparse covariance.}  Set $K=0$, hence $\Sig=\Sig_u$ itself is a banded matrix with bounded eigenvalues.

\textbf{Model 3:  cross-sectional AR(1).} Set $K=0$, but   $\Sig=\Sig_u=(0.85^{|i-j|})_{p\times p}$. Now $\Sig$ is no longer sparse (or banded), but is not too dense either since  $\Sigma_{ij} $  decreases to zero exponentially fast  as $|i-j|\rightarrow\infty$. This is the correlation matrix if $\{y_{it}\}_{i=1}^p$ follows a cross-sectional AR(1) process: $y_{it}=0.85y_{i-1,t}+\varepsilon_{it}$.

For each model,   POET uses an estimated $\widehat{K}$ based on IC1 of Bai and Ng (2002), while THR thresholds the sample covariance directly. We find that in Model 1, POET performs significantly better than THR as the latter misses the common factor. For Model 2, IC1 estimates $\widehat{K}=0$ precisely in each replication, and hence POET is identical to THR. For Model 3,   POET still outperforms.  The results are summarized in  Table \ref{comp2}.

\begin{table}[htdp]
\begin{center}
\caption{Method Comparison. $T=200$}

\label{comp2}
\begin{tabular}{cc|cc|cc|c}
\hline
 &  &  \multicolumn{2}{c|}{ $\|\hSig-\Sig\|$ }      &   \multicolumn{2}{c|}{ $\|\hSig^{-1}-\Sig^{-1}\|$ }     &\\
 &  & POET & THR & POET & THR & $\widehat{K}$\\
 \hline
  &  &  &  &  &  & \\
 $p=200$ & Model 1 & 26.20 & 240.18 & 1.31 & 2.67&1 \\
 & Model 2 & 2.04 & 2.04 & 2.07 & 2.07&0 \\
 & Model 3 & 7.73& 11.24 & 8.48&  11.40 &6.2\\
 \hline
 &  &  &  &  &  & \\
$p=300$ & Model 1 & 32.60 & 314.43 & 2.18 & 2.58&1 \\
 & Model 2 & 2.03 & 2.03 & 2.08 & 2.08 &0\\
 & Model 3 & 9.41  &11.29  &8.81   &11.41   &5.45 \\

 \hline

\end{tabular}

The reported numbers are the averages based on 100 replications.
\end{center}
\end{table}

\subsection{Simulated portfolio allocation}
We demonstrate the improvement of our method compared to the sample covariance and that based on the strict factor model (SFM), in a problem of portfolio allocation for risk minimization purposes.

Let $\hSig$ be a generic estimator of the covariance matrix of the return vector $\by_t$, and $\bw$ be the allocation vector of a portfolio consisting of the corresponding $p$ financial securities. Then the theoretical and the empirical risk of the given portfolio are $R(\bw)=\bw'\Sig\bw$ and $\hat R(\bw)=\bw'\hSig\bw$, respectively. Now, define
$$
\hw=\argmin_{\bw'{\bf 1} =1}\bw'\hSig\bw,
$$
the estimated (minimum variance) portfolio. Then the  actual risk of the estimated portfolio is defined as $ R(\hw)=\hw'\Sig\hw$, and the estimated risk (also called empirical risk) is equal to $\hat R (\hw)=\hw'\hSig\hw$. In practice, the actual risk is unknown, and only the empirical risk can be calculated.  

For each fixed $p$, the population $\Sig$ was generated in the same way as described in Section 6.1, with a sparse but not diagonal error covariance.  We use three different methods to estimate $\Sig$ and obtain $\hw$:  strict factor model $\hSig_{\diag}$ (estimate $\Sig_u$ using a diagonal matrix),  our POET estimator $\hSig_{\text{POET}}$, both are with unknown factors, and sample covariance $\hSig_{\text{sam}}$. We then calculate the corresponding  actual and empirical risks.

It is interesting to examine the accuracy and the performance of the actual risk of our portfolio $\hw$ in comparison to the oracle risk $R^*=\min_{\bw'\bone=1}\bw'\Sig\bw$,  which is the theoretical risk of the portfolio we would have created if we knew the true covariance matrix $\Sig$.  We thus compare
the regret $R(\hw)-R^*$, which is always nonnegative, for three estimators of $\hSig$. They are summarized by using the box plots over the 200 simulations.
The results are reported in Figure~\ref{fig5}.  In practice, we are also  concerned about the difference between the actual and empirical risk of the chosen portfolio $\hw$.  Hence, in Figure~\ref{fig6}, we also compare the average estimation error $|R(\hw)-\hat R(\hw)|$  and the average relative estimation error $|\hat R(\hw)/R(\hw) -1|$ over 200 simulations.
When $\hw$ is obtained based on the strict factor model, both differences - between actual and oracle risk, and between actual and empirical risk, are persistently greater than the corresponding differences for the approximate factor estimator.  Also, in terms of the relative estimation error, the factor model based method is negligible, where as the sample covariance does not process such a property.

\begin{figure}[htbp]
\begin{center}
\caption{Box plots of regrets $R(\hw)-R^*$ for $p = 80$ and $140$. In each panel, the box plots from left to right correspond to $\hw$ obtained using  $\hSig$ based on approximate factor model, strict factor model, and sample covariance, respectively.}
\includegraphics[width=6cm]{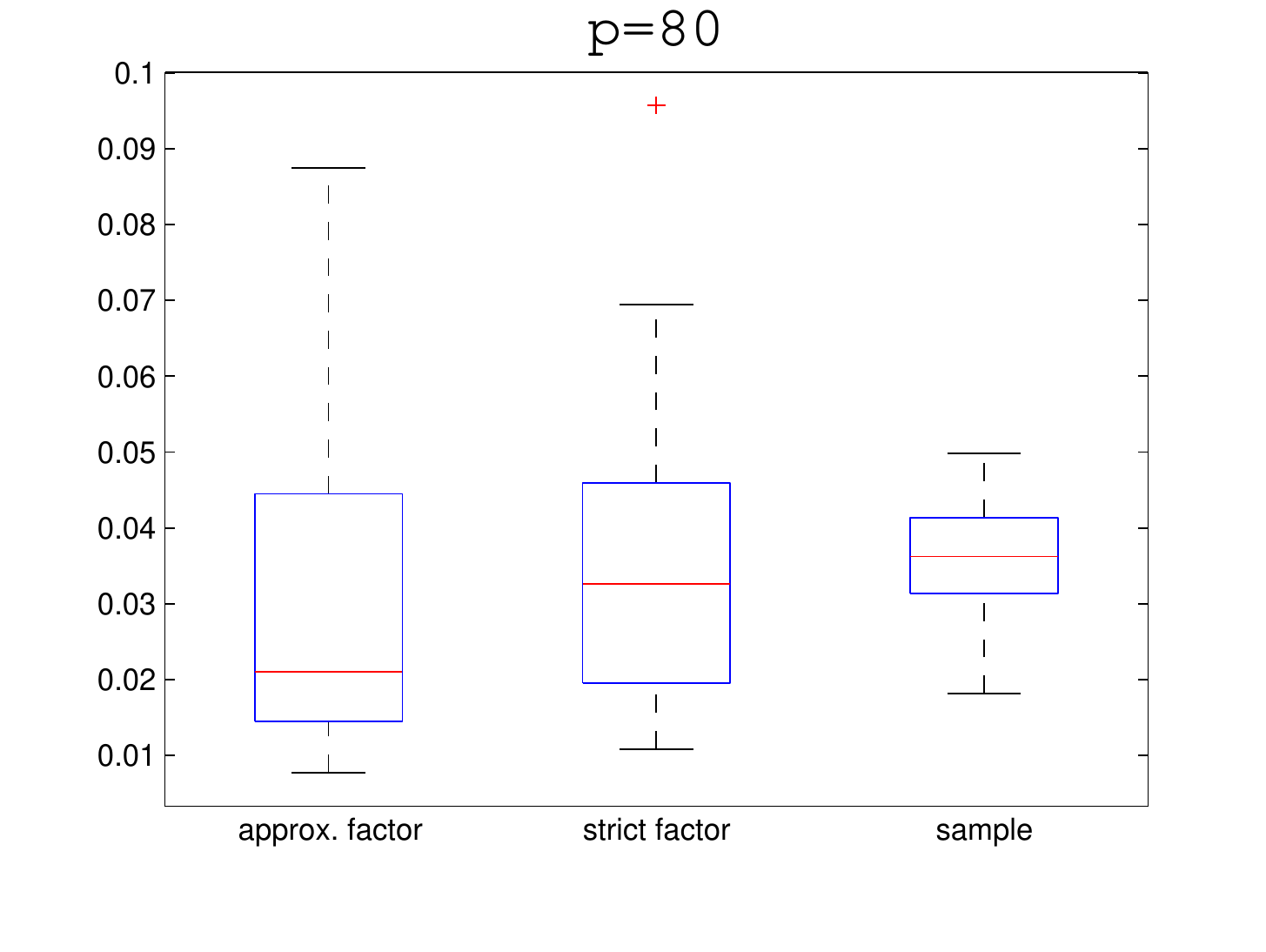}
\includegraphics[width=6cm]{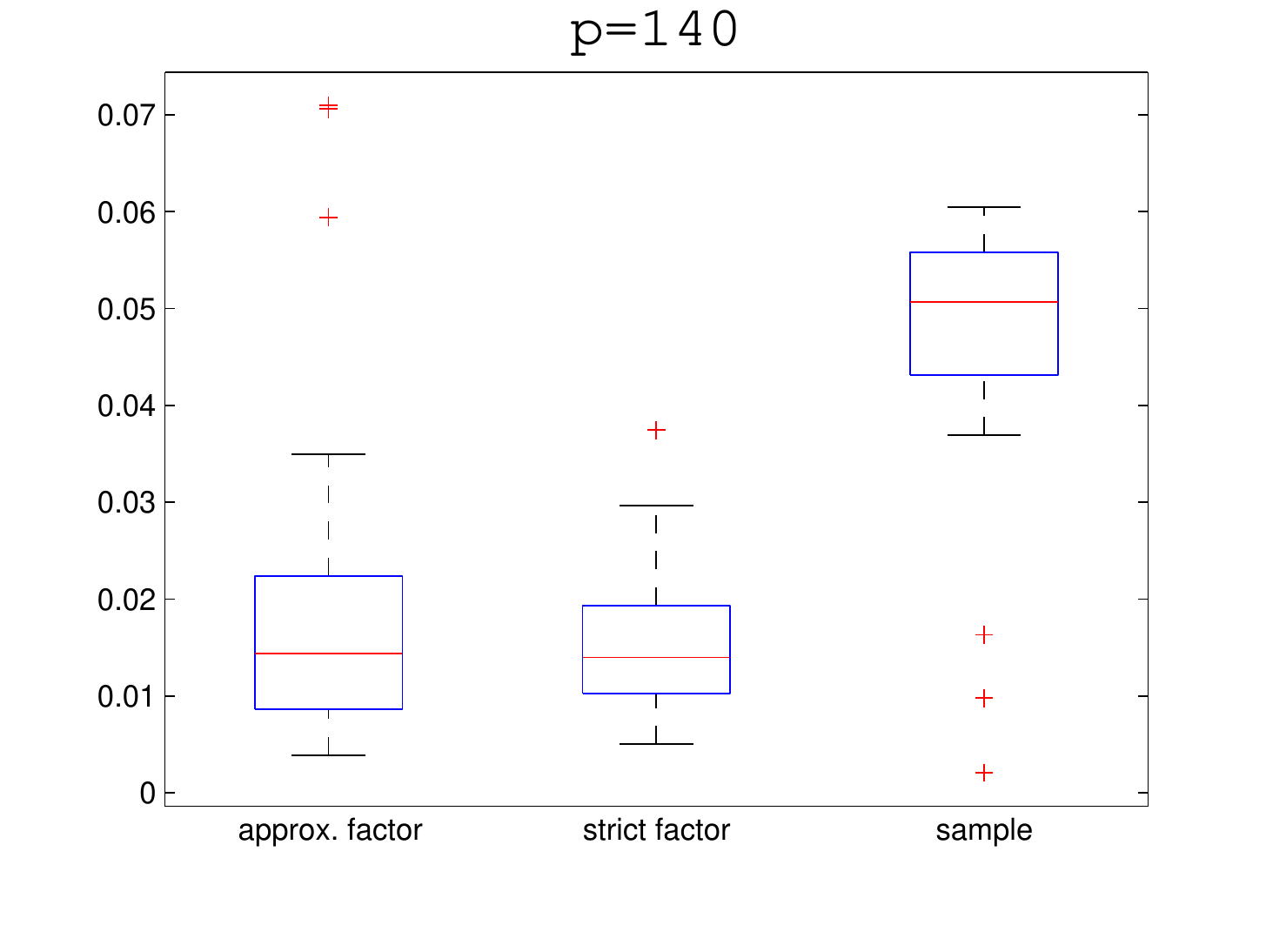}
\label{fig5}
\end{center}
\end{figure}

\begin{figure}[htbp]
\begin{center}
\caption{Estimation errors for risk assessments as a function of the portfolio size $p$. Left panel plots the average absolute error $|R(\hw)-\hat R(\hw)|$ and right panel depicts the average relative error  $|\hat{R}(\hw)/ R(\hw)-1|$.  Here,
$\hw$ and $\hat{R}$ are obtained based on three estimators of $\hSig$.
}
\includegraphics[width=6cm]{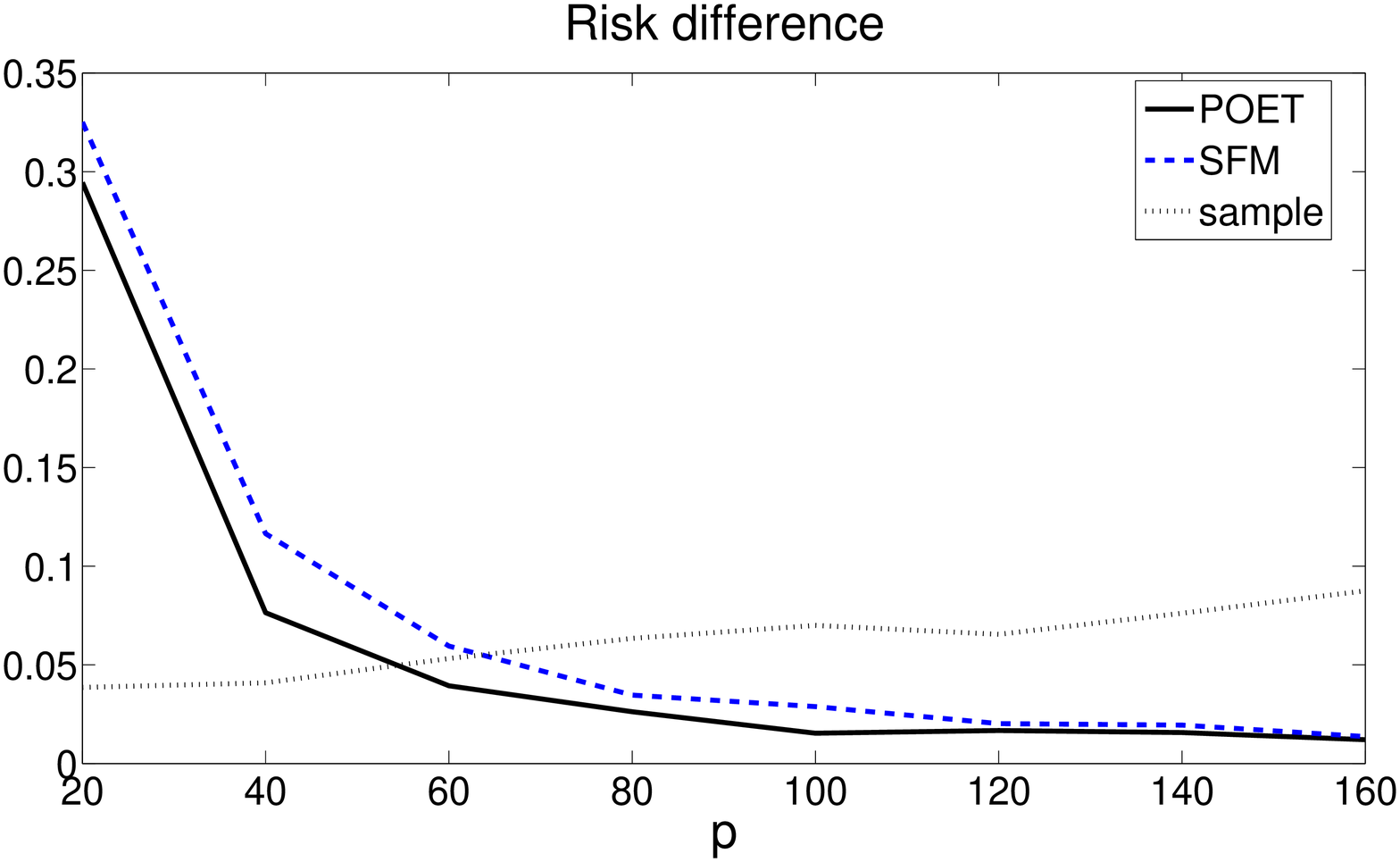}
\includegraphics[width=6cm]{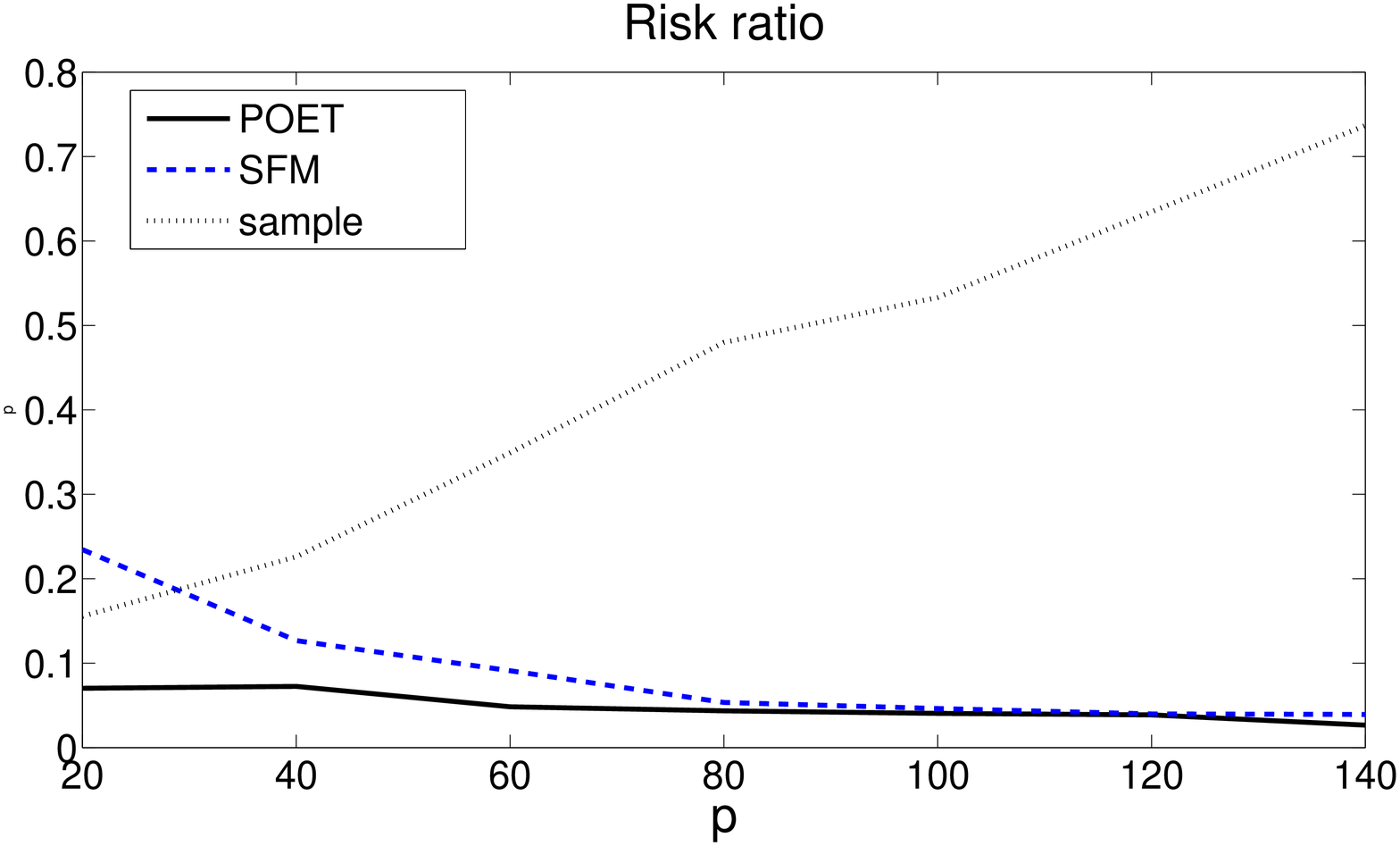}
\label{fig6}
\end{center}
\end{figure}

\section{Real Data Example}
We demonstrate the sparsity of  the approximate factor model on real data, and present the improvement of the POET estimator over the strict factor model (SFM) in a real-world application of portfolio allocation.

\subsection{Sparsity of Idiosyncratic Errors}
 The data were  obtained from the CRSP (The Center for Research in Security Prices) database, and consists of $p = 50$ stocks and their annualized daily returns for the period January $1^{st},2010$-December $31^{st}, 2010$  ($T = 252$). The stocks are chosen from $5$ different industry sectors,  (more specifically, Consumer Goods-Textile $\&$ Apparel Clothing,  Financial-Credit Services, Healthcare-Hospitals, Services-Restaurants, Utilities-Water utilities), with $10$ stocks from each sector. We made this selection to demonstrate a block diagonal trend in the sparsity. More specifically, we show that the non-zero elements are clustered mainly within companies in the same industry. We also notice that these are the same groups that show predominantly positive correlation.

The largest eigenvalues of the sample covariance equal $0.0102, 0.0045$ and $0.0039$, while the rest are bounded by $0.0020$. Hence   $K=0, 1, 2, 3$ are the possible values of the number of factors.   Figure \ref{fig7} shows the heatmap of the thresholded error correlation matrix (for simplicity, we applied hard thresholding). The threshold  has been chosen using the cross validation as described in Section 4. We compare the level of sparsity (percentage of non-zero off-diagonal elements) for the 5 diagonal blocks of size $10 \times 10$, versus the sparsity of the rest of the matrix. For $K=2$, our method results in $25.8\%$ non-zero off-diagonal elements in the $5$ diagonal blocks, as opposed to $7.3\%$ non-zero elements in the rest of the covariance matrix. Note that, out of the non-zero elements in the central $5$ blocks, $100\%$ are positive, as opposed to a distribution of $60.3\%$ positive and $39.7\%$ negative amongst   the non-zero elements in off-diagonal blocks. There is a strong positive correlation between the returns of companies in the same industry after the common factors are taken out, and the thresholding has preserved them. The results for $K=1, 2$ and $3$ show the same characteristics.  These provide stark evidence that the strict factor model is not appropriate.

\begin{figure}[htbp]
\begin{center}
\caption{Heatmap of thresholded error correlation matrix for number of factors $K=0$, $K=1$, $K=2$ and $K=3$.}
\includegraphics[width=6.3cm]{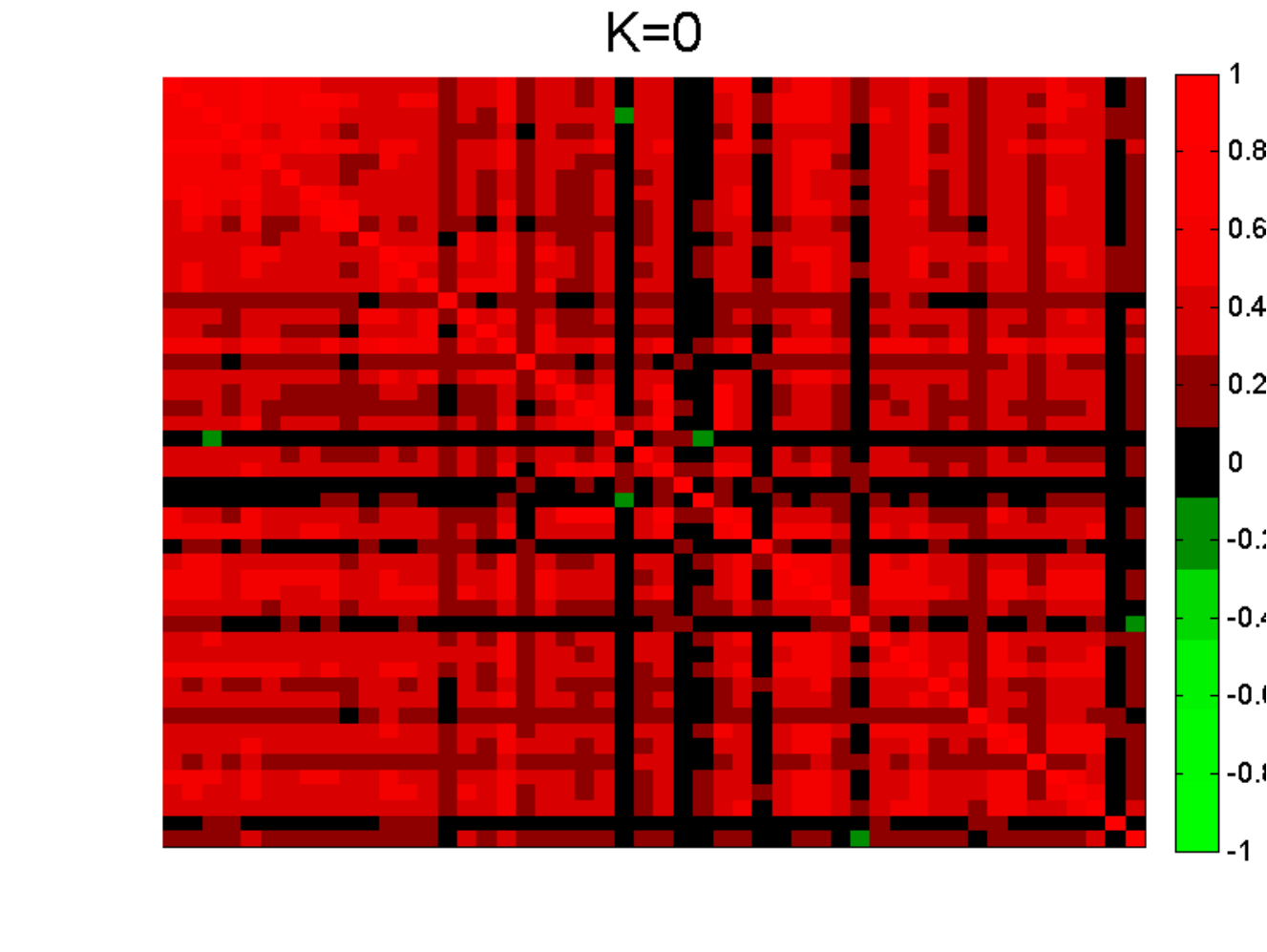}
\includegraphics[width=6.3cm]{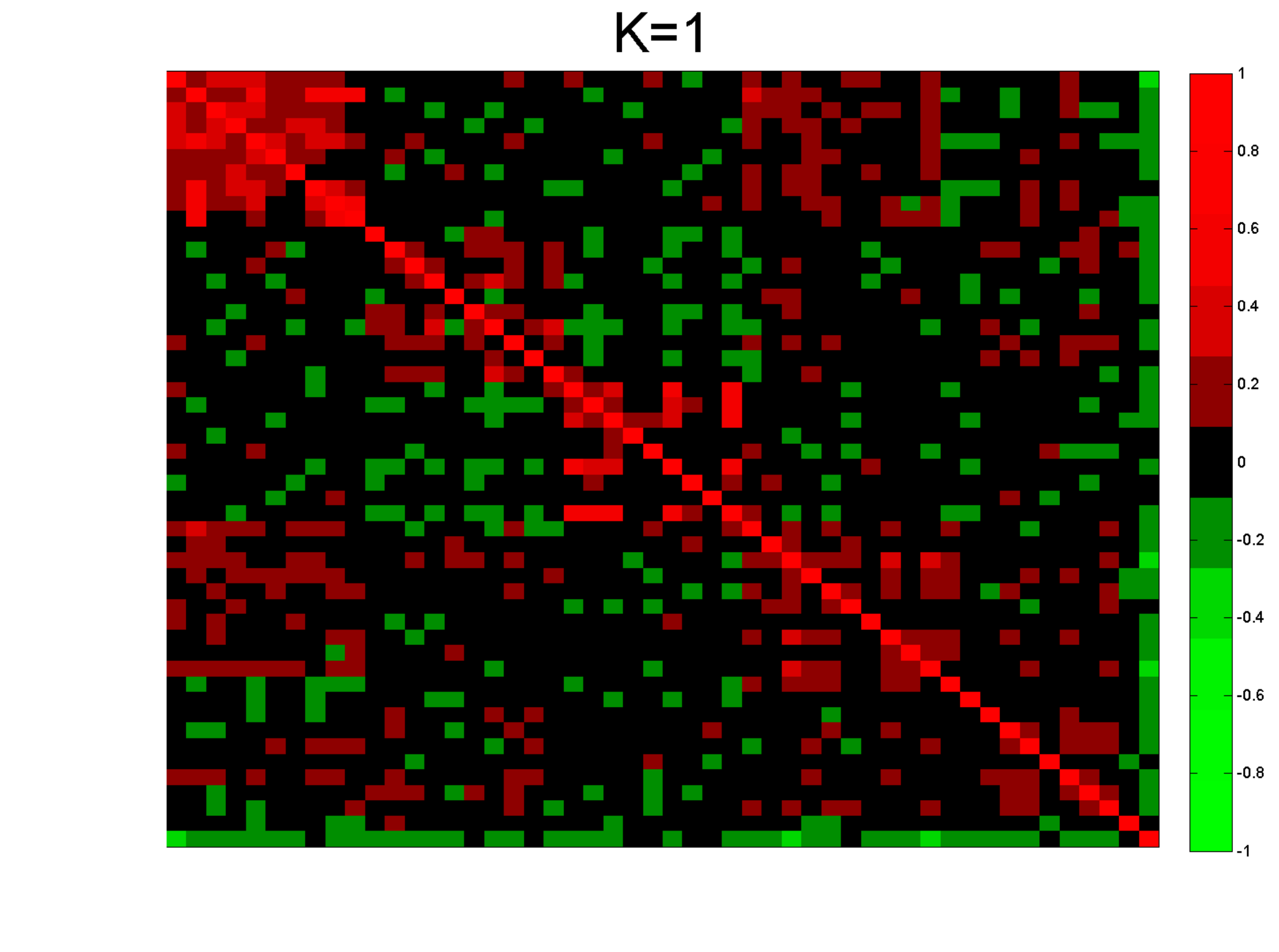}
\includegraphics[width=6.3cm]{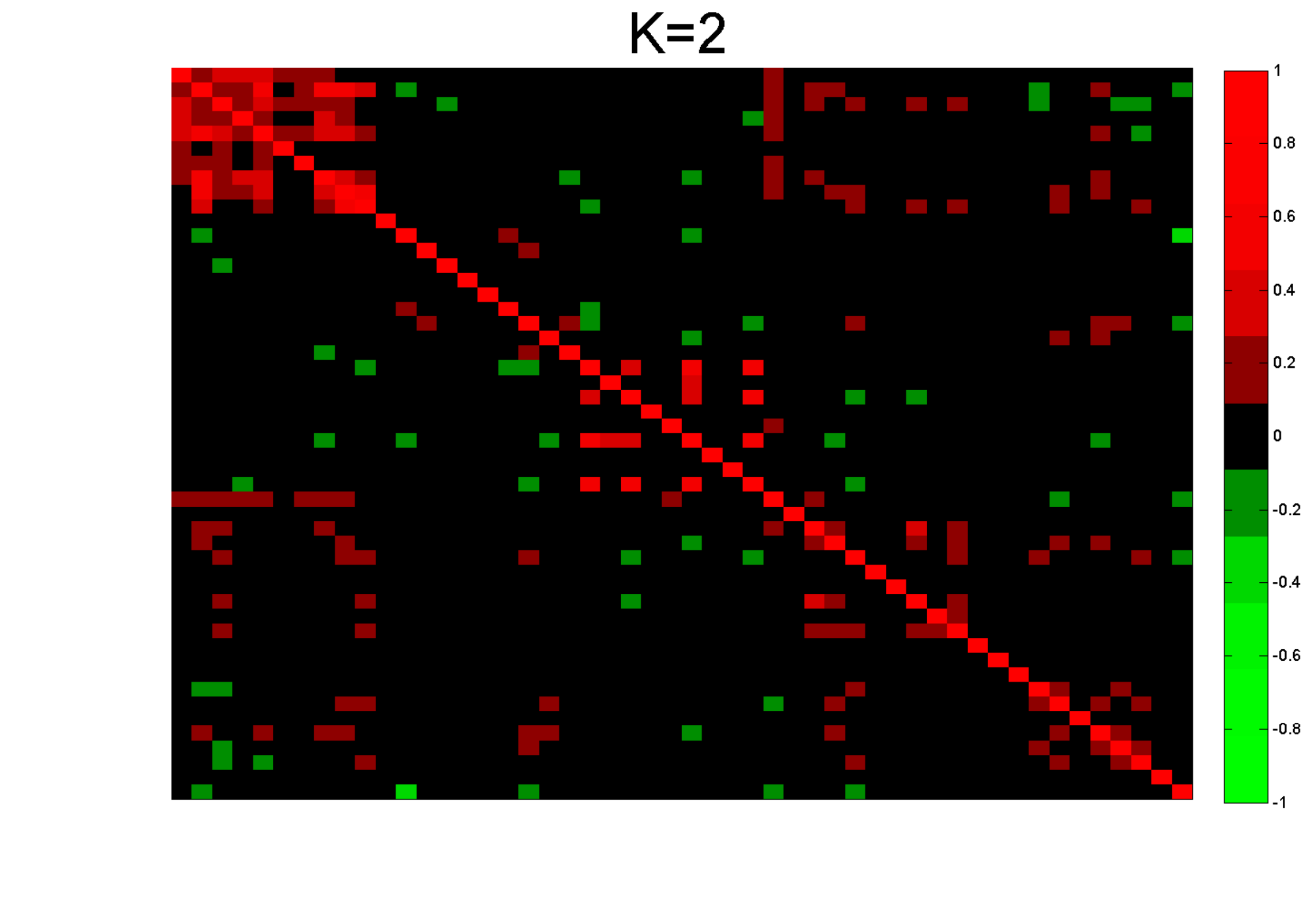}
\includegraphics[width=6.3cm]{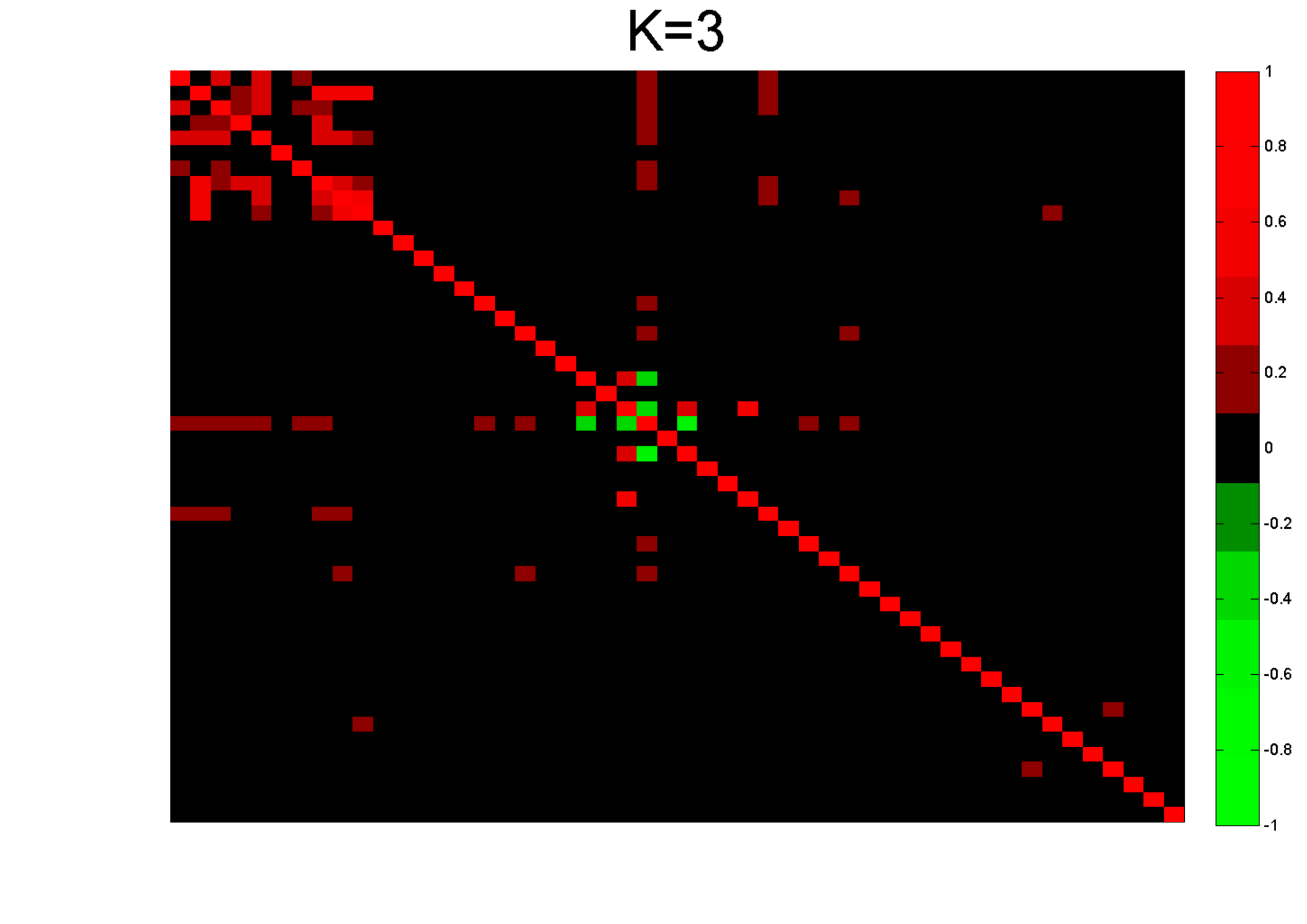}
\label{fig7}
\end{center}
\end{figure}

\subsection{Portfolio Allocation}

We extend our  data size by including larger industrial portfolios ($p=100$), and longer period (ten years): January $1^{st}$,2000 to December $31^{st}$, 2010 of annualized daily excess returns. Two portfolios are created at the beginning of each month, based on two different covariance estimates through approximate and strict factor models with unknown factors. At the end of each month, we compare the risks of both portfolios.  

 The number of factors is determined using the penalty function proposed by Bai and Ng (2002), as defined in (\ref{eq2.16add}). For calibration, we use the last $100$ consecutive business days of the above data, and both IC1 and IC2 give $\hat K=3$. On the $1^{st}$ of each month, we estimate $\hSig_{\diag}$ (SFM) and $\hSig_{\widehat{K}}$ (POET with soft thresholding) using the historical data of excess daily returns for the proceeding 12 months ($T=252$).
The value of the threshold is determined using the cross-validation procedure. We minimize the empirical risk of both portfolios to obtain the two respective optimal portfolio allocations $\hw=\hw_1$ and $\hw_2$ (based on $\hSig=\hSig_{\diag}$ and $\hSig_{\widehat{K}}$): $\hw=\arg\min_{\hw'\bone=1}\bw'\hSig\bw$. At the end of the month (21 trading days), their actual risks are compared, calculated by
$$R_i=\hw_i'\frac{1}{21}\sum_{t=1}^{21}\by_t\by_t'\hw_i\text{ , for $i=1,2.$}$$
We can see from Figure \ref{fig8} that the minimum-risk portfolio created by the POET estimator performs significantly better, achieving lower variance $76\%$ of the time. Amongst those months, the risk is decreased by $48.63\%$. On the other hand, during the months that POET produces a higher-risk portfolio, the risk is increased by only $17.66\%$.

Next, we demonstrate the impact of the choice of number of factors and threshold on the performance of POET.  If cross-validation seems computationally expensive, we can choose a common soft-threshold throughout the whole investment process. The average constant in the cross-validation was $0.53$, close to our suggested constant $0.5$ used for simulation. We also present the results based on various choices of constant $C=0.5$,$0.75$,$1$ and $1.25$, with soft threshold $C\sqrt{\hat{\theta}_{ij}}\omega_T$. The results are summarized in Table \ref{tab5}.  The performance of POET seems consistent across different choices of these parameters.

\begin{figure}[htbp]
\begin{center}
\caption{Risk of portfolios created with POET and SFM (strict factor model)}
\includegraphics[width=8cm]{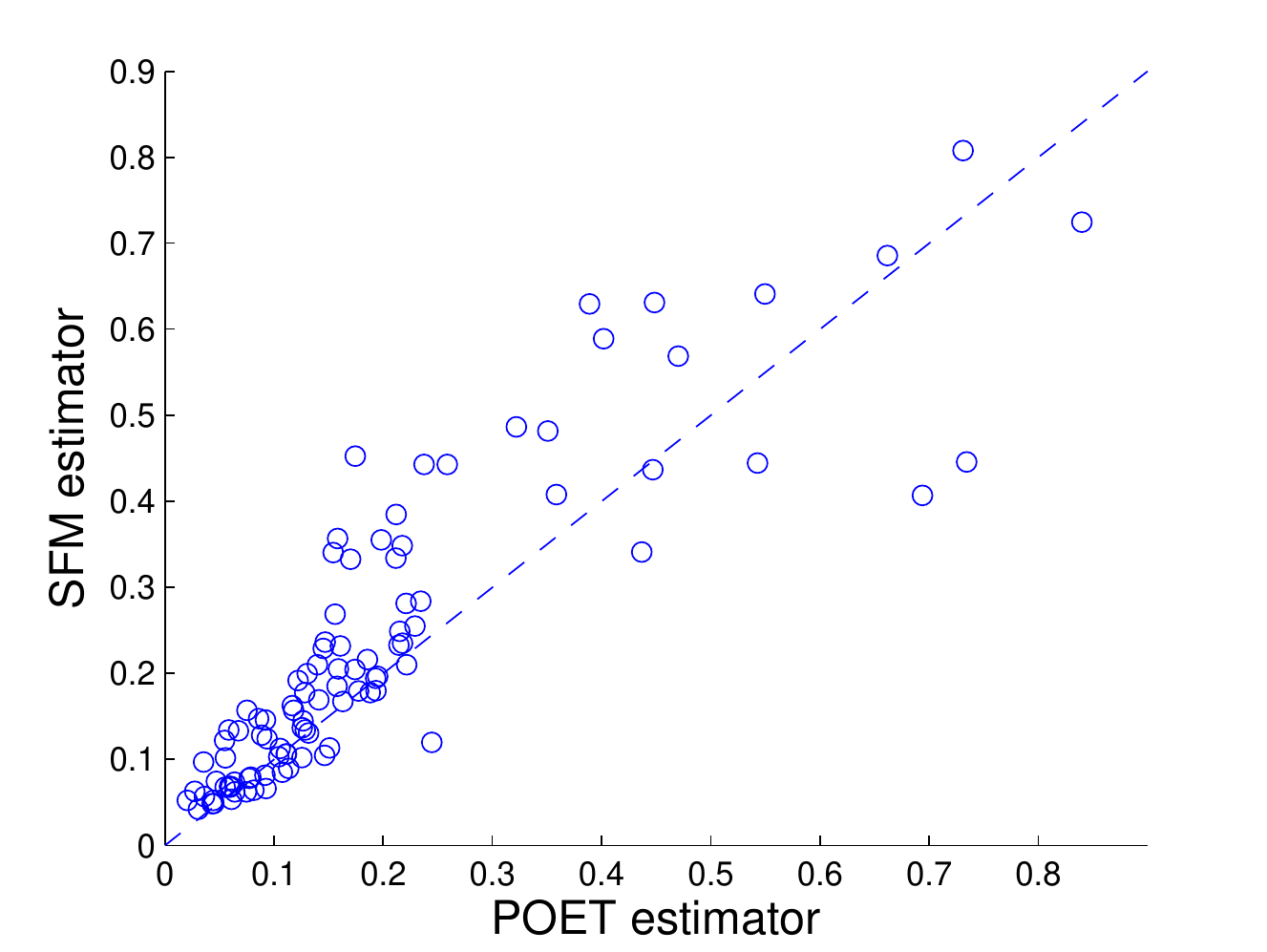}
\label{fig8}
\end{center}
\end{figure}

\begin{table}[htdp]

\begin{center}
\caption{Comparisons of the risks of portfolios using POET and SFM:  The first number is proportion of the time POET outperforms and the second number is percentage of average risk improvements. $C$ represents the constant in the threshold.}

\label{tab5}

\begin{tabular}{c|cccc}

\hline
 &  &  &  &  \\
 $C$&  $\hat K=1$ & $\hat K=2$ & $\hat K=3$ \\
\hline
 &  &  &  &  \\
$0.25$ & 0.58/29.6\% & 0.68/38\% & 0.71/33\%\\
 &  &  &   & \\
$0.5$ & 0.66/ 31.7\%& 0.70/ 38.2\%& 0.75/33.5\%\\
 &  &  &  &  \\
$0.75$ &  0.68/29.3\% & 0.70/29.6\% & 0.71/ 25.1\%\\
 &  &  &  &  \\
$1$ &  0.66/20.7\% & 0.62/19.4\% & 0.69/18\% \\

\hline

\end{tabular}
\end{center}
\end{table}

\section{Conclusion and Discussion}

We study the problem of estimating a high-dimensional covariance matrix with conditional sparsity. Realizing unconditional sparsity assumption is inappropriate in many applications, we introduce a  latent factor model that has a conditional sparsity feature, and propose the POET estimator to take advantage of the structure. This expands considerably the scope of the model based on the strict factor model, which assumes independent idiosyncratic noise and is too restrictive in practice. By assuming sparse error covariance matrix,  we allow for the presence of the cross-sectional correlation even after taking out the common factors. The sparse covariance is estimated by the adaptive thresholding technique.

It is found that the rates of convergence of the estimators have an extra term approximately $O_p(p^{-1/2})$ in addition to the results based on observable factors by Fan et al. (2008, 2011), which arises from the effect of estimating the unobservable factors. As we can see, this effect vanishes as the dimensionality increases, as more information about the common factors becomes available. When $p$ gets large enough, the effect of estimating the unknown factors is negligible, and we estimate the covariance matrices as if we knew the factors.

  The proposed POET also has wide applicability in statistical genomics.  For example, Carvalho  et al. (2008) applied a Bayesian sparse factor model  to study the breast cancer hormonal pathways. Their real-data results have identified  about  two common factors that have  highly loaded  genes (about half of 250 genes). As a result, these factors should be treated as ``pervasive" (see the explanation in Example 2.1),  which will result in one or two very spiked eigenvalues of the gene expressions' covariance matrix.  The POET can be applied to estimate such a covariance matrix and its network model.

\newpage

\begin{center}
{\Large \bf APPENDIX}
\end{center}

\appendix

\section{Estimating a sparse covariance with contaminated data}

We estimate $\Sig_u$ by applying the adaptive thresholding given by (\ref{eq2.13}). However, the task here is slightly different from the standard problem of estimating a sparse covariance matrix in the literature, as no direct observations for $\{\bu_t\}_{t=1}^T$ are available. In many cases the original  data are  contaminated, including any type of estimate of the data when direct
observations are not available. This typically happens when $\{\bu_t\}_{t=1}^T$ represent the error terms in regression models  or when data is subject to measurement of errors. Instead, we may observe $\{\hu_t\}_{t=1}^T$. For instance,  in the approximate factor models, $\hu_t=\by_t-\hb_i'\hf_t.$

We can estimate $\Sig_u$ using the adaptive thresholding proposed by Cai and Liu (2011):  for the threshold
$\tau_{ij}=C\sqrt{\hat{\theta}_{ij}}\omega_T,$
define
$$
 \hsig_{ij}=\frac{1}{T}\sum_{t=1}^T\hat{u}_{it}\hat{u}_{jt}, \quad \text{ and } \quad \hat{\theta}_{ij}=\frac{1}{T}\sum_{t=1}^T\left(\hat{u}_{it}\hat{u}_{jt}-\hsig_{ij}\right)^2.
$$
\begin{equation} \label{eq3.2add}
\hSig_{u}^{\mathcal{T}}=(s_{ij}(\hsig_{ij}))_{p\times p},
\end{equation}
where $s_{ij}(.)$ satisfies: for all $z\in\mathbb{R}$,
$
s_{ij}(z)=0, \text{ when }|z|\leq\tau_{ij};$ $|s_{ij}(z)-z|\leq \tau_{ij}.$

When $\{\hu_t\}_{t=1}^T$ is close enough to $\{\bu_t\}_{t=1}^T$, we can show that $\hSig_{u}^{\mathcal{T}}$ is also consistent. The following theorem extends the standard thresholding results in Bickel and Levina (2008) and Cai and Liu (2011) to the case when no direct observations are available, or the original data are contaminated.  For the tail and mixing parameters $r_1$ and $r_3$  defined in Assumptions \ref{a21} and \ref{a32}, let $\alpha=3r_1^{-1}+r_3^{-1}+1$.

\begin{thm} \label{tb1}
Suppose $(\log p)^{6\alpha}=o(T)$, and Assumptions \ref{a21} and \ref{a32} hold.  In addition, suppose  there is a sequence $a_T=o(1)$ so that
$\max_{i\leq p}\frac{1}{T}\sum_{t=1}^T|u_{it}-\hat{u}_{it}|^2=O_p(a_T^2),$ and $ \max_{i
\leq p, t\leq T}|u_{it}-\hat{u}_{it}|=o_p(1);$
Then there is a constant $C>0$ in the adaptive thresholding estimator (\ref{eq3.2add}) with
$$
\omega_T=\sqrt{\frac{\log p}{T}}+a_T
$$
such that
$$
    \|\hSig_u^{\mathcal{T}}-\Sig_u\|=O_p\left(\omega_T^{1-q}m_p \right).
$$
If further $\omega_Tm_p=o(1)$, then $\hSig_u^{\mathcal{T}}$ is invertible with probability approaching one, and
$$\|(\hSig_u^{\mathcal{T}})^{-1}-\Sig_u^{-1}\|=O_p\left(\omega_T^{1-q}m_p\right).$$
\end{thm}

\begin{proof} By Assumptions \ref{a21} and \ref{a32}, the conditions of Lemmas A.3 and A.4 of Fan, Liao and Mincheva (2011, \textit{Ann. Statist}, \textbf{39}, 3320-3356) are satisfied. Hence for any $\epsilon>0$, there are positive constants $M, \theta_1$ and $\theta_2$ such that each of the events
\begin{eqnarray*}
A_1&=&\{\max_{i\leq p, j\leq p}|\hsig_{ij}-\sigma_{u,ij}|<M\omega_T\}\cr
A_2&=&\{\theta_1>\sqrt{\hat{\theta}_{ij}}>\theta_2, \text{ all } i\leq p, j\leq p\}.
\end{eqnarray*}
occurs with probability at least $1-\epsilon$.
By the condition of threshold function, $s_{ij}(t)=s_{ij}(t)I_{|t|>C\omega_T\sqrt{\hat{\theta}_{ij}}}$. Now for $C=\theta_2^{-1}2M,$ under the event $A_1\cap A_2,$
 \begin{eqnarray*}
\|\hSig_u^{\mathcal{T}}-\Sig_u\|&\leq& \max_{i\leq p}\sum_{j=1}^p|s_{ij}(\hsig_{ij})-\sigma_{u,ij}|\cr
&=&        \max_{i\leq p}\sum_{j=1}^p|s_{ij}(\hsig_{ij})I_{(|\hsig_{ij}|>C\omega_T\sqrt{\hat{\theta}_{ij}})}-\sigma_{u,ij}I_{(|\hsig_{ij}|>C\omega_T\sqrt{\hat{\theta}_{ij}})}-\sigma_{u,ij}I_{(|\hsig_{ij}|\leq C\omega_T\sqrt{\hat{\theta}_{ij}})}|              \cr
&\leq&   \max_{i\leq p}\sum_{j=1}^p|s_{ij}(\hsig_{ij})-\hsig_{ij}|I_{(|\hsig_{ij}|>C\omega_T\sqrt{\hat{\theta}_{ij}})}+\sum_{j=1}^p|\hsig_{ij}-\sigma_{u,ij}|I_{(|\hsig_{ij}|>C\omega_T\sqrt{\hat{\theta}_{ij}})}\cr
&&+\sum_{j=1}^p|\sigma_{u,ij}|I_{(|\hsig_{ij}|\leq C\omega_T\sqrt{\hat{\theta}_{ij}})}\cr
&\leq& \max_{i\leq p}\sum_{j=1}^pC\omega_T\sqrt{\hat{\theta}_{ij}}I_{(|\hsig_{ij}|>C\omega_T\theta_2)}+M\omega_T\sum_{j=1}^pI_{(|\hsig_{ij}|>C\omega_T\theta_2)}+\sum_{j=1}^p|\sigma_{u,ij}|I_{(|\hsig_{ij}|\leq C\omega_T\theta_1)}\cr
&\leq& (C\theta_1+M)\omega_T\max_{i\leq p}\sum_{j=1}^pI_{(|\sigma_{u,ij}|>M\omega_T)}+\max_{i\leq p}\sum_{j=1}^p|\sigma_{u,ij}|I_{(|\sigma_{u,ij}|\leq C\omega_T\theta_1+M\omega_T)}\cr
&\leq& (C\theta_1+M)\omega_T\max_{i\leq p}\sum_{j=1}^p\frac{|\sigma_{u,ij}|^q}{M^q\omega_T^q}I_{(|\sigma_{u,ij}|>M\omega_T)}\cr
&&+\max_{i\leq p}\sum_{j=1}^p|\sigma_{u,ij}|\frac{(C\theta_1+M)^{1-q}\omega_T^{1-q}}{|\sigma_{u,ij}|^{1-q}}I_{(|\sigma_{u,ij}|\leq (C\theta_1+M)\omega_T)}\cr
&\leq& \frac{C\theta_1+M}{M^q}\omega_T^{1-q}\max_{i\leq p}\sum_{j=1}^p|\sigma_{u,ij}|^q+\max_{i\leq p}\sum_{j=1}^p|\sigma_{u,ij}|^q(C\theta_1+M)^{1-q}\omega_T^{1-q}\cr
&=&m_p\omega_T^{1-q}(C\theta_1+M)(M^{-q}+(C\theta_1+M)^{-q}).
\end{eqnarray*}
Let $M_1=(C\theta_1+M)(M^{-q}+(C\theta_1+M)^{-q}).$ Then with probability at least $1-2\epsilon$, $\|\hSig_u^{\mathcal{T}}-\Sig_u\|\leq m_p\omega_T^{1-q}M_1.$ Since $\epsilon$ is arbitrary, we have $\|\hSig_u^{\mathcal{T}}-\Sig_u\|=O_p(\omega_T^{1-q}m_p)$. If in addition,  $\omega_Tm_p=o(1)$, then the minimum eigenvalue of $\hSig_u^{\mathcal{T}}$ is bounded away from zero with probability approaching one since $\lambda_{\min}(\Sig_u)>c_1$. This then implies $\|(\hSig_u^{\mathcal{T}})^{-1}-\Sig_u^{-1}\|=O_p\left(\omega_T^{1-q}m_p\right).$

\end{proof}

\section{Proofs for Section 2}

We first  cite two useful theorems, which are needed to prove propositions  2.1 and 2.2. In Lemma \ref{lb.1} below,  let $\{\lambda_i\}_{i=1}^p$ be the eigenvalues of $\bSigma$ in descending order and $\{\bxi_i\}_{i=1}^p$ be their associated eigenvectors.  Correspondingly, let $\{\hlam_i\}_{i=1}^p$ be the eigenvalues of $\hSig$ in descending order and $\{\hxi_i\}_{i=1}^p$ be their associated eigenvectors.

\begin{lem}\label{lb.1}
\begin{enumerate}
\item {\bf (Weyl's Theorem)} $|\hlam_i - \lambda_i| \leq \|\hSig - \bSigma\|$.
\item {\bf ($\sin \theta$ Theorem, Davis and Kahan, 1970)}
$$
\|\hxi_i -\bxi_i \| \leq \frac{\sqrt{2} \| \hSig - \bSigma \|}{ \min( |\hlam_{i-1} - \lambda_i|,  |\lambda_i - \hlam_{i+1} | )}.
$$
\end{enumerate}
\end{lem}

\textbf{Proof of Proposition 2.1}
\begin{proof}
Since $\{\lambda_j\}_{j=1}^p$ are the eigenvalue of $\Sig$ and $\{\|\tb_j\|^2\}_{j=1}^K$ are the first $K$ eigenvalues of $\bB\bB'$ (the remaining $p-K$ eigenvalues are zero), then by the Weyl's theorem, for each $j\leq K$,
$$
|\lambda_j-\|\tb_j\|^2|\leq\|\Sig-\bB\bB'\|=\|\Sig_u\|.
$$
For $j> K$, $|\lambda_j|=|\lambda_j-0|\leq \|\Sig_u\|.$ On the other hand, the first $K$ eigenvalues of $\bB\bB$ are also the eigenvalues  of $\bB'\bB$. By the assumption, the eigenvalues of $p^{-1}\bB'\bB$ are bounded away from zero. Thus when $j\leq K$, $\|\tb_j\|^2/p$ are bounded away from zero for all large $p.$
\end{proof}

\textbf{Proof of Proposition 2.2}
\begin{proof}
Applying the $\sin\theta$ theorem yields
$$
\|\xi_j-\tb_j/\|\tb_j\|\|\leq\frac{\sqrt{2}\|\Sig_u\|}{\min(|\lambda_{j-1}-\|\tb_j\|^2|, |\|\tb_j\|^2-\lambda_{j+1}|)}
$$
For a generic constant $c>0$, $|\lambda_{j-1}-\|\tb_j\|^2|\geq |\|\tb_{j-1}\|^2-\|\tb_j\|^2|-|\lambda_{j-1}-\|\tb_{j-1}\|^2|\geq cp$ for all large $p$, since $|\|\tb_{j-1}\|^2-\|\tb_j\|^2|\geq cp$ but $|\lambda_{j-1}-\|\tb_{j-1}\|^2|$ is bounded  by Prosposition 2.1. On the other hand,  if $j<K$, the same argument implies
$ |\|\tb_j\|^2-\lambda_{j+1}|\geq cp$. If $j=K$, $|\|\tb_j\|^2-\lambda_{j+1}|=p|\|\tb_K\|^2/p-\lambda_{K+1}/p|$, where $\|\tb_K\|^2/p$ is bounded away from zero, but $\lambda_{K+1}/p=O(p^{-1})$. Hence again, $|\|\tb_j\|^2-\lambda_{j+1}|\geq cp.$

 \end{proof}

\textbf{Proof of Theorem \ref{thm2.1}}

\begin{proof} The sample covariance matrix of the residuals using least squares method is given by $\hSig_u\frac{1}{T}(\bY-\hLam\hF')(\bY'-\hF\hLam')=\frac{1}{T}\bY\bY'-\hLam \hLam'.$
where we used the normalization condition $\hF'\hF=T\bI_K$ and $\hLam=\bY\hF/T.$ If we show that $\hLam\hLam'=\sum_{i=1}^K\hlam_i\hxi_i\hxi_i'$, then from the decompositions of the sample covariance
\begin{equation*}
\frac{1}{T}\bY\bY' = \hLam \hLam'+\hSig_u =\sum_{i=1}^K\hlam_i\hxi_i\hxi_i'+\hR,
\end{equation*}
we have $\hR=\hSig_u$.
Consequently, applying thresholding on $\hSig_u$ is equivalent to applying thresholding on $\hR$, which gives the desired result.

We now show $\hLam \hLam'=\sum_{i=1}^K\hlam_i\hxi_i\hxi_i'$ indeed holds.
Consider again the least squares problem (\ref{eq2.10}) but with the following alternative normalization constraints:
$
\frac{1}{p}\sum_{i=1}^p\bb_i\bb_i'=\bI_K, $ and $\frac{1}{T}\sum_{t=1}^T\bff_t\bff_t'$ is diagonal.
Let $(\widetilde{\bLambda}, \widetilde{\bF})$ be the solution to the new optimization problem. Switching the roles of $\bB$ and $\bF$, then the solution of (\ref{eq2.12}) is $\widetilde{\bLambda} = (\hxi_1, \cdots, \hxi_K)$ and $\widetilde{\bF}=p^{-1}\bY'\widetilde{\bLambda}$. In addition,
$
    T^{-1}\widetilde{\bF}'\widetilde{\bF} = \diag(\widehat \lambda_1, \cdots, \widehat \lambda_K).
$
From $\hLam \hF' = \widetilde{\bLambda} \widetilde{\bF}'$, it follows that
$
\hLam \hLam'=\frac{1}{T}\hLam \hF'\hF\hLam' =\frac{1}{T}\widetilde{\bLambda}\widetilde{\bF}'\widetilde{\bF}\widetilde{\bLambda}'=\sum_{i=1}^K\hlam_i\hxi_i\hxi_i'.
$
\end{proof}

\section{Proofs for Section 3}

We will proceed by subsequently showing Theorems \ref{thm33}, \ref{thm31} and \ref{thm32}.

\subsection{Preliminary lemmas}
The following results are to be used subsequently. The proofs of Lemmas \ref{la1},\ref{la2} and \ref{lb4} are found in Fan, Liao and Mincheva (2011).

\begin{lem} \label{la1} Suppose $\bA, \bB$ are symmetric semi-positive definite matrices, and $\lambda_{\min}(\bB)>c_T$ for a sequence $c_T>0.$ If $\|\bA-\bB\|=o_p(c_T)$, then $\lambda_{\min}(\bA)>c_T/2$, and
$$
\|\bA^{-1}-\bB^{-1}\|=O_p(c_T^{-2})\|\bA-\bB\|.
$$
\end{lem}

\begin{lem} \label{la2} Suppose that the random variables $Z_1, Z_2$ both satisfy the exponential-type tail condition: There exist $r_1$, $r_2\in(0,1)$ and $b_1, b_2>0$, such that $\forall s>0$,
$$P(|Z_i|>s)\leq \exp(-(s/b_i)^{r_i}), \hspace{1em} i=1,2.$$
Then for some $r_3$ and $b_3>0$, and any $s>0$,
\begin{equation} \label{ea0}
P(|Z_1Z_2|>s)\leq \exp(1-(s/b_3)^{r_3}).
\end{equation}
\end{lem}

\begin{lem}\label{lb4} Under the assumptions of Theorem \ref{thm31},\\
(i) $\max_{i,j\leq K}|\frac{1}{T}\sum_{t=1}^Tf_{it}f_{jt}-Ef_{it}f_{jt}|=O_p(\sqrt{1/T})$.\\
(ii) $\max_{i,j\leq p}|\frac{1}{T}\sum_{t=1}^Tu_{it}u_{jt}-Eu_{it}u_{jt}|=O_p(\sqrt{(\log p)/T})$\\
(iii) $\max_{i\leq K, j\leq p}|\frac{1}{T}\sum_{t=1}^Tf_{it}u_{jt}|=O_p(\sqrt{(\log p)/T})$
\end{lem}

\begin{lem}\label{lb.5}
Let $\hlam_K$ denote the $K$th largest eigenvalue of $\hSig_{\sam}=\frac{1}{T}\sum_{t=1}^T\by_t\by_t'$, then $\hlam_K>C_1p$ with probability approaching one for some $C_1>0.$
\end{lem}

\begin{proof}

First of all, by Proposition~\ref{prop21}, under Assumption \ref{a35},
the $K{th}$ largest eigenvalue $\lambda_K$ of $\Sig$ satisfies: for some $c>0,$
$$
\lambda_K\geq\|\tb_K\|^2-|\lambda_K-\|\tb_K\|^2|\geq cp\|\Sig_u\|\geq cp /2
$$
 for sufficiently large $p$.  Using Weyl's theorem, we need only to prove that $\| \hSig_{\sam} - \Sig \| = o_p(p)$.  Without loss of generality, we prove the result under the identifiability condition (\ref{eq2.7}).  Using model (\ref{eq1.2}),
$
   \hSig_{\sam} = T^{-1} \sum_{t=1}^T (\bB \bff_t + \bu_t) (\bB \bff_t + \bu_t)'.
$
Using this and (\ref{eq1.3}), $\hSig_{\sam} - \Sig$ can be decomposed as the sum of the four terms:
\begin{eqnarray*}
  \bD_1  & = &  (T^{-1} \bB \sum_{t=1}^T \bff_t \bff_t' - \bI_K) \bB', \qquad \bD_2   =  T^{-1} \sum_{t=1}^T (\bu_t \bu_t' - \Sig_u),\\
  \bD_3  & = & \bB T^{-1} \sum_{t=1}^T \bff_t \bu_t', \qquad \bD_4 = \bD_3'
\end{eqnarray*}
We now deal them term by term.  We will repeatedly use the fact that for a $p\times p$ matrix $\bA$,
$$
    \| \bA \| \leq p  \| \bA \|_{\max}.
$$
First of all, by Lemma  \ref{lb4}, $
  \|T^{-1}\sum_{t=1}^T \bff_t \bff_t' - \bI_K\|  \leq K \|T^{-1}\sum_{t=1}^T \bff_t \bff_t' - \bI_K\|_{\max} = O_p(\sqrt{1/T}),$
which is $o_p(p)$ if $K\log p=o(T)$.
Consequently, by Assumption \ref{a35}, we have
$$
\|\bD_1\| \leq O_p(K\sqrt{(\log K)/T}) \|\bB \bB'\| = O_p( p\sqrt{1/T}).
$$
We now deal with $\bD_2$.  It follows from Lemma  \ref{lb4} that
$$
   \|\bD_2\| \leq p \| T^{-1} \sum_{t=1}^T (\bu_t \bu_t' - \Sig_u)\|_{\max} = O_p(p\sqrt{(\log p)/T}).
$$
Since $\|\bD_4\| = \|\bD_3\|$, it remains to deal with $\bD_3$, which is bounded by
$$
     \|\bD_3\|  \leq \| T^{-1} \sum_{t=1}^T   \bff_t \bu_t'\|\|\bB\|=O_p(p\sqrt{(\log p)/T}),
$$
which is $o_p(p)$ since $\log p=o(T)$.

\end{proof}

\begin{lem} \label{lc5add}Under Assumption \ref{a32},
$
\max_{t\leq T}\sum_{s=1}^T|E\bu_s'\bu_t|/p=O(1).
$
\end{lem}
\begin{proof} Since $\{\bu_t\}_{t=1}^T$ is weakly stationary, $\max_{t\leq T}\sum_{s=1}^T|E\bu_s'\bu_t|/p\leq2\sum_{t=1}^{\infty}|E\bu_1'\bu_t|/p.$ In addition, $E|u_{it}|^4<M$ for some constant $M$ and any $i, t$ since $u_{it}$ has exponential tail. Hence by Davydov's inequality (Corollary 16.2.4 in Athreya and Lahiri 2006), there is   a constant $C>0$, for  all $i\leq p, t\leq T$,  $|Eu_{i1}u_{it}|\leq C\sqrt{\alpha(t)}$, where $\alpha(t)$ is the $\alpha$-mixing coefficient.  By Assumption \ref{a32},
$\sum_{t=1}^{\infty}\sqrt{\alpha(t)}<\infty.$ Thus uniformly in $T$,
\begin{eqnarray*}
 \max_{t\leq T}\sum_{s=1}^T|E\bu_s'\bu_t|/p&\leq&2\sum_{t=1}^{\infty}|E\bu_1'\bu_t|/p\leq 2\sum_{t=1}^{\infty}\max_{i\leq p}|Eu_{i1}u_{it}|\leq2C\sum_{t=1}^{\infty}\sqrt{\alpha(t)}<\infty.
\end{eqnarray*}
\end{proof}

\subsection{Proof of Theorem \ref{thm33}}

Our derivation below relies on a result obtained by Bai and Ng (2002), which showed that the estimated number of factors is consistent, in the sense that $\widehat{K}$ equals the true $K $ with probability approaching one.  Note that   under our Assumptions \ref{a35}-\ref{a33}, all the assumptions in Bai and Ng (2002) are satisfied.  Thus immediately we have the following Lemma.

\begin{lem}[Theorem 2 in Bai and Ng (2002)] For $\widehat{K}$ defined in (\ref{eq2.16add}),
$$P(\widehat{K}=K)\rightarrow1.$$
\end{lem}
\begin{proof} See Bai and Ng (2002). \end{proof}

Using (A.1) in Bai (2003), we have the following identity:
\begin{equation}\label{eb1}
\hf_t-\bH\bff_t=(\bV/p)^{-1}\left(\frac{1}{T}\sum_{s=1}^T\hf_sE({\bu}_s'{\bu}_t)/p+\frac{1}{T}\sum_{s=1}^T\hf_s\zeta_{st}+\frac{1}{T}\sum_{s=1}^T\hf_s\eta_{st}+\frac{1}{T}\sum_{s=1}^T\hf_s\xi_{st}\right)
\end{equation}
where $\zeta_{st}={\bu}_s'{\bu}_t/p-E({\bu}_s'{\bu}_t)/p$, $\eta_{st}=\bff_s'\sum_{i=1}^p{\bb}_iu_{it}/p$, and
$\xi_{st}=\bff_t'\sum_{i=1}^p{\bb}_iu_{is}/p$.

We first prove some preliminary results in the following Lemmas. Denote by $\hf_t=(\hat{f}_{1t},...,\hat{f}_{\widehat{K}t})'.$

\begin{lem} \label{lb1}For all $i\leq \widehat{K}$,\\
(i) $\frac{1}{T}\sum_{t=1}^T(\frac{1}{T}\sum_{s=1}^T\hat{f}_{is} E(\bu_s'\bu_t)/p)^2=O_p(T^{-1})$,\\
(ii) $\frac{1}{T}\sum_{t=1}^T(\frac{1}{T}\sum_{s=1}^T\hat{f}_{is} \zeta_{st})^2=O_p(p^{-1})$,\\
(iii) $\frac{1}{T}\sum_{t=1}^T(\frac{1}{T}\sum_{s=1}^T\hat{f}_{is} \eta_{st})^2=O_p(p^{-1})$,\\
(iv)  $\frac{1}{T}\sum_{t=1}^T(\frac{1}{T}\sum_{s=1}^T\hat{f}_{is} \xi_{st})^2=O_p(p^{-1})$.
\end{lem}
\begin{proof}
(i) We have $\forall i$, $\sum_{s=1}^T\hat{f}_{is} ^2=T$. By the Cauchy-Schwarz inequality,
\begin{eqnarray*}
&& \frac{1}{T}\sum_{t=1}^T(\frac{1}{T}\sum_{s=1}^T\hat{f}_{is} E(\bu_s'\bu_t)/p)^2
\leq\frac{1}{T}\sum_{t=1}^T\frac{1}{T}\sum_{s=1}^T(E\bu_s'\bu_t/p)^2\cr
&\leq&\max_{t\leq T}\frac{1}{T}\sum_{s=1}^T(E\bu_s'\bu_t/p)^2\leq\max_{s,t}|E\bu_{s}'\bu_{t}/p|\max_{t\leq T}\frac{1}{T}\sum_{s=1}^T|E\bu_s'\bu_t/p|
\end{eqnarray*}
By Lemma \ref{lc5add},  $\max_{t\leq T}\sum_{s=1}^T|E\bu_s'\bu_t/p|=O(1)$, which then yields the result.

(ii) By the Cauchy-Schwarz inequality,
\begin{eqnarray*}
&& \frac{1}{T}\sum_{t=1}^T(\frac{1}{T}\sum_{s=1}^T\hat{f}_{is} \zeta_{st})^2= \frac{1}{T^3}\sum_{s=1}^T\sum_{l=1}^T\hat{f}_{is} \hat{f}_{il}(\sum_{t=1}^T\zeta_{st}\zeta_{lt})\leq\frac{1}{T^3}\left(\sum_{sl}(\hat{f}_{is} \hat{f}_{il})^2\sum_{sl}(\sum_{t=1}^T\zeta_{st}\zeta_{lt})^2\right)^{1/2}\cr
&\leq& \frac{1}{T^3}\sum_{s=1}^T\hat{f}_{is} ^2\left(\sum_{sl}(\sum_{t=1}^T\zeta_{st}\zeta_{lt})^2\right)^{1/2}=\frac{1}{T^2}\left(\sum_{s=1}^T\sum_{l=1}^T(\sum_{t=1}^T\zeta_{st}\zeta_{lt})^2\right)^{1/2}.
\end{eqnarray*}
Note that $E(\sum_{s=1}^T\sum_{l=1}^T(\sum_{t=1}^T\zeta_{st}\zeta_{lt})^2)=T^2E(\sum_{t=1}^T\zeta_{st}\zeta_{lt})^2\leq T^4\max_{st}E|\zeta_{st}|^4.$ By Assumption \ref{a33}, $\max_{st}E\zeta_{st}^4=O(p^{-2})$,  which implies that $\sum_{s,l}(\sum_{t=1}^T\zeta_{st}\zeta_{lt})^2=O_p(T^4/p^2)$, and yields the result.

(iii) By definition, $\eta_{st}=\bff_s'\sum_{i=1}^p\bb_iu_{it}/p$. We first bound $\|\sum_{i=1}^p\bb_iu_{it}\|$. Assumption \ref{a33}  implies
$
E\frac{1}{T}\sum_{t=1}^T\|\sum_{i=1}^p\bb_iu_{it}\|^2=E\|\sum_{i=1}^p\bb_iu_{it}\|^2=O(p ).
$
Therefore, by the Cauchy-Schwarz inequality,
\begin{eqnarray*}
&&\frac{1}{T}\sum_{t=1}^T(\frac{1}{T}\sum_{s=1}^T\hat{f}_{is} \eta_{st})^2\leq \|\frac{1}{T}\sum_{s=1}^T\hat{f}_{is} \bff_s'\|^2\frac{1}{T}\sum_{t=1}^T \|\sum_{j=1}^p\bb_ju_{jt}\frac{1}{p} \|^2\cr
&\leq& \frac{1}{Tp^2}\sum_{t=1}^T\|\sum_{j=1}^p\bb_ju_{jt}\|^2\left(\frac{1}{T}\sum_{s=1}^T\hat{f}_{is} ^2\frac{1}{T}\sum_{s=1}^T\|\bff_s\|^2\right)=O_p\left(\frac{1}{p}\right).
\end{eqnarray*}

(iv) Similar to part (iii),  noting that $\xi_{st}$ is a scalar, we have:
\begin{eqnarray*}
&&\frac{1}{T}\sum_{t=1}^T(\frac{1}{T}\sum_{s=1}^T\hat{f}_{is} \xi_{st})^2= \frac{1}{T}\sum_{t=1}^T\bigg{|}\frac{1}{T}\sum_{s=1}^T\bff_t'\sum_{j=1}^p\bb_ju_{js}\frac{1}{p}\hat{f}_{is} \bigg{|}^2\leq \frac{1}{T}\sum_{t=1}^T\|\bff_t\|^2\cdot\norm\frac{1}{T}\sum_{s=1}^T\sum_{j=1}^p\bb_ju_{js}\frac{1}{p}\hat{f}_{is} \norm^2\cr
&\leq& O_p(1) \frac{1}{T}\sum_{s=1}^T\norm\sum_{j=1}^p\bb_ju_{js}\frac{1}{p}\norm^2\cdot\frac{1}{T}\sum_{s=1}^T\hat{f}_{is} ^2\leq O_p\left(\frac{1}{p}\right),
\end{eqnarray*}
where the third line follows from the Cauchy-Schwarz inequality.
\end{proof}

\begin{lem}\label{lb2}   (i) $\max_{t\leq T}\|\frac{1}{Tp}\sum_{s=1}^T\hf_sE(\bu_s'\bu_t)\|=O_p(\sqrt{1/T})$,\\
(ii) $\max_{t\leq T}\|\frac{1}{T}\sum_{s=1}^T\hf_s\zeta_{st}\|=O_p(T^{1/4}/\sqrt{p})$,\\
(iii) $\max_{t\leq T}\|\frac{1}{T}\sum_{s=1}^T\hf_s\eta_{st}\|=O_p(T^{1/4}/\sqrt{p})$,\\
(iv)  $\max_{t\leq T}\|\frac{1}{T}\sum_{s=1}^T\hf_s\xi_{st}\|=O_p(T^{1/4}/\sqrt{p})$.
\end{lem}
\begin{proof} (i) By the Cauchy-Schwarz inequality and the fact that $\frac{1}{T}\sum_{t=1}^T\|\hf_t\|^2=O_p(1),$
\begin{eqnarray*}
&&\max_{t\leq T}\|\frac{1}{Tp}\sum_{s=1}^T\hf_sE(\bu_s'\bu_t)\|\leq\max_{t\leq T}\left(\frac{1}{T}\sum_{s=1}^T\|\hf_s\|^2\frac{1}{T}\sum_{s=1}^T(E\bu_s'\bu_t/p)^2\right)^{1/2}\cr
&\leq&O_p(1)\max_{t\leq T}\left(\frac{1}{T}\sum_{s=1}^T(E\bu_s'\bu_t/p)^2\right)^{1/2}\leq O_p(1)\max_{s,t}\sqrt{|E\bu_s'\bu_t/p|}\max_{t\leq T}\left(\frac{1}{T}\sum_{s=1}^T|E\bu_s'\bu_t/p|\right)^{1/2}.
\end{eqnarray*}
The result then follows from Assumption \ref{a32}.

(ii) By the Cauchy-Schwarz inequality,
$$\max_{t\leq T}\|\frac{1}{T}\sum_{s=1}^T\hf_s\zeta_{st}\|\leq\max_{t\leq T}\frac{1}{T}\left(\sum_{s=1}^T\|\hf_s\|^2\sum_{s=1}^T\zeta_{st}^2\right)^{1/2}\leq\left(O_p(1)\max_t\frac{1}{T}\sum_{s=1}^T\zeta_{st}^2\right)^{1/2}.$$
It follows from Assumption \ref{a33} that $E(\frac{1}{T}\sum_{s=1}^T\zeta_{st}^2)^2\leq\max_{s,t\leq T}E\zeta_{st}^4=O(\frac{1}{p^2}).$
It then follows from the Chebyshev's inequality and Bonferroni's method that $\max_t\frac{1}{T}\sum_{s=1}^T\zeta_{st}^2=O_p(\sqrt{T}/p)$.

(iii) By Assumption \ref{a33}, $E\|\frac{1}{\sqrt{p}}\sum_{i=1}^p\bb_iu_{it}\|^4\leq K^2M$.   Chebyshev's inequality and Bonferroni's method yield $\max_{t\leq T}\|\sum_{i=1}^p\bb_iu_{it}\|=O_p(T^{1/4}\sqrt{p})$ with probability one, which then implies:
$\max_{t\leq T}\|\frac{1}{T}\sum_{s=1}^T\hf_s\eta_{st}\|\leq\|\frac{1}{T}\sum_{s=1}^T\hf_s\bff_s'\|\max_{t}\|\frac{1}{p}\sum_{i=1}^p\bb_iu_{it}\|
=o_p(T^{1/4}/p^{1/2}).$

(iv) By the Cauchy-Schwarz inequality and Assumption \ref{a33}, we have demonstrated that
$
\|\frac{1}{T}\sum_{s=1}^T\sum_{i=1}^p\bb_iu_{is}\frac{1}{p}\hf_s\|=O_p(p^{-1/2}).
$
In addition, since $E\|K^{-2}\bff_t\|^4<M$,  $\max_{t\leq T}\|\bff_t\|=O_p(T^{1/4})$. It  follows  that
 $\max_{t\leq T}\|\frac{1}{T}\sum_{s=1}^T\hf_s\xi_{st}\|\leq \max_{t\leq T}\|\bff_t\|\cdot\|\frac{1}{T}\sum_{s=1}^T\sum_{i=1}^p\bb_iu_{is}\frac{1}{p}\hf_s\|=O_p(T^{1/4}/p^{1/2}).$
\end{proof}

\begin{lem}\label{lb3} (i) $\max_{i\leq K}\frac{1}{T}\sum_{t=1}^T(\hf_t-\bH\bff_t)_i^2=O_p(1/T+1/p)$.\\
(ii) $\frac{1}{T}\sum_{t=1}^T\|\hf_t-\bH\bff_t\|^2=O_p(1/T+1/p)$.\\
(iii) $\max_{t\leq T}\|\hf_t-\bH\bff_t\|=O_p(\sqrt{1/T}+T^{1/4}/\sqrt{p})$.
\end{lem}

\begin{proof}  We  prove this lemma conditioning on the event $\widehat{K}=K.$ Once this is done, due to $P(\widehat{K}\neq K)=o(1)$, it  then implies the unconditional arguments.

(i) When $\widehat{K}=K$, by Lemma \ref{lb.5}, all the eigenvalues of $\bV/p$ are bounded away from zero.
Using the inequality $(a+b+c+d)^2\leq 4(a^2+b^2+c^2+d^2)$ and the identity   (\ref{eb1}),
 we have,  for some constant $C>0$,
\begin{eqnarray*}
\max_{i\leq K}\frac{1}{T}\sum_{t=1}^T(\hf_t-\bH\bff_t)_i^2&\leq& C\max_{i\leq K}\frac{1}{T}\sum_{t=1}^T(\frac{1}{T}\sum_{s=1}^T\hat{f}_{is} E(\bu_s'\bu_t)/p)^2+C\max_{i\leq K}\frac{1}{T}\sum_{t=1}^T(\frac{1}{T}\sum_{s=1}^T\hat{f}_{is} \zeta_{st})^2\cr
&&+C\max_{i\leq K}\frac{1}{T}\sum_{t=1}^T(\frac{1}{T}\sum_{s=1}^T\hat{f}_{is} \eta_{st})^2+C\max_{i\leq K}\frac{1}{T}\sum_{t=1}^T(\frac{1}{T}\sum_{s=1}^T\hat{f}_{is} \xi_{st})^2.
\end{eqnarray*}
Each of the four terms on the right hand side above are bounded in Lemma \ref{lb1}, which then yields the desired result.

(ii) It follows  from part (i) and that $
\frac{1}{T}\sum_{t=1}^T\|\hf_t-\bH\bff_t\|^2\leq K\max_{i\leq K}\frac{1}{T}\sum_{t=1}^T(\hf_t-\bH\bff_t)_i^2.$

Part (iii) is implied by (\ref{eb1}) and  Lemma \ref{lb2}. \end{proof}

\begin{lem}\label{lb5} (i)  $\bH\bH'=\bI_{\widehat{K}}+O_p(1/\sqrt{T}+1/\sqrt{p})$.\\
 (ii) $\bH'\bH=\bI_K+O_p(1/\sqrt{T}+1/\sqrt{p})$.
\end{lem}
\begin{proof} We first condition on $\widehat{K}=K$. (i)    Lemma \ref{lb.5} implies   $\|\bV^{-1}\|=O_p(p^{-1})$. Also $\|\bF \|=\lambda_{\max}^{1/2}(\bF \bF')=\lambda_{\max}^{1/2}(\sum_{t=1}^T\bff_t\bff_t')=O_p(\sqrt{T}).$
In addition, $\|\hF\|=\sqrt{T}$. It then follows from the definition of $\bH$ that $\|\bH\|=O_p(1)$.  Define $\hcov(\bH\bff_t)=\frac{1}{T}\sum_{t=1}^T\bH{\bff_t}(\bH{\bff_t})'.$ Applying the triangular inequality gives:
\begin{eqnarray} \label{ec.6}
\|\bH\bH'-\bI_{\widehat{K}}\|_F&\leq&\|\bH\bH'-\hcov(\bH{\bff_t})\|_F+\|\hcov(\bH{\bff_t})-\bI_{\widehat{K}}\|_F
\end{eqnarray}
By Lemma \ref{lb4}, the first term in (\ref{ec.6}) is
$
\|\bH\bH'-\hcov(\bH{\bff_t})\|_F\leq \|\bH\|^2\|\bI_K-\hcov({\bff_t})\|_F=O_p\left( \frac{1}{\sqrt{T}}\right).
$
The second term of (\ref{ec.6}) can be bounded, by the Cauchy-Schwarz inequality and Lemma \ref{lb3}, as follows:
\begin{eqnarray*}
 && \norm\frac{1}{T}\sum_{t=1}^T\bH{\bff_t}(\bH{\bff_t})'-\frac{1}{T}\sum_{t=1}^T\hf_t\hf_t'\norm_F
\leq\norm\frac{1}{T}\sum_t(\bH{\bff_t}-\hf_t)(\bH{\bff_t})'\norm_F+\norm\frac{1}{T}\sum_t\hf_t(\hf_t'-(\bH{\bff_t})')\norm_F\cr
 &\leq& \left(\frac{1}{T}\sum_t\|\bH{\bff_t}-\hf_t\|^2\frac{1}{T}\sum_t\|\bH{\bff_t}\|^2\right)^{1/2}+\left(\frac{1}{T}\sum_t\|\bH{\bff_t}-\hf_t\|^2\frac{1}{T}\sum_t\|\hf_t\|^2\right)^{1/2}\cr
 &=&O_p\left(\frac{1}{\sqrt{T}}+\frac{1}{\sqrt{p}}\right).
\end{eqnarray*}
(ii) Still conditioning on $\widehat{K}=K$, since $\bH\bH'=\bI_{K}+O_p(1/\sqrt{T}+1/\sqrt{p})$ and $\|\bH\|=O_p(1)$, right multiplying $\bH$   gives $\bH\bH'\bH=\bH+O_p(1/\sqrt{T}+1/\sqrt{p})$.  Part (i) also gives, conditioning on  $\widehat{K}=K$, $\|\bH^{-1}\|=O_p(1)$. Hence further left multiplying $\bH^{-1}$ yields $\bH'\bH=\bI_{K}+O_p(1/\sqrt{T}+\sqrt{p})$. Due to $P(\widehat{K}=K)\rightarrow1$, we reach the desired result.

\end{proof}

\textbf{Proof of Theorem \ref{thm33}}

\begin{proof}
The second part of this theorem was proved in Lemma \ref{lb3}. We now derive the convergence rate of $\max_{i\leq p}\|\hb_i-\bH\bb_i\|$.

Using the facts that $\hb_i=\frac{1}{T}\sum_{t=1}^Ty_{it}\hf_t$, and that $\frac{1}{T}\sum_{t=1}^T\hf_t\hf_t'=I_k$,  we have
\begin{equation}\label{eb2}
\hb_i-\bH{\bb}_i=\frac{1}{T}\sum_{t=1}^T\bH\bff_tu_{it}+\frac{1}{T}\sum_{t=1}^Ty_{it}(\hf_t-\bH\bff_t)+\bH(\frac{1}{T}\sum_{t=1}^T\bff_t\bff_t'-\bI_K)\bb_i.
\end{equation}
We bound the three terms on the right hand side respectively. It follows from Lemmas \ref{lb4} and \ref{lb5} that
$
\max_{i\leq p}\|\frac{1}{T}\sum_{t=1}^T\bH\bff_tu_{it}\|\leq \|\bH\|\max_{i}\sqrt{\sum_{k=1}^K(\frac{1}{T}\sum_{t=1}^Tf_{kt}u_{it})^2}=O_p\left(\sqrt{\frac{\log p}{T}}\right).
$
For the second term, $Ey_{it}^2=O(1)$. Therefore,    $\max_iT^{-1}\sum_{t=1}^Ty_{it}^2=O_p(1)$. The Cauchy-Schwarz inequality and Lemma \ref{lb3} imply $$\max_i\|\frac{1}{T}\sum_{t=1}^Ty_{it}(\hf_t-\bH\bff_t)\|\leq\max_i\left(\frac{1}{T}\sum_{t=1}^Ty_{it}^2\frac{1}{T}\sum_{t=1}^T\|\hf_t-\bH\bff_t\|^2\right)^{1/2}=O_p(\frac{1}{\sqrt{T}}+\frac{1}{\sqrt{p}}).$$
Finally, $\|\frac{1}{T}\sum_{t=1}^T\bff_t\bff_t'-\bI_K\|=O_p(T^{-1/2})$ and $\max_i\|\bb_i\|=O(1)$ imply that the third term is $O_p(T^{-1/2}).$

\end{proof}

\textbf{Proof of Corollary \ref{c31}}

Under Assumption \ref{a32}, it can be shown by Bonferroni's method that
$
\max_{t\leq T}\|\bff_t\|=O_p((\log T)^{1/r_2}).
$
By Theorem \ref{thm33},  uniformly in $i$ and $t$,
\begin{eqnarray*}
\|\hb_i'\hf_t-\bb_i'\bff_t\|&\leq&\|\hb_i-\bH\bb_i\|\|\hf_t-\bH\bff_t\|+\|\bH\bb_i\|\|\hf_t-\bH\bff_t\|\cr
&&+\|\hb_i-\bH\bb_i\|\|\bH\bff_t\|+\|\bb_i\|\bff_t\|\|\bH'\bH-\bI_K\|\cr
&=&O_p\left((\log T)^{1/r_2}\sqrt{\frac{\log p}{T}}+\frac{T^{1/4}}{\sqrt{p}}\right).
\end{eqnarray*}

\subsection{Proof of Theorem \ref{thm31}}

\begin{lem}\label{lb6}
$\max_{i\leq p}\frac{1}{T}\sum_{t=1}^T|u_{it}-\hat{u}_{it}|^2=O_p\left(\omega_T^2\right),$ and $\max_{i,t}|u_{it}-\hat{u}_{it}|=o_p(1).$
\end{lem}
\begin{proof} We have,
$
u_{it}-\hat{u}_{it}={\bb}_i'\bH'(\hf_t-\bH\bff_t)+(\hb_i'-\bb_i'\bH')\hf_t+\bb_i'(\bH'\bH-\bI_K)\bff_t.
$
Therefore, using the inequality $(a+b+c)^2\leq 4a^2+4b^2+4c^2$, we have:
\begin{eqnarray*}
&&\max_{i\leq p}\frac{1}{T}\sum_{t=1}^T(u_{it}-\hat{u}_{it})^2\leq
4\max_i\|\bb_i'\bH'\|^2\frac{1}{T}\sum_{t=1}^T\|\hf_t-\bH\bff_t\|^2\cr
&&+4\max_i\|\hb_i'-\bb_i'\bH'\|^2\frac{1}{T}\sum_{t=1}^T\|\hf_t\|^2+4\max_i\|\bb_i\|^2\frac{1}{T}\sum_{t=1}^T\|\bff_t\|^2\|\bH'\bH-\bI_K\|_F^2,
\end{eqnarray*}
The first part of the lemma then follows from Theorem \ref{thm33} and Lemma \ref{lb3}. The second part follows from Corollary \ref{c31}.

\end{proof}

\textbf{Proof of Theorem \ref{thm31}}: The theorem follows immediately from Theorem \ref{tb1} and Lemma \ref{lb6}.

\subsection{Proof of Theorem \ref{thm32}}

Define
$${\bC_T}=\hLam -{\bB}\bH'.$$

\begin{lem} \label{lc12add}
(i)$\|\bC_T\|_F^2=O_p(\omega_T^2p)$,   and $\|\bC_T'\bC_T\|_{\Sigma}^2=O_p(\omega_T^4p)$.\\
(ii)  $\|\hSig_{u,\widehat{K}}^{\mathcal{T}}-\Sig_u\|^2_{\Sigma}=O_p(\omega_T^{2-2q}m_p^2).$\\
(iii) $\|\bB\bH'\bC_T'\|_{\Sigma}^2=O_p(\omega_T^2).$\\
(iv) $\|\bB(\bH'\bH-\bI_K)\bB'\|_{\Sigma}^2=O_p(p^{-2}+(pT)^{-1}).$
\end{lem}
\begin{proof}  (i) We have $
\|\bC_T\|_F^2\leq\max_{i\leq p}\|\hb_i-\bH\bb_i\|^2p=O_p(\omega_T^2p).$ Moreover, since all the eigenvalues of $\Sig$ are bounded away from zero, for any matrix $\bA$, $\|\bA\|_{\Sigma}^2=O_p( p^{-1})\|\bA\|_F^2$.    Hence
$
\|\bC_T'\bC_T\|_{\Sigma}^2=O_p(p^{-1}\|\bC_T\|_F^4)=O_p(p\omega_T^4).
$\\
(ii) By Theorem \ref{thm31},
$
\|\hSig_{u,\widehat{K}}^{\mathcal{T}}-\Sig_u\|^2_{\Sigma}=O_p(p^{-1}\|\hSig_{u,\widehat{K}}^{\mathcal{T}}-\Sig_u\|^2_F)=O_p(\|\hSig_{u,\widehat{K}}^{\mathcal{T}}-\Sig_u\|^2)=O_p(\omega_T^{2-2q}m_p^2).
$\\
(iii) The same argument of the proof of Theorem 2 in Fan, Fan and Lv (2008) implies that $\|{\bB}'\Sig^{-1}{\bB}\|=O(1)$. Thus, $\|{\bB}\bH'\bC_T'\|^2_{\Sigma}=p^{-1}\tr(\bH' \bC_T'\Sig^{-1}\bC_T \bH{\bB}'\Sig^{-1}{\bB})$ is upper bounded by
$p^{-1} \|\bH\|^2 \|{\bB}'\Sig^{-1}{\bB}\|\|\Sig^{-1}\|\|\bC_T\|_F^2=O_p(p^{-1}\|\bC_T\|_F^2)=O_p(\omega_T^2).
$\\
(iv) Again, by $\|{\bB}'\Sig^{-1}{\bB}\|=O(1)$, and Lemma  \ref{lb5},
\begin{eqnarray}
\|\bB(\bH'\bH-\bI_K)\bB'\|_{\Sigma}^2&=&p^{-1}\tr((\bH'\bH-\bI_K)\bB'\Sig^{-1}\bB(\bH'\bH-\bI_K)\bB'\Sig^{-1}\bB)\cr
&\leq& p^{-1}\|\bH'\bH-\bI_K\|_F^2\|\bB'\Sig^{-1}\bB\|^2=O_p(p^{-2}+(pT)^{-1}).
\end{eqnarray}
\end{proof}

\textbf{Proof of Theorem \ref{thm32} (i)}
\begin{proof}
By Lemma \ref{lc12add},
$
\|{\bB}(\bH'\bH-\bI_K)\bB'\|^2_{\Sigma}
+\|{\bB}\bH'{\bC_T}'\|^2_{\Sigma} +\|{\bC_T}{\bC_T}'\|^2_{\Sigma}=O_p(\omega_T^2+\frac{p\log^2 p}{T^2}).
$
Hence for a generic constant $C>0$,
\begin{eqnarray*}
\|\hSig_{\widehat{K}}-\Sig\|_{\Sigma}^2&\leq& C\|\hLam\hLam'-\bB\bB'\|_{\Sigma}^2+C\|\hSig_{u,\widehat{K}}^{\mathcal{T}}-\Sig_u\|_{\Sigma}^2\cr
&\leq&C[\|{\bB}(\bH'\bH-\bI_K)\bB'\|^2_{\Sigma}
+\|{\bB}\bH'{\bC_T}'\|^2_{\Sigma}  +\|{\bC_T}{\bC_T}'\|^2_{\Sigma}
+\|\hSig_{u,\widehat{K}}^{\mathcal{T}}-\Sig_u\|^2_{\Sigma}]\cr
&=&O_p(\omega_T^{2-2q}m_p^2+\frac{p\log^2 p}{T^2}).
\end{eqnarray*}
\end{proof}

\begin{lem}\label{lc.13}
 $\|\hLam'(\hSig_{u, \widehat{K}}^{\mathcal{T}})^{-1}\hLam -({\bB}\bH')'\Sig_u^{-1}{\bB}\bH'\|=O_p(p\omega_T^{1-q}m_p)$.
\end{lem}
\begin{proof} $\|\bC_T\|_F^2=O_p(\omega_T^2p).$ Hence
\begin{eqnarray} \label{eb15}
&&\|\hLam'(\hSig_{u, \widehat{K}}^{\mathcal{T}})^{-1}\hLam -({\bB}\bH')'\Sig_u^{-1}{\bB}\bH'\|\leq \|\bC_T'(\hSig_{u, \widehat{K}}^{\mathcal{T}})^{-1}\bC_T\|\cr
&+&2\|\bC_T'(\hSig_{u, \widehat{K}}^{\mathcal{T}})^{-1}\bB \bH'\|+\|\bB \bH'((\hSig_{u, \widehat{K}}^{\mathcal{T}})^{-1}-\Sig_u^{-1})\bB \bH'\|=O_p(p\omega_T^{1-q}m_p)
\end{eqnarray}
\end{proof}

\begin{lem} \label{lc.14} If $\omega_T^{1-q}m_p=o(1)$, then with probability approaching one, for some $c>0,$\\
(i) $\lambda_{\min}(\bI_K+(\bB \bH')'\Sig_u^{-1}{\bB}\bH')\geq cp$.\\
(ii) $\lambda_{\min}(\bI_K+\hLam'(\hSig_{u, \widehat{K}}^{\mathcal{T}})^{-1}\hLam )\geq cp.$\\
(iii)   $\lambda_{\min}(\bI_K+\bB'\Sig_u^{-1}{\bB})\geq cp$.\\
(iv)   $\lambda_{\min}((\bH\bH')^{-1}+\bB'\Sig_u^{-1}{\bB})\geq cp$.
\end{lem}
\begin{proof}

(i)  By Lemma \ref{lb5}, with probability approaching one,  $\lambda_{\min}(\bH\bH')$ is bounded away from zero. Hence,
  \begin{eqnarray*}
&&\lambda_{\min}(\bI_K+({\bB}\bH')'\Sig_u^{-1}{\bB}\bH')\geq \lambda_{\min}(({\bB}\bH')'\Sig_u^{-1}{\bB}\bH'))\cr
&\geq& \lambda_{\min}(\Sig_u^{-1})\lambda_{\min}(\bH{\bB}'{\bB}\bH')\geq\lambda_{\min}(\Sig_u^{-1})\lambda_{\min}({\bB}'{\bB})\lambda_{\min}(\bH\bH')\geq cp.
\end{eqnarray*}
(ii) The result follows from part (i) and Lemma \ref{lc.13}. Part (iii) and (iv) follow from a similar argument of part (i) and Lemma \ref{lb5}.

\end{proof}

\textbf{Proof of Theorem \ref{thm32}}:

\begin{proof}
We derive the rate for $\|\hSig^{-1}_{\widehat{K}}-\Sig^{-1}\|.$ Define
$$
\tilde{\Sig}=\bB\bH'\bH\bB'+\Sig_u.
$$
Note that $\hSig_{\widehat{K}}=\hLam\hLam'+\hSig_{u, \widehat{K}}^{\mathcal{T}}$ and $\Sig=\bB\bB'+\Sig_u.$ The triangular inequality gives
$$
\|\hSig^{-1}_{\widehat{K}}-\Sig^{-1}\|\leq\|\hSig^{-1}_{\widehat{K}}-\tilde{\Sig}^{-1}\|+\|\tilde{\Sig}^{-1}-\Sig^{-1}\|.
$$
Using the Sherman-Morrison-Woodbury formula, we have $
\|\hSig^{-1}_{\widehat{K}}-\tilde{\Sig}^{-1}\|\leq\sum_{i=1}^6L_i,$
where
\begin{eqnarray}
L_1&=&\|(\hSig_{u, \widehat{K}}^{\mathcal{T}})^{-1}-\Sig_u^{-1}\|\cr
L_2&=&\|((\hSig_{u, \widehat{K}}^{\mathcal{T}})^{-1}-\Sig_u^{-1})\hLam[\bI_K+\hLam'(\hSig_{u, \widehat{K}}^{\mathcal{T}})^{-1}\hLam ]^{-1}\hLam'(\hSig_{u, \widehat{K}}^{\mathcal{T}})^{-1}\|
\cr
L_3&=&\|((\hSig_{u, \widehat{K}}^{\mathcal{T}})^{-1}-\Sig_u^{-1})\hLam[\bI_K+\hLam'(\hSig_{u, \widehat{K}}^{\mathcal{T}})^{-1}\hLam]^{-1}\hLam'\Sig_u^{-1}\|
\cr
L_4&=&\|\Sig_u^{-1}(\hLam-{\bB}\bH')[\bI_K+\hLam'(\hSig_{u, \widehat{K}}^{\mathcal{T}})^{-1}\hLam ]^{-1}\hLam'\Sig_u^{-1}\|
\cr
L_5&=&\|\Sig_u^{-1}(\hLam-{\bB}\bH')[\bI_K+\hLam'(\hSig_{u, \widehat{K}}^{\mathcal{T}})^{-1}\hLam ]^{-1}\bH{\bB}'\Sig_u^{-1}\|
\cr
L_6&=&\|\Sig_u^{-1}{\bB}\bH'([\bI_K+\hLam'(\hSig_{u, \widehat{K}}^{\mathcal{T}})^{-1}\hLam]^{-1}-[\bI_K+\bH{\bB}'\Sig_u^{-1}{\bB}\bH']^{-1})\bH{\bB}'\Sig_u^{-1}\|.
\end{eqnarray}

We bound each of the six terms respectively. First of all, $L_1$ is bounded by Theorem \ref{thm31}. Let $\bG=[\bI_K+\hLam'(\hSig_{u, \widehat{K}}^{\mathcal{T}})^{-1}\hLam ]^{-1}$, then
$$
L_2\leq \|(\hSig_{u, \widehat{K}}^{\mathcal{T}})^{-1}-\Sig_u^{-1}\|\cdot\|\hLam \bG\hLam'\| \cdot\|(\hSig_{u, \widehat{K}}^{\mathcal{T}})^{-1}\|.
$$
Note that Theorem \ref{thm31} implies $\|(\hSig_{u, \widehat{K}}^{\mathcal{T}})^{-1}\|=O_p(1)$. Lemma \ref{lc.14} then implies
$
\|\bG\|=O_p(p^{-1}).
$ This shows that $L_2=O_p(L_1)$. Similarly $L_3=O_p(L_1)$. In addition,  since $\|\bC_T\|_F^2=\|\hLam -{\bB}\bH'\|_F^2=O_p(\omega_T^2p)$,
$L_4\leq
\|\Sig_u^{-1}(\hLam -{\bB}\bH')\|\|\bG\|\|{\hLam }'\Sig_u^{-1}\|=O_p(\omega_T).
$
Similarly  $L_5=O_p(L_4)$. Finally, let $\bG_1=[\bI_K+({\bB}\bH')'\Sig_u^{-1}{\bB}\bH']^{-1}.$
 By Lemma \ref{lc.14}, $\|\bG_1\|=O_p(p^{-1})$. Then by Lemma \ref{lc.13},
\begin{eqnarray*}
\|\bG-\bG_1\|&=&\|\bG(\bG^{-1}-\bG_1^{-1})\bG_1\|\leq O_p(p^{-2})\|({\bB}\bH')'\Sig_u^{-1}{\bB}\bH'-\hLam'(\hSig_{u, \widehat{K}}^{\mathcal{T}})^{-1}\hLam \|\cr
&=&O_p\left( p^{-1}\omega_T^{1-q}m_p\right).
\end{eqnarray*}
Consequently, $L_6\leq \|\Sig_u^{-1}{\bB}\bH'\|^2\|\bG-\bG_1\|=O_p\left(\omega_T^{1-q}m_p\right).$
Adding up $L_1$-$L_6$ gives $$
\|\hSig^{-1}_{\widehat{K}}-\tilde{\Sig}^{-1}\|=O_p(\omega_T^{1-q}m_p).$$
One the other hand,  using  Sherman-Morrison-Woodbury formula again implies
\begin{eqnarray*}
\|\tilde{\Sig}^{-1}-\Sig^{-1}\|&\leq& \|\Sig_u^{-1}{\bB}([(\bH'\bH)^{-1}+{\bB}'\Sig_u^{-1}{\bB}]^{-1}-[\bI_K+{\bB}'\Sig_u^{-1}{\bB}]^{-1}){\bB}'\Sig_u^{-1}\|\cr
&\leq& O(p)\|[(\bH'\bH)^{-1}+{\bB}'\Sig_u^{-1}{\bB}]^{-1}-[\bI_K+{\bB}'\Sig_u^{-1}{\bB}]^{-1}\|\cr
&=&O_p(p^{-1})\|(\bH'\bH)^{-1}-\bI_K\|=o_p(\omega_T^{1-q}m_p).
\end{eqnarray*}
\end{proof}

\textbf{Proof of Theorem \ref{thm32}: $\|\hSig^{\mathcal{T}}-\Sig\|_{\max}$}
\begin{proof}
We first bound $\|\hLam\hLam'-\bB\bB'\|_{\max}$. Repeatedly using the triangular inequality yields
\begin{eqnarray*}
\|\hLam\hLam'-\bB\bB'\|_{\max}&=&\max_{i,j\leq p}|\hb_i'\hb_j-\bb_i'\bb_j|\cr
&\leq&\max_{ij}[|(\hb_i-\bH\bb_i)'\hb_j|+|\bb_i'\bH'(\hb_j-\bH\bb_j)|+|\bb_i'(\bH'\bH-\bI_K)\bb_j|]\cr
&\leq&(\max_i\|\hb_i-\bH\bb_i\|)^2+2\max_{ij}\|\hb_i-\bH\bb_i\|\|\bH\bb_j\|+\max_i\|\bb_i\|^2\|\bH'\bH-\bI_K\|\cr
&=&O_p(\omega_T).
\end{eqnarray*}
On the other hand, let $\sigma_{u,ij}$ be the $(i,j)$ entry of $\Sig_u$. Then $\max_{ij}|\hsig_{ij}-\sigma_{u,ij}|=O_p(\omega_T)$.
$$
\max_{ij}|s_{ij}(\hsig_{ij})-\sigma_{u,ij}|\leq \max_{ij}|s_{ij}(\hsig_{ij})-\hsig_{ij}|+|\hsig_{ij}-\sigma_{u,ij}|\leq\max_{ij}\tau_{ij}+O_p(\omega_T)=O_p(\omega_T).
$$
 Hence $
\|\hSig_{u,\widehat{K}}^{\mathcal{T}}-\Sig_u\|_{\max}=O_p(\omega_T).$
The result then follows immediately.
\end{proof}


\begin{thebibliography}{99} \itemsep -0.005in


\item \textsc{Agarwal, A., Negahban, S. and Martin J. Wainwright, M. J.} (2012).  Noisy matrix decomposition via convex relaxation: Optimal rates in high dimensions. {\em Ann. Statist.}  {\bf 40}, 1171-1197.

\bibitem{a} \textsc{Ahn, S., Lee, Y.} and \textsc{Schmidt, P.} (2001). GMM estimation of linear panel data models with time-varying individual effects. \textit{J. Econometrics}. \textbf{101}, 219-255.

\bibitem{aa} \textsc{Alessi, L., Barigozzi, M.} and \textsc{Capassoc, M.} (2010).  Improved penalization for determining the number of factors in approximate factor models. \textit{Statistics and Probability Letters}, \textbf{80}, 1806-1813.

\bibitem{AW} \textsc{Amini, A. A.} and \textsc{Wainwright, M. J.} (2009). High-dimensional analysis of semidefinite
        relaxations for sparse principal components. \textit{Ann. Statist.}, {\bf 37}, 2877-2921.


\bibitem{A} \textsc{Antoniadis, A.} and \textsc{Fan, J.} (2001). Regularized wavelet approximations.  \textit{ J. Amer. Statist. Assoc.}  \textbf{96}, 939-967.

\bibitem{al} \textsc{Athreya, K.} and \textsc{Lahiri, S.} (2006) \textit{Measure theory and probability theorey.} Springer, New York.


\bibitem{Bai} \textsc{Bai, J.} (2003). Inferential theory for factor models of large dimensions. \textit{Econometrica}. \textbf{71} 135-171.

\bibitem{Bab} \textsc{Bai, J.}  and \textsc{Ng, S.}(2002). Determining the number of factors in approximate factor models. \textit{Econometrica}. \textbf{70} 191-221.

\bibitem{Bab} \textsc{Bai, J.}  and \textsc{Ng, S.}(2008). Large dimensional factor analysis. \textit{Foundations and trends in econometrics}. \textbf{3} 89-163.


\bibitem{Babd} \textsc{Bai, J.}  and \textsc{Shi, S.}(2011). Estimating high dimensional covariance matrices and its applications. \textit{Annals of Economics and Finance}. \textbf{12} 199-215.


 



\bibitem{BLa} \textsc{Bickel, P.} and \textsc{Levina, E.} (2008). Covariance regularization by thresholding. \textit{Ann. Statist.} \textbf{36} 2577-2604.


 

\bibitem{Bb} \textsc{Birnbaum, A., Johnstone, I., Nadler, B.} and \textsc{Paul, D.} (2012). Minimax bounds for sparse PCA with noisy high-dimensional data.  To appear in \textit{Ann. Statist.}

\bibitem{BN} \textsc{Boivin, J.} and \textsc{Ng, S.} (2006). Are More Data Always Better for Factor Analysis? \textit{J. Econometrics}. \textbf{132}, 169-194.


\bibitem{CC} \textsc{Cai, J., Cand\`{e}s, E.} and \textsc{Shen, Z.} (2008).  A singular value thresholding algorithm for matrix completion.  \textit{SIAM J. on Optimization}, \textbf{20}, 1956-1982.

\bibitem{Cl} \textsc{Cai, T.} and \textsc{Liu, W.} (2011). Adaptive thresholding for sparse covariance matrix estimation.\textit{ J. Amer. Statist. Assoc.} \textbf{106}, 672-684.


\bibitem{CZ} \textsc{Cai, T.} and \textsc{Zhou, H.} (2010). Optimal rates of convergence for sparse covariance matrix estimation. \textit{Manuscript}. University of Pennsylvania.


\bibitem{Cad} \textsc{Cand\`{e}s, E., Li, X., Ma, Y.} and \textsc{Wright, J.} (2011).  Robust principal component analysis? \textit{J. ACM},  \textbf{58},  3.


\bibitem{Cv} \textsc{Carvalho, C., Chang, J., Lucas, J., Nevins, J., Wang, Q.} and \textsc{West, M.} (2008). High-dimensional sparse factor modeling:  applications in gene expression genomics. \textit{ J. Amer. Statist. Assoc.} \textbf{103}, 1438-1456.

\bibitem{C} \textsc{Chamberlain, G.} and \textsc{Rothschild, M.} (1983). Arbitrage, factor structure and mean-variance analyssi in large asset markets. \textit{Econometrica}. \textbf{51} 1305-1324.



\bibitem{DGR} \textsc{Doz, C., Giannone, D.} and \textsc{Reichlin, L.} (2011). A two-step estimator for large approximate dynamic factor models based on Kalman filtering. \textit{J. Econometrics}. \textbf{164}, 188-205.


\bibitem{DBL} d'\textsc{Aspremont, A., Bach, F.} and  \textsc{El Ghaoui, L.} (2008). Optimal solutions
      for sparse principal component analysis. {\em  J. Mach. Learn. Res.},   {\bf 9}, 1269-1294.

\bibitem{DK} \textsc{Davis, C.} and \textsc{Kahan, W.} (1970). The rotation of eigenvectors by a perturbation III. {\em SIAM J. Numer. Anal.}, {\bf 7}, 1-46.

\bibitem{EF1} \textsc{Efron, B.} (2007). Correlation and large-scale simultaneous significance
        testing. {\em  J. Amer. Statist. Assoc.}  {\bf 102}, 93-103.


\bibitem{EF} \textsc{Efron, B.} (2010). Correlated z-values and the accuracy of large-scale
        statistical estimates. {\em  J. Amer. Statist. Assoc.} {\bf 105}, 1042-1055.


\bibitem{FF} \textsc{Fama, E.} and \textsc{French, K.} (1992).  The cross-section of expected stock returns. \textit{Journal of Finance}. \textbf{47},  427-465.



\bibitem{FL} \textsc{Fan, J., Fan, Y.} and \textsc{Lv, J.} (2008). High dimensional covariance matrix estimation using a factor model. {\em J. Econometrics},
    {\bf 147}, 186-197.

\bibitem{FHG} \textsc{Fan, J., Han, X.,} and \textsc{Gu, W.}(2012). Control of the false discovery rate under arbitrary covariance dependence (with discussion).
       {\em  J. Amer. Statist. Assoc.},  {\bf 107},
       1019--1048.

\bibitem{FL} \textsc{Fan, J., Liao, Y.} and \textsc{Mincheva, M.} (2011). High dimensional covariance matrix estimation in approximate factor models. \textit{Ann. Statist}. {\bf 39}, 3320-3356.


\bibitem{FLM} \textsc{Fan, J., Liao, Y.} and \textsc{Mincheva, M.} (2011). Large covariance estimation by thresholding principal orthogonal complements.    \textit{Working paper} of this article.  arxiv.org/pdf/1201.0175.pdf



\bibitem{fz} \textsc{Fan, J., Zhang, J.,} and \textsc{Yu, K.} (2012). Vast portfolio selection
    with gross-exposure constraints.
        {\em  J. Amer. Statist. Assoc.}  {\bf 107}, 592-606.







\bibitem{Ff} \textsc{Forni, M., Hallin, M., Lippi, M.} and \textsc{Reichlin, L.} (2000).  The generalized dynamic factor model: identification and estimation.  \textit{Review of Economics and Statistics}. \textbf{82} 540-554.




\bibitem{Ff} \textsc{Forni, M., Hallin, M., Lippi, M.} and \textsc{Reichlin, L.} (2004).  The generalized dynamic factor model consistency and rates.  \textit{ J. Econometrics}. \textbf{119} 231-255.



\bibitem{Fe} \textsc{Forni, M.} and \textsc{Lippi, M.} (2001). The generalized dynamic factor model:  representation theory.  \textit{Econometric Theory},  \textbf{17}, 1113-1141.



\bibitem{fad} \textsc{Fryzlewicz, P.} (2012).   High-dimensional volatility matrix estimation via wavelets and thresholding. 
 \textit{Manuscript}. London School of Economics  and Political Science



\bibitem{HL} \textsc{Hallin, M.} and \textsc{Li\v{s}ka, R.} (2007). Determining the number of factors in the general dynamic factor model. \textit{J. Amer. Statist. Assoc.} \textbf{102}, 603-617.


\bibitem{HL} \textsc{Hallin, M.} and \textsc{Li\v{s}ka, R.} (2011). Dynamic factors in the presence of blocks. \textit{J.  Econometrics}, \textbf{163}, 29-41.



 

\bibitem{HTF} \textsc{Hastie, T.J., Tibshirani, R.} and \textsc{Friedman, J.} (2009).  {\em
       The elements of Statistical Learning:  Data Mining, Inference,
       and Prediction} (2nd ed).  {\em Springer}, New York.

\bibitem{JS} \textsc{James, W.} and \textsc{Stein, C.} (1961). Estimation with quadratic loss, in \textit{Proc. Fourth Berkeley Symp. Math. Statist. Probab.} \textbf{1} 361-379. Univ. California Press. Berkeley.

\bibitem{IJ} \textsc{Johnstone, I. M.} (2001). On the distribution of the largest eigenvalue in principal components analysis. {\em Ann. Statist.},
    {\bf 29}, 295--327.
\bibitem{JL} \textsc{Johnstone, I.M.}  and  \textsc{Lu, A.Y.} (2009). On consistency and sparsity for principal components analysis in high dimensions. {\em J. Amer. Statist. Assoc.} {\bf 104}, 682-693.

\bibitem{JM} \textsc{Jung, S.} and \textsc{Marron, J.S.} (2009).  PCA consistency in high dimension, low sample size context.
    {\em Ann. Statist.}, {\bf 37}, 4104-4130.

\bibitem{K3} \textsc{Kapetanios, G.} (2010). A testing procedure for determining the number of factors in approximate factor models with large datasets. \textit{ J. Bus. Econom. Statist. },  \textbf{28},  397-409.



\bibitem{LF} \textsc{Lam, C.} and \textsc{Fan, J.} (2009). Sparsistency and rates of convergence in large covariance matrix estimation. \textit{Ann.  Statist.} \textbf{37} 4254-4278.


\bibitem{LF} \textsc{Lawley, D.} and \textsc{Maxwell, A.} (1971).  \textit{Factor analysis as a statistical method.} Second ed. London, Butterworths.




\bibitem{LS} \textsc{Leek, J.} and \textsc{Storey, J.} (2008).  A general framework for multiple testing dependence.  \textit{PNAS}. \textbf{105}. 18718-18723.

\bibitem{LL} \textsc{Lin, Z., Ganesh, A., Wright, J., Wu, L., Chen, M.} and \textsc{Ma, Y.} (2009). Fast convex optimization algorithms for exact recovery of a corrupted low-rank matrix.  \textit{Manuscript}.  Microsoft Research Asia.


\bibitem{LUO} \textsc{Luo, X.} (2011).  High dimensional low rank and sparse covariance
            matrix estimation via convex minimization.  {\em Manuscript}.


\bibitem{MA} \textsc{Ma, Z.} (2011).  Sparse principal components analysis and iterative thresholding.
      {\em Manuscript.}

\bibitem{MEN} \textsc{Meinshausen, N.} and \textsc{B\"uhlmann, P.} (2006).
   High dimensional graphs and variable selection with the Lasso.
   {\em  Ann. Statist.}, {\em 34}, 1436--1462.


 
\bibitem{O} \textsc{Onatski, A.} (2010). Determining the number of factors from empirical distribution of eigenvalues. \textit{The Review of Economics and Statistics}. \textbf{92}, 1004-1016.



\bibitem{fs} \textsc{Pati, D., Bhattacharya, A., Pillai, N.} and \textsc{Dunson, D.}  (2012) Posterior contraction in sparse Bayesian factor models for massive covariance matrices. \textit{Manuscript}, Duke University

\bibitem{P} \textsc{Paul, D.}  (2007). Asymptotics of sample eigenstructure for a large dimensional  spiked covariance model. \textit{ Statist. Sinica},  \textbf{17}.  1617-1642.


\bibitem{P}\textsc{Pesaran, M.H.}  (2006). Estimation and inference in large heterogeneous panels with a multifactor error structure.   \textit{Econometrica.} \textbf{74}, 967-1012.

\bibitem{Ph} \textsc{Phan, Q.}  (2012).  On the sparsity assumption of the idiosyncratic errors covariance matrix-Support from the FTSE 100 stock returns. \textit{Manuscript}. University of Warwick.


\bibitem{ROS} \textsc{Ross, S.A.} (1976).  The arbitrage theory of capital asset
        pricing.  {\em  Journal of Economic Theory},
        {\bf 13}, 341-360.

\bibitem{R}\textsc{Rothman, A., Levina, E.} and \textsc{Zhu, J.} (2009). Generalized thresholding of large covariance matrices. \textit{J. Amer. Statist. Assoc.} \textbf{104} 177-186.




\bibitem{SEN} \textsc{Sentana, E.} (2009). The econometrics of mean-variance efficiency tests: a survey {\em Econometrics Jour.}, {\bf 12},  65-101.


\bibitem{SH2}\textsc{Shen, H.} and \textsc{Huang, J.} (2008).  Sparse principal component analysis via regularized low rank matrix approximation. \textit{ J. Multivariate Analysis} \textbf{99}, 1015-1034.

 
\bibitem{SHA} \textsc{Sharpe, W.} (1964).  Capital asset prices:  A theory of
        market equilibrium under conditons of risks.  {\em Journal
        of Finance}, {\bf 19}, 425-442.


 

\bibitem{SW2}\textsc{Stock, J.} and \textsc{Watson, M.} (1998). Diffusion Indexes,  NBER Working Paper 6702.


\bibitem{SW2}\textsc{Stock, J.} and \textsc{Watson, M.} (2002).  Forecasting using principal components from a large number of predictors. \textit{ J. Amer. Statist. Assoc.} \textbf{97}, 1167-1179.


 

\bibitem{WTT} \textsc{Witten, D.M., Tibshirani, R.} and \textsc{Hastie, T.} (2009). A penalized matrix decomposition, with applications to sparse principal components and canonical correlation analysis. {\em Biostatistics}, {\bf 10}, 515-534.



\bibitem{dW} \textsc{Wright, J., Peng, Y., M, Y., Ganesh, A.} and \textsc{Rao, S.} (2009). Robust principal component analysis: exact recovery of corrupted low-rank matrices by convex optimization. \textit{Manuscript}.  Microsoft Research Asia



\bibitem{HGC} \textsc{Xiong, H., Goulding, E.H., Carlson,  E.J., Tecott, L.H., McCulloch, C.E.} and  \textsc{Sen, S.} (2011).  A flexible estimating equations approach for mapping function-valued traits. \textit{Genetics}, {\bf 189}, 305--316.


\bibitem{YFW09} \textsc{Yap, J.S., Fan, J.,} and \textsc{Wu, R.} (2009).  Nonparametric modeling
    of longitudinal covariance structure in functional mapping of
    quantitative trait loci. {\em Biometrics}, {\bf 65}, 1068-1077.

\bibitem{z} \textsc{Zhang, Y.} and \textsc{El Ghaui, L.} (2011)  Large-scale sparse principal component analysis with application to text data.  {\em NIPS}.
\end{thebibliography}
\end{document}